\documentclass{elsart}

\usepackage{amssymb}
\usepackage{color}
\usepackage{graphicx}\usepackage{subfigure}
\usepackage{amsmath}
\usepackage{amsfonts}
\usepackage{listings}
\usepackage{grffile}
\usepackage{float}
\usepackage{caption}
\usepackage{epstopdf}
\usepackage[title]{appendix}
\allowdisplaybreaks
\usepackage{psfrag}

\usepackage{paralist}
\usepackage[colorlinks=true]{hyperref}

\usepackage{multirow,booktabs}
\newtheorem{theorem}{Theorem}[section]
\numberwithin{equation}{section}
\numberwithin{figure}{section}
\numberwithin{table}{section}

\renewcommand{\vec}[1]{\mbox{\boldmath \small $#1$}}

\renewcommand{\qed}{\hfill \nobreak \ifvmode \relax \else
      \ifdim\lastskip<1.5em \hskip-\lastskip
      \hskip1.5em plus0em minus0.5em \fi \nobreak
      \vrule height0.75em width0.5em depth0.25em\fi}

\newtheorem{example}{Example}[section]

\newtheorem{remark}{Remark}[section]
\numberwithin{equation}{section}
\numberwithin{figure}{section}
\numberwithin{table}{section}

\newenvironment{proof}[1][Proof]{\begin{trivlist}
\item[\hskip \labelsep {\bfseries #1}]}{\end{trivlist}}
\renewcommand{\qed}{\hfill \nobreak \ifvmode \relax \else
      \ifdim\lastskip<1.5em \hskip-\lastskip
      \hskip1.5em plus0em minus0.5em \fi \nobreak
      \vrule height0.75em width0.5em depth0.25em\fi}
\begin{document}

\begin{frontmatter}
\title{Second-order accurate genuine BGK schemes for {the} ultra-relativistic flow simulations}

\author{Yaping Chen},
\ead{cyaping0918@126.com}
\author{Yangyu Kuang}
\ead{kyy@pku.edu.cn}
\address{HEDPS, CAPT \& LMAM, School of Mathematical Sciences, Peking University,
Beijing 100871, P.R. China}
\author[label2]{Huazhong Tang}
\thanks[label2]{Corresponding author. Tel:~+86-10-62757018;
Fax:~+86-10-62751801.}
\ead{hztang@math.pku.edu.cn}
\address{HEDPS, CAPT \& LMAM, School of Mathematical Sciences, Peking University,
Beijing 100871, P.R. China; School of Mathematics and Computational Science,
 Xiangtan University, Hunan Province, Xiangtan 411105, P.R. China}
 \date{\today{}}

\maketitle

\begin{abstract}
This paper presents second-order accurate genuine BGK (Bhatnagar-Gross-Krook)
schemes in the framework of finite volume method for the ultra-relativistic flows.
Different from the existing kinetic flux-vector
splitting (KFVS) or BGK-type schemes for  the ultra-relativistic Euler equations,
the  present genuine  BGK schemes are derived from the analytical solution of the
Anderson-Witting model, which is given for the first time and includes the ``genuine'' particle collisions in the gas transport process.
The BGK schemes for the ultra-relativistic viscous flows are also developed and two examples of ultra-relativistic viscous {flow} are designed.
Several 1D and 2D numerical experiments are conducted to demonstrate that the proposed BGK schemes
not only are accurate and stable in simulating
 ultra-relativistic inviscid and viscous flows, but also have higher resolution at the contact discontinuity  than the KFVS or BGK-type schemes.
\end{abstract}

\begin{keyword}
BGK scheme, Anderson-Witting model, ultra-relativistic Euler equations, ultra-relativistic Navier-Stokes equations
\end{keyword}
\end{frontmatter}


\section{Introduction}
\label{sec:intro}
Relativistic hydrodynamics (RHD) arise in astrophysics, nuclear physics, plasma physics and other fields.
In many radiation hydrodynamics problems of astrophysical interest,
the fluid moves at extremely high velocities near the speed of light,
and relativistic effects become important.
Examples of such flows are supernova explosions,
the cosmic expansion, and solar flares.

The relativistic hydrodynamical equations are highly nonlinear, making the analytic treatment
of practical problems extremely difficult. The numerical simulation is
the primary and powerful way to study and understand the relativistic hydrodynamics.
This work will mainly focus on the numerical methods for the special RHDs,
where  there is no strong gravitational field involved.
The pioneering numerical work may date back to
the   finite difference code via artificial viscosity  for the spherically symmetric general RHD equations in the Lagrangian coordinate \cite{May1966,May1967} and  the finite difference method with the artificial viscosity technique
 for the multi-dimensional RHD equations in the Eulerian coordinate \cite{R.Wilson1972}.
Since 1990s, the numerical study of the RHDs began to attract considerable attention, and various modern shock-capturing methods with an exact
or approximate Riemann solver have been developed for the RHD equations.
Some examples are the local characteristic approach \cite{Mart1991Numerical},
the two-shock approximation solvers \cite{balsara1994riemann,Dai1997},
 the Roe solver \cite{F.Eulderink1995}, the flux corrected transport method \cite{duncan1994}, the flux-splitting method based on the spectral decomposition \cite{Donat1998}, the piecewise parabolic method \cite{Marti1996,Mignone2005}, the HLL (Harten-Lax-van Leer) method \cite{schneider1993new}, the HLLC (Harten-Lax-van Leer-Contact) method \cite{mignone2005hllc} and the Steger-Warming flux vector splitting
 method \cite{Zhao2014Steger}.
 The analytical solution of the  Riemann problem in relativistic hydrodynamics was studied in
 \cite{mart1994}.
 Some other higher-order accurate methods have also been well studied in the literature, e.g.
 the ENO (essentially non-oscillatory) and weighted ENO methods \cite{dolezal1995relativistic,del2002efficient,Tchekhovskoy2007},
the discontinuous Galerkin (DG) method \cite{RezzollaDG2011}, the adaptive moving mesh methods \cite{he2012adaptive1,he2012adaptive2},
the Runge-Kutta DG {methods} with WENO limiter \cite{zhao2013runge,ZhaoTang-CiCP2017,ZhaoTang-JCP2017},
the direct Eulerian GRP {schemes} \cite{yang2011direct,yang2012direct,wu2014third},
and the local evolution Galerkin method \cite{wu2014finite}.
Recently some physical-constraints-preserving (PCP) schemes were developed for the special RHD
 equations. They are the high-order accurate PCP finite difference weighted essentially
 non-oscillatory (WENO) schemes and discontinuous Galerkin (DG) methods proposed
 in \cite{wu2015high,wu2016physical,qin2016bound}.
The readers are also referred to the early review articles \cite{marti2003review,Font2008}
 as well as references therein.


The gas-kinetic schemes present a gas evolution process from a kinetic scale to a
hydrodynamic scale, where both inviscid and viscous fluxes are recovered from moments
of a single time-dependent gas distribution function \cite{PanXu2015}.
The development of gas-kinetic schemes, such as the kinetic flux
vector splitting  (KFVS) and Bhatnagar-Gross-Krook (BGK) schemes,
has attracted much attention and significant progress has been
made in the non-relativistic hydrodynamics.
  They utilize the well-known connection that the macroscopic governing equations are the moments of the Boltzmann equation whenever the distribution function is at equilibrium.
The KFVS schemes are constructed by applying upwind technique
directly to the collisionless Boltzmann equation, see e.g. \cite{Pullin1980,Mandal1994,Chou1997,HongLui2001,Reitz1981,Perthame1992,TangMagnet2000,TangRadia2000,Tang1999High}.
Due to the lack of collision in the numerical flux calculations, the KFVS schemes
smear the solutions, especially the contact discontinuity.
To overcome this problem, the BGK schemes are constructed
by taking into account the particle collisions in the whole gas evolution process within a time step,
see e.g. \cite{LiXu2010,Xu2001,LiuTang2014}.
%
Moreover, due to their specific derivation,
 they are also able to present the accurate Navier-Stokes solution in the smooth flow regime and have favorable shock capturing capability in the shock region.
The  kinetic  beam  scheme was first proposed for the relativistic gas dynamics in \cite{yang1997kinetic}. After that,
the kinetic schemes for  the ultra-relativistic Euler equations  were developed in  \cite{Kunik2003ultra,kunik2003second,Kunik2004}.
The BGK-type  schemes  \cite{Xu1999,TangMagnet2000} were extended to the ultra-relativistic Euler equations  in \cite{kunik2004bgktype,QamarRMHD2005} in order to reduce the numerical dissipation.
Those kinetic schemes resulted directly from the moments of the relativistic J$\ddot{\text{u}}$ttner equilibrium distribution
without including the ``genuine" particle collisions in the gas transport process.

This paper will develop  second-order genuine BGK schemes
for the ultra-relativistic inviscid and viscous flow simulations.
%
%
It is organized as follows.
Section \ref{sec:basic} introduces the special relativistic Boltzmann
equation and discusses  how to recover some macroscopic quantities from the kinetic theory.
Section \ref{sec:GovernEqns} presents the ultra-relativistic hydrodynamical
equations through  the Chapman-Enskog expansion. Section \ref{sec:scheme}  develops
second-order accurate genuine BGK schemes for the 1D and 2D ultra-relativistic Euler equations and 2D ultra-relativistic
 Navier-Stokes equations. Section \ref{sec:test} gives several   numerical experiments to demonstrate accuracy, robustness and effectiveness of
 the proposed schemes in simulating inviscid and viscous ultra-relativistic fluid flows.
 Section \ref{sec:conclusion} concludes the paper.

\section{Preliminaries and notations}
\label{sec:basic}
In the special relativistic kinetic theory of gases \cite{rbebook}, a microscopic gas particle is characterized by the four-dimensional space-time coordinates $(x^{\alpha}) = (x^0,\vec{x})$ and four-momentum vectors $(p^{\alpha}) = (p^0,\vec{p})$, where $x^0 = ct$, $c$ denotes the speed of light in vacuum,   $t$ and $\vec{x}$ are the time and 3D spatial coordinates, respectively, and the Greek index $\alpha$  runs from $0, 1, 2, 3$.
Besides the contravariant notation (e.g. $p^{\alpha}$), the covariant notation such as $p_{\alpha}$ will also be used in the following,
while both notations $p^\alpha$ and $p_{\alpha}$ are related by
\[p_\alpha = g_{\alpha\beta}p^{\beta},\quad p^{\alpha} = g^{\alpha\beta}p_{\beta},\]
where the Einstein summation convention over repeated indices has been used, $(g^{\alpha\beta})$ is the Minkowski space-time metric tensor and  chosen as
$(g^{\alpha\beta}) = \text{diag}\{1, -1, -1, -1\}$, while $(g_{\alpha\beta})$ denotes the
inverse of $(g^{\alpha\beta})$.

For a free relativistic particle,  the relativistic energy-momentum
relation (aka ``on-shell'' or ``mass-shell'' condition)
$E^2-|\vec p|^2 c^2=m^2 c^4$ holds,  where $m$ denotes the mass of each structure-less particle which is assumed to be the same for all particles.  The ``mass-shell'' condition {can} be rewritten as
$p^{\alpha}p_{\alpha}=m^2c^2$ if putting $p^0= c^{-1}E=\sqrt{|\vec{p}|^2+m^2c^2}$,
which becomes
 $p^0=|\vec{p}|$ in the ultra-relativistic limit, i.e. $m\to0$.

Similar to the non-relativistic case,
the relativistic Boltzmann equation describes the evolution of one-particle distribution function $f(\vec{x},t,\vec{p})$  in the phase space spanned by the space-time coordinates $x^\alpha$ and momentum $p^\alpha$ of particles. It reads
\begin{equation}
  p^{\alpha}\frac{\partial f}{\partial x^{\alpha}} = Q(f,f),
\end{equation}
where $Q(f,f)$ denotes the collision term and depends on the product of distribution functions of two particles at collision. In the literature, there exist several simple collision models.
The Anderson-Witting model \cite{AW}
\begin{equation}\label{AWM}
  p^{\alpha}\frac{\partial f}{\partial x^{\alpha}} = -\frac{U_\alpha p^{\alpha}}{\tau c^2}(f-g),
\end{equation}
is similar to the BGK model in the non-relativistic kinetic theory and
will be considered in this paper, where  $\tau$ is the relaxation time,
in {the} Landau-Lifshitz frame, the hydrodynamic four-velocities  $U_{\alpha}$
are defined  by
\begin{equation}\label{EQ:Landau-Lifshitz frame}
  U_{\beta}T^{\alpha\beta} = \varepsilon g^{\alpha\beta}U_{\alpha},
\end{equation}
which implies that $(\varepsilon, U_{\alpha})$ is a generalized characteristic pair of $(T^{\alpha\beta},g^{\alpha\beta})$, $\varepsilon$ and $ T^{\alpha\beta}$ are
the energy density and  energy-momentum tensor, respectively,
  and $g=g(\vec{x},t,\vec{p})$ denotes the distribution function at the local thermodynamic
equilibrium, the so-called J$\ddot{\text{u}}$ttner
equilibrium (or relativistic Maxwellian) distribution. In the ultra-relativistic case,  it becomes \cite{Kunik2003ultra}
\begin{equation}\label{juttner}
  g={\frac{ n c^3}{8\pi k^3T^3}}\exp\left(-\frac{U_{\alpha}p^{\alpha}}{{kT}}\right)={\frac{ n c^3}{8\pi k^3T^3}}\exp\left(-\frac{|\vec{p}|}{{kT}}\left(U_0-\sum_{i=1}^3U_i\frac{p^i}{|\vec{p}|}\right)\right),
\end{equation}
where $n$ and $T$ denote { the} number density and thermodynamic temperature, respectively,
 and $k$ is the Boltzmann's constant.
The Anderson-Witting model \eqref{AWM} can tend to the BGK model in the non-relativistic limit
and the collision term $-\frac{U_\alpha p^{\alpha}}{\tau c^2}(f-g)$
satisfies the following identities
\begin{equation}\label{convc}
  \int_{\mathbb{R}^3}\frac{U_\alpha p^{\alpha}}{\tau c^2}(f-g)\Psi \frac{d^3\vec{p}}{p^0} = 0,
  \ \
  \vec{\Psi} = (1, p^{i}, p^0)^T,
\end{equation}
which imply the conservation of  particle number, momentum and energy
\begin{equation}\label{conv}
  \partial_{\alpha}N^{\alpha} = 0,\quad \partial_{\beta}T^{\alpha\beta} = 0,
\end{equation}
where the particle four-flow $N^{\alpha}$ and the energy-momentum tensor $ T^{\alpha\beta}$ are
related to the distribution $f$ by
  \begin{align}
    N^{\alpha} = c\int_{\mathbb{R}^3}p^{\alpha}f\frac{d^3\vec{p}}{p^0},
  \ \
  \label{intTT}
    T^{\alpha\beta}=c\int_{\mathbb{R}^3}p^{\alpha}p^{\beta}f\frac{d^3\vec{p}}{p^0}.
  \end{align}
In the Landau-Lifshitz
decomposition, both $N^{\alpha}$ and $T^{\alpha\beta}$ are rewritten  as follows
\begin{align}
  \label{NN} N^{\alpha} &= n U^{\alpha} + n^{\alpha},\\
  \label{TT} T^{\alpha\beta} &= c^{-2}\varepsilon U^{\alpha}U^{\beta} - \Delta^{\alpha\beta}(p+\varPi) + \pi^{\alpha\beta},
\end{align}
where   $\Delta^{\alpha\beta}$ is defined by
\begin{equation}
  \Delta^{\alpha\beta} = g^{\alpha\beta} - \frac{1}{c^2}U^{\alpha}U^{\beta},
\end{equation}
 satisfying $\Delta^{\alpha\beta} U_{\beta} = 0$,
the number density $n$,   particle-diffusion current $n^\alpha$,
energy density $\varepsilon$, and shear-stress tensor {$\pi^{\alpha\beta}$}
 can be calculated by
  \begin{align}
       n=& \frac{1}{c^2} U_{\alpha}N^{\alpha} = \frac{1}{c} \int_{\mathbb{R}^3}Ef\frac{d^3\vec{p}}{p^0},
       \\
 \label{n}
    n^{\alpha} =& \Delta^{\alpha}_{\beta}N^{\beta} = c\int_{\mathbb{R}^3}p^{<\alpha>}f\frac{d^3\vec{p}}{p^0},
\\
    \varepsilon =& \frac{1}{c^2}U_{\alpha}U_{\beta}T^{\alpha\beta} = \frac{1}{c}\int_{\mathbb{R}^3}E^2f\frac{d^3\vec{p}}{p^0},
\\
\label{pi}
    \pi^{\alpha\beta} =& \Delta^{\alpha\beta}_{\mu\nu}T^{\mu\nu} = c\int_{\mathbb{R}^3}p^{<\alpha\beta>}f\frac{d^3\vec{p}}{p^0},
    \end{align}
and the sum of thermodynamic pressure $p$ and bulk viscous pressure $\varPi$ is
  \begin{equation}\label{varPi}
    p + \varPi = -\frac{1}{3}\Delta_{\alpha\beta}T^{\alpha\beta} = \frac{1}{3c}\int_{\mathbb{R}^3}(E^2-m^2c^4)f\frac{d^3\vec{p}}{p^0}.
  \end{equation}
Here
$E=U_{\alpha}p^{\alpha}$, $p^{<\alpha>}=\Delta^{{\alpha}}_{\gamma}p^{\gamma}$, $p^{<\alpha\beta>}=\Delta^{\alpha\beta}_{\gamma\delta}p^{\gamma}p^{{\delta}}$, and
  \begin{equation}
      \Delta^{\alpha\beta}_{\mu\nu}=\frac{1}{2}(\Delta^{\alpha}_{\mu}\Delta^{\beta}_{\nu} + \Delta^{\beta}_{\mu}\Delta^{\alpha}_{\nu}-\frac{2}{3}\Delta_{\mu\nu}\Delta^{\alpha\beta}).
  \end{equation}

\begin{remark}\label{remark2.1}
The quantities $n^{\alpha}, \varPi$, and $\pi^{\alpha\beta}$ become zero
at the local thermodynamic equilibrium $f=g$.
\end{remark}

The following gives a general recovery procedure of the admissible primitive variables $n$, $\vec u$, and $T$ from the nonnegative distribution $f(\vec x, t, \vec p)$, where $\vec u$ is
the macroscopic velocity in the $(x^i)$ space. 
Such recovery procedure will be useful in  our BGK scheme.


\begin{theorem}\label{thm:NT}
  For any nonnegative  distribution $f(\vec{x},t,\vec{p})$ which is not always be zero, the  number density $n$, velocity $\vec{u}$ and temperature $T$  can be uniquely obtained as follows:
  \begin{enumerate}
  \item $T^{\alpha\beta}$ is positive definite and $(T^{\alpha\beta},g^{\alpha\beta})$ has only one positive generalized eigenvalue, i.e. the energy density $\varepsilon$,  and  $U_{\alpha}$ is corresponding generalized eigenvector  satisfying $U_0=\sqrt{U^2_1+U^2_2+U^2_3+c^2}$.
Thus, the macroscopic velocity $\vec{u}$ can be calculated by
$\vec{u}=-c(U^{-1}_0U_1,U^{-1}_0U_2,U^{-1}_0U_3)^{T}$, satisfying  $|\vec{u}|<c$ and
  \begin{equation}
   (U_{\alpha}) = (\gamma c,-\gamma\vec{u}), \ \ (U^{\alpha}) = (\gamma c,\gamma\vec{u}),
  \end{equation}
  where $\gamma=(1-c^{-2}|\vec{u}|^2)^{-\frac{1}{2}}$ denotes the Lorentz factor.
  \item
  The number density $ n $ is calculated by
  \begin{equation}
     n  = c^{-2} U_{\alpha}N^{\alpha}>0.
  \end{equation}
  \item The temperature $T$ solves the nonlinear algebraic equation
  \begin{equation}
    \varepsilon =  n m c^2(G(\zeta)-\zeta^{-1}),
  \end{equation}
   where $\zeta = \frac{ m c^2}{kT}$, $G(\zeta)=\frac{K_3(\zeta)}{K_2(\zeta)}$, and
    $K_{{\nu}}(\zeta)$  is  modified Bessel function of the second kind, defined by
\[
K_{{\nu}}(\zeta):=\int_{0}^{\infty}\cosh({\nu}\vartheta)\exp(-\zeta\cosh\vartheta)d\vartheta, \ \ {\nu}\geq 0.
\]
  In the ultra-relativistic case, $K_2(\zeta)$ and $K_3(\zeta)$ reduce to $\frac{2}{\zeta^2}$ and
  $\frac{8}{\zeta^3}$, respectively, so that one has $G(\zeta)=\frac{4}{\zeta}$, and then
  \begin{equation}\label{eT}
    \varepsilon = 3k n  T.
  \end{equation}
\end{enumerate}
\end{theorem}
\begin{proof}
  \begin{enumerate}
  \item Since the nonnegative distribution $f(\vec{x},t,\vec{p})$  is not identically zero,  using the relation \eqref{intTT} gives
\begin{align}\nonumber
  \vec{X}^{T}T^{\alpha\beta}\vec{X} &= c\vec{X}^{T}\int_{\mathbb{R}^3}p^{\alpha}p^{\beta}f\frac{d^3\vec{p}}{p^0}\vec{X}
       =c\int_{\mathbb{R}^3}x_{\alpha}p^{\alpha}p^{\beta}x_{\beta}f\frac{d^3\vec{p}}{p^0}
       \\
       &= c\int_{\mathbb{R}^3}(x_{\alpha}p^{\alpha})^2f\frac{d^3\vec{p}}{p^0}>0,
\end{align}
for any nonzero vector $\vec{X}=(x_0,x_1,x_2,x_3)^T\in\mathbb{R}^4$. Thus,
the matrix $T^{\alpha\beta}$ is positive definite.

  Thanks to $g^{\alpha\beta} = \text{diag}\{1,-1,-1,-1\}$ and \eqref{EQ:Landau-Lifshitz frame}, the matrix-pair $(T^{\alpha\beta},g^{\alpha\beta})$ has an unique positive generalized eigenvalue
$\varepsilon$,  satisfying
\begin{equation}
  0<U_{\alpha}T^{\alpha\beta}U_{\beta} = \varepsilon U_{\alpha}g^{\alpha\beta}U_{\beta},
\end{equation}
which implies $U^2_0>U^2_1+U^2_2+U^2_3$. Thus, one can obtain $U_0=\sqrt{U^2_1+U^2_2+U^2_3+c^2}$  via multiplying $(U_{\alpha})$ by a scaling constant
$c(U^2_0-U^2_1-U^2_2-U^2_3)^{-1/2}$. As a result, the macroscopic velocity $\vec{u}$
can be calculated by
$\vec{u}=-c(U^{-1}_0U_1,U^{-1}_0U_2,U^{-1}_0U_3)^{T}$, satisfying
\begin{equation}
  |\vec{u}| = cU^{-1}_0\sqrt{U^2_1+U^2_2+U^2_3}<c.
\end{equation}

\item For ${U}_i\in\mathbb{R}$ and ${p}^i\in\mathbb{R}$, $i=1,2,3$,
 using the Cauchy-Schwarz inequality gives
\begin{align}
  \vec{U}\cdot\vec{p} &\leqslant  |\vec{U}||\vec{p}|
                            < \sqrt{U^2_1+U^2_2+U^2_3+c^2}\cdot p^0
                            =U_0p^0,
\end{align}
which implies $E=U_{\alpha}p^{\alpha}>0$. Thus  one has
\begin{equation}
     n  = \frac{1}{c^2} U_{\alpha}N^{\alpha} = \frac1c \int_{\mathbb{R}^3}Ef\frac{d^3\vec{p}}{p^0}>0.
\end{equation}

\item It is obvious that the positive temperature $T$ can be obtained from  \eqref{eT}.
\end{enumerate}\qed
\end{proof}

\section{Ultra-relativistic hydrodynamic equations}
\label{sec:GovernEqns}
This section gives  the ultra-relativistic hydrodynamic equations, which can be derived from
the Anderson-Witting model by using the Chapman-Enskog expansion.
For the sake of convenience,  units in which the speed of light and the Boltzmann's constant
are equal to one will be used here and hereafter.

\subsection{Euler equations}
In the ultra-relativistic limit, the macroscopic variables $ n, \varepsilon, p$
are related to $g$  by
\begin{align}
  n &= \int_{\mathbb{R}^3}Eg\frac{d^3\vec{p}}{|\vec{p}|},\\
  \varepsilon &= \int_{\mathbb{R}^3}E^2g\frac{d^3\vec{p}}{|\vec{p}|}=3 n  T,\\
  p &=\frac13\int_{\mathbb{R}^3}E^2g\frac{d^3\vec{p}}{|\vec{p}|}= \frac{1}{3}\varepsilon.
\end{align}
If taking the zero order   Chapman-Enskog expansion  $f=g$ and using the conclusion in
Remark \ref{remark2.1},  the  ultra-relativistic Euler equations are derived as follows
\begin{equation}
  \frac{\partial \vec{W}}{\partial t} + \sum^{3}_{k=1}\frac{\partial\vec F^{k}(\vec{W})}{\partial x^k} = 0,
\end{equation}
where
\begin{equation}
  \vec{W} = \left(N^0, T^{0i},   T^{00}
                \right)^T
             = \left(
                   n  U^0,
                   n  h U^0U^i,
                   n  h U^0U^0 - p
                  \right)^T,
\end{equation}
and
\begin{equation}
\vec{F^k(W)} = \left( N^k,
                             T^{ki},
                             T^{k0}
                \right)^T
           = \left(
                 n  U^k,
                 n  h U^kU^i + p\delta^{ik},
                 n  h U^kU^0
               \right)^T.
\end{equation}
Here $i=1,2,3$ and $h=4T$ denotes the specific enthalpy.
For the given conservative vector  $\vec{W}$, one can get
the primitive variables $ n , U^k$ and $p$
by \cite{kunik2004bgktype}
\begin{align}\label{eq:con2pri}
\begin{aligned}
  p &= \frac{1}{3}\left(-T^{00}+\sqrt{4(T^{00})^2-3\sum^3_{i=1}(T^{0i})^2}\right),\\
  U^i &= \frac{T^{0i}}{\sqrt{4p(p+T^{00})}},\ \
   n   = \frac{N^0}{\sqrt{1+\sum_{i=1}^3(U^i)^2}}, \quad i=1,2,3.
\end{aligned}\end{align}

\subsection{Navier-Stokes equations}
If taking the first order Chapman-Enskog expansion
\begin{equation}\label{ce}
  f=g(1-\frac{\tau}{U_{\alpha}p^{\alpha}}\varphi),
\end{equation}
with
\begin{equation}\label{dev}
  \varphi = -\frac{p_{\alpha}p_{\beta}}{T}\nabla^{<\alpha}U^{\beta>} + \frac{p_{\alpha}}{T^2}(U_{\beta}p^{\beta}-h)(\nabla^{\alpha}T-\frac{T}{ n  h}\nabla^{\alpha}p),
\end{equation}
where 
$\nabla^{\alpha} = \Delta^{\alpha\beta}\partial_{\beta}$ and  $\nabla^{<\alpha}U^{\beta>} = \Delta^{\alpha\beta}_{\gamma\delta}\nabla^{\gamma}U^{\delta}$,
then   \eqref{n}, \eqref{pi} and \eqref{varPi} give
\begin{align}
  n^{\alpha} = -\frac{\lambda}{h}(\nabla^{\alpha}T-\frac{T}{ n  h}\nabla^{\alpha}p), \ \
  \pi^{\alpha\beta} = 2\mu\nabla^{<\alpha}U^{\beta>},\ \
  \varPi = 0 ,
\end{align}
where $\lambda=\frac{4}{3T}p\tau$ and $\mu=\frac{4}{5}p\tau$.
Based on those, the ultra-relativistic Navier-Stokes equations see \cite{rbebook} can be obtained
as follows
\begin{equation}\label{eq-NS01}
  \frac{\partial \vec{W}}{\partial t} + \sum^{3}_{k=1}\frac{\partial\vec F^{k}(\vec{W})}{\partial x^k} = 0,
\end{equation}
where
\begin{equation}
  \vec{W} = \begin{pmatrix} N^0\\
                               T^{0i}\\
                               T^{00}
                \end{pmatrix}
             = \begin{pmatrix}
                   n  U^0 - \frac{\lambda}{h}(\nabla^0 T-\frac{T}{ n  h}\nabla^0p)\\
                   n  h U^0U^i + 2\mu\nabla^{<0}U^{i>}\\
                   n  h U^0U^0 - p + 2\mu\nabla^{<0}U^{0>}
                \end{pmatrix},
\end{equation}
and
\begin{equation}\label{eq-NS03}
\vec{F^k(W)} = \begin{pmatrix} N^k\\
                             T^{ki}\\
                             T^{k0}
              \end{pmatrix}
           = \begin{pmatrix}
                 n  U^k - \frac{\lambda}{h}(\nabla^k T-\frac{T}{ n  h}\nabla^kp)\\
                 n  h U^kU^i + p\delta^{ik} + 2\mu\nabla^{<k}U^{i>}\\
                 n  h U^kU^0 + 2\mu\nabla^{<k}U^{0>}
              \end{pmatrix}.
\end{equation}
It shows that one cannot recover
the values of primitive variables $n, {\vec{u}}$ and ${T}$ only from the given
 conservative vector  $\vec{W}$. In practice,
the values of  $n, {\vec{u}}$ and ${T}$ have to be recovered
from the given  $\vec{W}$  and $\vec{F^k(W)}$ or $N^\alpha$ and $T^{\alpha\beta}$
by  using Theorem \ref{thm:NT}.

\section{Numerical schemes}\label{sec:scheme}
This section develops second-order accurate genuine BGK schemes for the 1D and 2D ultra-relativistic Euler and Navier-Stokes equations.
{The BGK} schemes are derived from the analytical solution of the Anderson-Witting model \eqref{AWM}, which is given
for the first time and includes the ``genuine'' particle collisions in the gas transport
process.

\subsection{1D  Euler equations}\label{sec:scheme-euler1d}
Consider the 1D ultra-relativistic Euler equations with $\vec{u}=(u,0,0)^{T}$ as
\begin{equation}\label{1DEuler}
  \frac{\partial \vec{W}}{\partial t} + \frac{\partial \vec{F}(\vec{W})}{\partial x} = 0,
\end{equation}
where
\begin{equation}
  \vec{W} = \left(  n  U^0,
                   n  h U^0U^1,
                   n  h U^0U^0 - p
                \right)^T,\
\vec{F(W)} = \left(
                   n  U^1,
                   n  h U^1U^1 + p,
                   n  h U^0U^1\right)^T.
\end{equation}
It is strictly hyperbolic because there are three real and distinct eigenvalues of the Jacobian matrix $A(\vec{W})=\partial\vec{F}/\partial \vec{W}$    \cite{wu2015high}
\begin{equation}
  \lambda_1=\frac{u(1-c_s^2)-c_s(1-u^2)}{1-u^2c_s^2},\quad \lambda_2=u,\quad\lambda_3=\frac{u(1-c_s^2)+c_s(1-u^2)}{1-u^2c_s^2},
\end{equation}
where $c_s=1/\sqrt{3}$ is the speed of sound. 

Divide the spatial domain into a uniform mesh with the step size $\Delta x$ and the $j$th cell  $I_j = (x_{j-\frac{1}{2}}, x_{j+\frac{1}{2}})$,  where $x_{j+\frac{1}{2}} = \frac{1}{2}(x_{j}+x_{j+1})$ and $x_j = j\Delta x$, $j\in\mathbb{Z}$. The time interval $[0,T]$
is also divided into a (non-uniform) mesh $\{t_{n+1}=t_n+\Delta t_n, t_0=0, n\geqslant0\}$, where  the step size $\Delta t_n$ is determined by
\begin{equation}\label{TimeStep1D}
  \Delta t_n = \frac{C\Delta x}{\max\limits_j \bar{\varrho}_j},
\end{equation}
 the constant $C$ denotes the CFL number, and $\bar{\varrho}_j$ denotes a suitable approximation of the spectral radius of $A(\vec{W})$ within the   cell $I_j$.
For the given approximate cell-average values $\{\vec{\bar{W}}^n_j\}$, i.e.
\begin{equation*}
  \vec{\bar{W}}^{n}_j \approx \frac{1}{\Delta x}\int_{I_j}\vec{W}(x,t_n)dx,
\end{equation*}
 reconstruct a  piecewise linear function  as follows
\begin{equation}\label{EQ:initial-reconstruction}
  \vec{W}_h(x,t_n)=\sum\vec{W}^n_j(x)\chi_j(x), \ \  \vec{W}^n_j(x):=  {\vec{\bar{W}}}^n_j + \vec{W}^{n,x}_j(x-x_j),
\end{equation}
where $\vec{W}^{n,x}_j$ is the approximate slope in the   cell $I_j$ obtained by using some slope limiter and $\chi_j(x)$ denotes the characteristic function of $I_j$.

In the 1D case,
 the  Anderson-Witting model \eqref{AWM} reduces to
\begin{equation}\label{1DAW}
  p^0\frac{\partial f}{\partial t} + p^1\frac{\partial f}{\partial x} = \frac{U_{\alpha}p^{\alpha}}{\tau}(g-f),
\end{equation}
whose analytical solution is given by
\begin{align}
\nonumber f(x,t,\vec{p})&=\int_{0}^{t}g(x',t',\vec{p})\exp\left(-\int_{t'}^{t}
\frac{U_{\alpha}(x'',t'')p^{\alpha}}{p^{0}\tau}dt''\right)
\frac{U_{\alpha}(x',t')p^{\alpha}}{p^{0}\tau}dt'\\
\label{eq:1DAWsolu}&+\exp\left(-\int_{0}^{t}
\frac{U_{\alpha}(x',t')p^{\alpha}}{\tau p^{0}}dt'\right)f_{0}(x-v_1t,\vec{p}),
\end{align}
where $v_1=p^1/p^0$ is the velocity of particle in $x$ direction, $x'=x-v_1(t-t')$ and  $x''=x-v_1(t-t'')$ are {the} particle trajectories,
and $f_{0}$ is the initial particle velocity distribution function, i.e. $f(x,0,\vec{p})=f_{0}(x,\vec{p})$.

 Taking the moments of \eqref{1DAW}
and integrating them over the space-time cell $I_j\times[t_n,t_{n+1})$ {yield}
\begin{equation}
  \int_{t_n}^{t_{n+1}}\int_{I_j}\int_{\mathbb{R}^3}\vec{\Psi}(p^0\frac{\partial f}{\partial t} + p^1\frac{\partial f}{\partial x} -  \frac{U_{\alpha}p^{\alpha}}{\tau}(g-f))d\varXi dxdt=0,
\end{equation}
where $d\varXi = \frac{d^3\vec{p}}{|\vec{p}|}$.
Using the conservation constraints \eqref{convc} gives
\begin{align}\nonumber
  \int_{I_j}&\int_{\mathbb{R}^3}\vec{\Psi} p^0f(x,t_{n+1},\vec{p})d\varXi dx
  =\int_{I_j}\int_{\mathbb{R}^3}\vec{\Psi} p^0f(x,t_{n},\vec{p})d\varXi dx
   \\ & - \int_{t_n}^{t_{n+1}}\int_{\mathbb{R}^3}\vec{\Psi} p^1\left(f(x_{j+\frac{1}{2}},t,\vec{p}) - f(x_{j-\frac{1}{2}},t,\vec{p})\right)~d\varXi dxdt,\label{BGK1D-0001}
\end{align}
which is the starting point of our 1D second-order accurate BGK scheme.
If  replacing  the distribution
  $f(x_{j\pm \frac{1}{2}},t,\vec{p})$  in \eqref{BGK1D-0001} with an approximate distribution $\hat{f}(x_{j\pm \frac{1}{2}},t,\vec{p})$, then one gets the following finite volume scheme
\begin{equation}
  \vec{\bar{W}}^{n+1}_j = \vec{\bar{W}}^{n}_j - \frac{\Delta t_n}{\Delta x}(\hat{\vec{F}}^n_{j+\frac{1}{2}} - \hat{\vec{F}}^n_{j-\frac{1}{2}}),
\end{equation}
where the numerical flux $\hat{\vec{F}}^n_{j+\frac{1}{2}}$ is given by
\begin{equation}\label{eq:1DF}
  \hat{\vec{F}}^n_{j+\frac{1}{2}}=\frac{1}{\Delta t_n}\int_{t_n}^{t_{n+1}}\int_{\mathbb{R}^3}\vec{\Psi} p^1\hat{f}(x_{j+\frac{1}{2}},t,\vec{p})d\varXi dxdt,
\end{equation}
with
\begin{align}\label{eq:faprox}
\nonumber \hat{f}(x_{j+\frac{1}{2}},t,\vec{p})&=\int_{t_n}^{t}g_h(x',t',\vec{p})\exp\left(-\int_{t'}^{t}
\frac{U_{\alpha}(x'',t'')p^{\alpha}}{p^{0}\tau}dt''\right)
\frac{U_{\alpha}(x',t')p^{\alpha}}{p^{0}\tau}dt'\\
&+\exp\left(-\int_{t_n}^{t}
\frac{U_{\alpha}(x',t')p^{\alpha}}{p^{0}\tau}dt'\right)f_{h,0}(x_{j+\frac{1}{2}}-v_1(t-t_n),\vec{p}),
\end{align}
{here} $v_1=p^1/p^0$, $x'=x_{j+\frac{1}{2}}-v_1(t-t')$, $x''=x_{j+\frac{1}{2}}-v_1(t-t'')$, $f_{h,0}(x_{j+\frac{1}{2}}-v_1(t-t_n),\vec{p})\approx f_{0}(x_{j+\frac{1}{2}}-v_1(t-t_n),\vec{p})$
and $g_h(x',t',\vec{p})\approx g(x',t',\vec{p})$.
%
It is worth noting that it is very expensive to get $U_{\alpha}(x'',t'')$ and $U_{\alpha}(x',t')$ {at} the right hand side of \eqref{eq:faprox}.   In practice, $U_{\alpha}(x'',t'')$ and $U_{\alpha}(x',t')$ in the first term may be approximated  as  $U_{\alpha,j+\frac{1}{2}}^n$ while
  $U_{\alpha}(x',t')$  in the second term may be simplified as $U_{\alpha,j+\frac{1}{2},L}^n$ or $U_{\alpha,j+\frac{1}{2},R}^n$ depending on the sign of $v_1$ and {will be given} in Section \ref{sec:fluxevolution}.

The remaining tasks {are to derive} the
 approximate initial velocity distribution function $f_{h,0}(x_{j+\frac{1}{2}}-v_1(t-t_n),\vec{p})$
 and equilibrium velocity distribution function $g_h(x',t',\vec{p})$.

\subsubsection{Equilibrium distribution $g_0$ at the point $(x_{j+\frac{1}{2}},t_n)$}
\label{sec:fluxevolution}

At the cell interface ${x=x_{j+\frac{1}{2}}}$, \eqref{EQ:initial-reconstruction} gives the following left and right  limiting values
\begin{equation}
\begin{aligned}
  \vec{W}^n_{j+\frac{1}{2},L} &:= \vec{W}_h(x_{j+\frac{1}{2}}-0,t_n)=\vec{W}^n_j(x_{j+\frac{1}{2}}),\\
  \vec{W}^n_{j+\frac{1}{2},R} &:= \vec{W}_h(x_{j+\frac{1}{2}}+0,t_n)=\vec{W}^n_{j+1}(x_{j+\frac{1}{2}}),\\
  \vec{W}^{n,x}_{j+\frac{1}{2},L} &:= \frac{d\vec{W}_h}{dx}(x_{j+\frac{1}{2}}-0,t_n)=\frac{d\vec{W}^n_j}{dx}(x_{j+\frac{1}{2}}),\\
  \vec{W}^{n,x}_{j+\frac{1}{2},R} &:= \frac{d\vec{W}_h}{dx}(x_{j+\frac{1}{2}}+0,t_n)=\frac{d\vec{W}^n_{j+1}}{dx}(x_{j+\frac{1}{2}}).
\end{aligned}
\label{eq:RL}
\end{equation}
 Using \eqref{juttner},   $\vec{W}^n_{j+\frac{1}{2},L}$ and $\vec{W}^n_{j+\frac{1}{2},R}$
 {gives}
  the J$\ddot{\text{u}}$ttner distributions at the left and right  of cell interface $x=x_{j+\frac{1}{2}}$ as follows
\begin{equation}
\begin{aligned}
    g_L&=\frac{ n _{j+1/2,L}}{8\pi T_{j+1/2,L}^3}e^{-\frac{U_{\alpha,j+1/2,L}p^{\alpha}}{T_{j+1/2,L}}},\ \
    g_R =\frac{ n _{j+1/2,R}}{8\pi T_{j+1/2,R}^3}e^{-\frac{U_{\alpha,j+1/2,R}p^{\alpha}}{T_{j+1/2,R}}},
\end{aligned}
\label{eq:gLR}
\end{equation}
and the particle four-flow $N^{\alpha}$ and the energy-momentum tensor $T^{\alpha\beta}$ at the point $(x_{j+\frac{1}{2}},t_n)$
\begin{align*}
(N^{0},T^{01},T^{00})_{j+\frac{1}{2}}^{n,T}&:=\int_{\mathbb{R}^{3}\cap{p^{1}    >0}}\vec{\Psi}p^{0}g_{L}d\Xi+\int_{\mathbb{R}^{3}\cap{p^{1}<0}}\vec{\Psi}p^{0}g_{R}d\Xi, \\
(N^{1},T^{11},T^{01})_{j+\frac{1}{2}}^{n,T}&:=\int_{\mathbb{R}^{3}\cap{p^{1}    >0}}\vec{\Psi}p^{1}g_{L}d\Xi+\int_{\mathbb{R}^{3}\cap{p^{1}<0}}\vec{\Psi}p^{1}g_{R}d\Xi.
\end{align*}
Using those and Theorem \ref{thm:NT} calculates the macroscopic quantities $ n ^n_{j+\frac{1}{2}}, T^n_{j+\frac{1}{2}}$, and $U_{{\alpha,j+\frac{1}{2}}}^n$, and then gives
the J$\ddot{\text{u}}$ttner distribution function at the point $(x_{j+\frac{1}{2}},t_n)$ as follows
\begin{align}
    g_0&=\frac{ n ^n_{j+1/2}}{8\pi (T^n_{j+1/2})^3}\exp(-\frac{U_{{\alpha,j+\frac{1}{2}}}^np^{\alpha}}{T^n_{j+1/2}}),
\end{align}
which will be used to derive the equilibrium velocity distribution $g_h(x,t,\vec{p})$,
see Section \ref{subsection4.1.3}.

\subsubsection{Initial distribution function $f_{h,0}(x,\vec{p})$}\label{subsection1d-fh0}
Assuming that $f(x,t,\vec{p})$ and $g(x,t,\vec{p})$ are sufficiently smooth and borrowing the idea in the Chapman-Enskog expansion, $f(x,t,\vec{p})$ is supposed to be
 expanded as follows
\begin{equation}\label{eq:ceEuler}
  f(x,t,\vec{p})=g-\frac{\tau}{U_{\alpha} p^{\alpha}}(p^0 {g_t+p^1g_x})+O(\tau^2)=:g\left(1-\frac{\tau}{U_{\alpha} p^{\alpha}}(p^0 A+p^1 a)\right)+O(\tau^2),
\end{equation}
with \begin{equation}\label{eq:aAform}
A=A_{1}+A_{2}p^{1}+A_{3}p^{0}, \quad a=a_{1}+a_{2}p^{1}+a_{3}p^{0}.
\end{equation}
The conservation constraints \eqref{convc}  give the constraints on $A$ and $a$
\begin{equation}
\label{eq:aAcons}
\int_{\mathbb{R}^{3}}\vec{\Psi}(p^{0}A+p^{1}a)gd\Xi=\int_{\mathbb{R}^{3}}\vec{\Psi}(p^{0}g_{t}+p^{1}g_{x})d\Xi
=\frac{1}{\tau}\int_{\mathbb{R}^{3}}\vec{\Psi}U_{\alpha}p^{\alpha}(g-f)d\Xi=0.
\end{equation}

Setting $t=t_n$
and using \eqref{eq:ceEuler} and the Taylor series expansion of $f(x,t,\vec{p})$ with respect to $x$ from both sides of the cell interface $x=x_{j+\frac{1}{2}}$  give  the following approximate initial non-equilibrium distribution function
\begin{equation}
\label{eq:fh0}
f_{h,0}(x,t_n,\vec{p}):=\left\{\begin{aligned}
&g_{L}\left(1-\frac{\tau}{U_{\alpha,L}p^{\alpha}}(p^{0}A_{L}+p^{1}a_{L})+a_{L}\tilde{x}\right),&\tilde{x}<0,\\
&g_{R}\left(1-\frac{\tau}{U_{\alpha,R}p^{\alpha}}(p^{0}A_{R}+p^{1}a_{R})+a_{R}\tilde{x}\right),&\tilde{x}>0,
\end{aligned}
\right.
\end{equation}
where $\tilde{x}=x-x_{j+\frac{1}{2}}$, $g_{L}$ and $g_{R}$ are given   in \eqref{eq:gLR}, $(a_{L},A_{L})$ and $(a_{R},A_{R})$ are considered as the left and right limits of $(a,A)$ at the cell interface $x=x_{j+\frac{1}{2}}$ respectively.
The slopes $a_L$ and $a_R$ come from the spatial derivative of
J$\ddot{\text{u}}$ttner distribution and have   unique
correspondences with the slopes of the conservative variables $\vec W$ by
\[
<a_{\omega}{p^{0}}>=\vec{W}_{j+\frac{1}{2},\omega}^{n,x},\
\]
where
\[
{ <a_{\omega}>:=\int_{\mathbb{R}^{3}}a_{\omega}g_{\omega} \vec{\Psi}d\Xi},\ \ \omega=L,R.
\]
Those correspondences form  the linear system for the unknow $\vec{a}_\omega:=({a_{\omega,1}, a_{\omega,2}, a_{\omega,3}})^T$
\begin{equation}\label{eq:la}
M_{0}^{\omega}\vec{a}_\omega=\vec{W}_{j+\frac{1}{2},\omega}^{n,x},
\end{equation}
where
 the coefficient matrix $M_{0}^{\omega}$ is given by
\[
  M_{0}^{\omega}=\int_{\mathbb{R}^{3}}p^{0}g_{\omega}\vec{\Psi}\vec{\Psi}^Td\Xi, \ \omega=L, R.
\]
 {Using} the conservation constraints \eqref{eq:aAcons} and  $a_{\omega}$  gives the linear system for $A_\omega$  as follows
\[
<a_{\omega}p^{1}+A_{\omega}p^{0}>=0,
\]
which can be cast into the following form
\begin{equation}\label{eq:lA}
M_{0}^{\omega}\vec{A}_\omega =-M_{1}^{\omega}\vec{a}_\omega,
\end{equation}
with
\begin{equation*}
M_{1}^{\omega}=\int_{\mathbb{R}^{3}}g_{\omega}p^{1}\vec{\Psi}\vec{\Psi}^Td\Xi, \ \omega={L},{R}.
\end{equation*}

The rest is to calculate all elements of $M_0$ and $M_1$, whose superscript $L$ or $R$ has been omitted for the sake of convenience. In the ultra-relativistic limit, those can be exactly gotten.
 Because $p^0=|\vec{p}|$, the triple integrals in $M_0$ and $M_1$ can be simplified by using polar coordinate transformation 
\begin{equation}\label{eq-polar-tran}
  p^1 = |\vec{p}|\xi, \ p^2=|\vec{p}|\sqrt{1-\xi^2}\sin\varphi,\ p^3=|\vec{p}|\sqrt{1-\xi^2}\cos\varphi, \  \xi\in[-1,1], \varphi\in[-\pi,\pi],
\end{equation}
  which implies $d\Xi=|\vec{p}|d|\vec{p}|d\xi d\varphi$.
 In fact, the above  transformation can convert the triple integrals in the matrices $M_0$ and $M_1$
 into a single integral with respect to $|\vec{p}|$ and a double integral with
 respect to $\xi$ and $\varphi$. On the other hand, in the 1D case, the integrands   do  not depend on the variable $\varphi$, so the double integral can further {reduce} to a single integral with respect to $\xi$ which can be exactly calculated.
 Those lead to 
\begin{align}\nonumber
  M_0&=\int_{\mathbb{R}^{3}}p^{0}g\vec{\Psi}\vec{\Psi}^Td\Xi
  =\begin{pmatrix}
    \int^1_{-1}\Phi(x,\xi)d\xi       & \int^1_{-1}\xi\Psi(x,\xi)d\xi    & \int^1_{-1}\Psi(x,\xi)d\xi\\
    \int^1_{-1}\xi\Psi(x,\xi)d\xi    & \int^1_{-1}\xi^2\Upsilon(x,\xi)d\xi    & \int^1_{-1}\xi\Upsilon(x,\xi)d\xi\\
    \int^1_{-1}\Psi(x,\xi)d\xi    & \int^1_{-1}\xi\Upsilon(x,\xi)d\xi    & \int^1_{-1}\Upsilon(x,\xi)d\xi\\
    \end{pmatrix}\\
&=\begin{pmatrix}
     n  U^0           & 4 n  T U^1U^0              &  n  T(4U^1U^1+3)\\
    4 n  T U^1U^0      & 4 n  T^2 U^0(6U^1U^1+1)    & 4 n  T^2 U^1(6U^1U^1+5)\\
     n  T(4U^1U^1+3)   & 4 n  T^2 U^1(6U^1U^1+5)    & 12 n  T^2 U^0(2U^1U^1 + 1)
  \end{pmatrix},\label{M0}
\end{align}
and
\begin{align}\notag
  M_1&=\int_{\mathbb{R}^{3}}p^{1}g\vec{\Psi}\vec{\Psi}^Td\Xi
  =\begin{pmatrix}
    \int^1_{-1}\xi\Phi(x,\xi)d\xi     & \int^1_{-1}\xi^2\Psi(x,\xi)d\xi    & \int^1_{-1}\xi\Psi(x,\xi)d\xi\\
    \int^1_{-1}\xi^2\Psi(x,\xi)d\xi   & \int^1_{-1}\xi^3\Upsilon(x,\xi)d\xi    & \int^1_{-1}\xi^2\Upsilon(x,\xi)d\xi\\
    \int^1_{-1}\xi\Psi(x,\xi)d\xi     & \int^1_{-1}\xi^2\Upsilon(x,\xi)d\xi    & \int^1_{-1}\xi\Upsilon(x,\xi)d\xi\\
    \end{pmatrix}\\
&=\begin{pmatrix}
     n  U^1            &  n  T (4U^1U^1+1)           & 4 n  T U^1U^0\\
     n  T (4U^1U^1+1)   & 12 n  T^2 U^1(2U^1U^1+1)    & 4 n  T^2 U^0(6U^1U^1+1)\\
    4 n  T U^1U^0       & 4 n  T^2 U^0(6U^1U^1+1)     & 4 n  T^2 U^1(6U^1U^1 + 5)
  \end{pmatrix},\label{M1}
\end{align}
where
\begin{align}
  \Phi(x,\xi) &=\frac{1}{2}\frac{ n (x)}{(U^0(x)-\xi U^1(x))^3},\notag\\
  \Psi(x,\xi) &=\frac{3}{2}\frac{( n  T)(x)}{(U^0(x)-\xi U^1(x))^4},\\
  \Upsilon(x,\xi) &= \frac{6( n  T^2)(x)}{(U^0(x)-\xi U^1(x))^5}.\notag
\end{align}

\subsubsection{Equilibrium velocity distribution $g_h(x,t,\vec{p})$}\label{subsection4.1.3}
 Using $\vec{W}_{0}:=\vec{W}_{j+\frac{1}{2}}^{n}$ derived in
Section \ref{sec:fluxevolution} and the approximate cell average values $\vec{\bar{W}}_{j+1}$ and $\vec{\bar{W}}_{j}$  {reconstructs} a cell-vertex based linear polynomial around the cell interface $x=x_{j+\frac{1}{2}}$ as follows
\[
\vec{W}_{0}(x)=\vec{W}_{0}+\vec{W}_{0}^{x}(x-x_{j+\frac{1}{2}}),
\]
where $\vec{W}_{0}^{x}=\frac{1}{\Delta x}(\vec{\bar{W}}_{j+1}-\vec{\bar{W}}_{j})$.
Again the Taylor series expansion of $g$ at the cell interface $x=x_{j+\frac{1}{2}}$
gives
\begin{equation}
\label{eq:gh}
g_{h}(x,t,\vec{p})=g_{0}(1+a_{0}(x-x_{j+\frac{1}{2}})+A_{0}(t-t_n)),
\end{equation}
where $(a_0,A_0)$ are the values of $(a,A)$ at the point $(x_{j+\frac{1}{2}},t_n)$.
Similarly,
the slope $a_0$ comes from the spatial derivative of
J$\ddot{\text{u}}$ttner distribution and has  a unique
correspondence with the slope of the conservative variables $\vec W$ by
\[
{<a_{0} p^0>}=\vec{W}_{0}^{x},
\]
and then the conservation constraints and $a_0$ gives
  the following linear system
$$  \quad <A_{0}p^{0}+a_{0}p^{1}>=0.
  $$
Those can be rewritten as
\[
M_{0}^{0}\vec{a}_{0}=\vec{W}_{0}^{x},\quad M_{0}^{0}\vec{A}_{0}=- M_{1}^{0}\vec{a}_{0},
\]
where $\vec{a}_{0}=(a_{0,1},a_{0,2}, a_{0,3})^T$, $\vec{A}_{0}=(A_{0,1},A_{0,2}, A_{0,3})^T$, and
 $M_0^0$ and $M_1^0$  can be calculated by \eqref{M0} and \eqref{M1} with $ n, T$ and $U^{\alpha}$ instead of  $ n ^n_{j+\frac{1}{2}}, T^n_{j+\frac{1}{2}}$ and $U^{n,\alpha}_{j+\frac{1}{2}}$. Those systems can be solved by using the  subroutine for \eqref{eq:la} and \eqref{eq:lA}.

Up to now, all parameters in the initial gas distribution function $f_{h,0}$ and the equilibrium state $g_h$ have been determined. Substituting  \eqref{eq:fh0} and \eqref{eq:gh} into \eqref{eq:faprox} gives our distribution function $\hat{f}$ at a cell interface $x=x_{j+\frac{1}{2}}$ as follows
\begin{small}
\begin{align}\notag
 &\hat{f}(x_{j+\frac{1}{2}},t,\vec{p})
 =g_0\Big(1-\exp\big(-\frac{U_{\alpha,j+\frac{1}{2}}^np^{\alpha}}{p^0\tau}(t-t_n)\big)\Big)
 \notag\\
&+g_0a_0v_1\Big(\big(t-t_n+\frac{p^0\tau}{U_{\alpha,j+\frac{1}{2}}^n
p^{\alpha}}\big)\exp\big(-\frac{U_{\alpha,j+\frac{1}{2}}^np^{\alpha}}
{p^0\tau}(t-t_n)\big)-\frac{p^0\tau}{U_{\alpha,j+\frac{1}{2}}^np^{\alpha}}\Big)\notag\\
&+g_0A_0\Big((t-t_n)-\frac{p^0\tau}{U_{\alpha,j+\frac{1}{2}}^np^{\alpha}}
\Big(1-\exp\big(-\frac{U_{\alpha,j+\frac{1}{2}}^np^{\alpha}}{p^0\tau}(	t-t_n)\big)\Big)\Big)\notag\\
&+H[v_1]g_L\big(1-\frac{\tau}{U_{\alpha,j+\frac{1}{2},L}^np^{\alpha}}(p^0A_L+p^1a_L)-a_Lv_1(t-t_n)\big)
\exp\big(-\frac{U_{\alpha,j+\frac{1}{2},L}^np^{\alpha}}{p^0\tau}(t-t_n)\big)
\notag\\
&+(1-H[v_1])g_R\big(1-\frac{\tau}{U_{\alpha,j+\frac{1}{2},R}^np^{\alpha}}
(p^0A_R+p^1a_R)-a_Rv_1(t-t_n)\big)\exp\big(-\frac{U_{\alpha,j+\frac{1}{2},R}^np^{\alpha}}{p^0\tau}(t-t_n)\big),\label{EQ0000000000}
\end{align}
\end{small}
where $H[x]$ is the Heaviside function defined by
\begin{equation*}
  H[x]=\begin{cases}
            0,& x<0,\\
            1,& x\geqslant0.
          \end{cases}
\end{equation*}
Finally, substituting \eqref{EQ0000000000} into the integral \eqref{eq:1DF} yields the numerical flux $\hat{\vec{F}}^n_{j+\frac{1}{2}}$.

\subsection{2D Euler equations}\label{sec:2DGKSEuler}
This section extends the above BGK scheme to the 2D ultra-relativistic Euler equations
\begin{equation}\label{2DEuler}
  \frac{\partial \vec{W}}{\partial t} + \frac{\partial \vec{F}(\vec{W})}{\partial x} + \frac{\partial \vec{G}(\vec{W})}{\partial y}= 0,
\end{equation}
where
\begin{align}
  \vec{W} = \begin{pmatrix}
                   n  U^0\\
                   n  h U^0U^1\\
                   n  h U^0U^2\\
                   n  h U^0U^0 - p
                \end{pmatrix},
\vec{F}(\vec{W}) = \begin{pmatrix}
                   n  U^1\\
                   n  h U^1U^1 + p\\
                   n  h U^2U^1\\
                   n  h U^0U^1
                  \end{pmatrix},
\vec{G}(\vec{W}) = \begin{pmatrix}
               n  U^2\\
               n  h U^1U^2\\
               n  h U^2U^2 + p\\
               n  h U^0U^2
              \end{pmatrix},
\end{align}
with $h=4T$, $p= n  T$, and $\vec u = (u_1, u_2, 0)^T$.
Four real eigenvalues of the Jacobian matrix $A_1(\vec{W})=\partial\vec{F}/\partial \vec{W}$ and $A_2(\vec{W})=\partial\vec{G}/\partial \vec{W}$ can be given as follows
\begin{align}
  \lambda_k^{(1)}&=\frac{u_k(1-c_s^2)-c_s\gamma(\vec{u})\sqrt{1-u_k^2-(|\vec{u}|^2-u_k^2)c_s^2}}
  {1-|\vec{u}|^2c_s^2},\nonumber\\
  \lambda_k^{(2)}&=\lambda_k^{(3)}=u_k,\nonumber\\
  \lambda_k^{(4)}&=\frac{u_k(1-c_s^2)+c_s\gamma(\vec{u})\sqrt{1-u_k^2
  -(|\vec{u}|^2-u_k^2)c_s^2}}{1-|\vec{u}|^2c_s^2},\nonumber
\end{align}
where $k=1,2$, and $c_s = \frac{1}{\sqrt{3}}$ is the speed of sound.

Divide the spatial domain $\Omega$ into a rectangular mesh with the cell $I_{i,j}=\{(x,y)|x_{i-\frac{1}{2}}<x<x_{i+\frac{1}{2}}, y_{j-\frac{1}{2}}<y<y_{j+\frac{1}{2}}\}$, where
$x_{i+\frac{1}{2}} = \frac{1}{2}(x_{i}+x_{i+1}), y_{j+\frac{1}{2}}=\frac{1}{2}(y_{j}+y_{j+1})$,  $x_i=i\Delta x, y_j = j\Delta y$, and $i,j\in\mathbb{Z}$. The time interval $[0,T]$
is also partitioned into a (non-uniform) mesh ${t_{n+1}=t_n+\Delta t_n, t_0=0, n\geqslant0}$, where  the time step size $\Delta t_n$ is determined by
\begin{equation}\label{TimeStep2D}
  \Delta t_n = \frac{C \min\{\Delta x, \Delta y\}}{\max\limits_{ij}\{\bar{\varrho}^1_{i,j},\bar{\varrho}^2_{i,j}\}},
\end{equation}
 the constant $C$ denotes the CFL number, and $\bar{\varrho}^k_{i,j}$ denotes the  approximation of the spectral radius of $A_k(\vec{W})$ over the cell  $I_{i,j}$, $k=1,2$.

The  2D Anderson-Witting model becomes
\begin{equation}\label{2DAW}
  p^0\frac{\partial f}{\partial t} + p^1\frac{\partial f}{\partial x} + p^2\frac{\partial f}{\partial y}= \frac{U_{\alpha}p^{\alpha}}{\tau}(g-f),
\end{equation}
whose analytical solution can be given by
\begin{align}
\nonumber f(x,y,t,\vec{p})=&\int_{0}^{t}g(x',y', t',\vec{p})\exp\left(-\int_{t'}^{t}
\frac{U_{\alpha}(x'',y'',t'')p^{\alpha}}{p^{0}\tau}dt''\right)
\frac{U_{\alpha}(x',y',t')p^{\alpha}}{p^{0}\tau}dt'\\
&+\exp\left(-\int_{0}^{t}
\frac{U_{\alpha}(x',y',t')p^{\alpha}}{\tau p^{0}}dt'\right)f_{0}(x-v_1t,y-v_2t,\vec{p}),
\label{eq:2DAWsolu}
\end{align}
where $v_1=p^1/p^0$ and $v_2=p^2/p^0$ are the particle velocities in $x$  and $y$ directions respectively, $\{x'=x-v_1(t-t'), y'=y-v_2(t-t')\}$  and  $\{x''=x-v_1(t-t''), y''=y-v_2(t-t'')\}$ are the particle trajectories,
and $f_{0}(x,y,\vec{p})$ is the initial particle velocity distribution function, i.e. $f(x,y,0,\vec{p})=f_{0}(x,y,\vec{p})$.

Taking the moments of \eqref{2DAW}
and integrating them over  $I_{i,j}\times[t_n,t_{n+1})$ {yield} the 2D finite volume
scheme
\begin{equation}
\label{eq:2DEulerDisc}
  \vec{\bar{W}}^{n+1}_{i,j} = \vec{\bar{W}}^{n}_{i,j} - \frac{\Delta t_n}{\Delta x}(\hat{\vec{F}}^{n}_{i+\frac{1}{2},j} - \hat{\vec{F}}^n_{i-\frac{1}{2},j}) - \frac{\Delta t_n}{\Delta y}(\hat{\vec{G}}^n_{i,j+\frac{1}{2}} - \hat{\vec{G}}^n_{i,j-\frac{1}{2}}),
\end{equation}
where $\vec{\bar{W}}^{n}_{i,j}$ is the cell average approximation of   conservative vector $\vec{W}(x,y,t)$ over the cell $I_{i,j}$ at time $t_n$, i.e.
\begin{equation*}
  \vec{\bar{W}}^{n}_{i,j} \approx \frac{1}{\Delta x{\Delta y}}\int_{I_{i,j}}\vec{W}(x,y,t_n)dx{dy},
\end{equation*}
and
\begin{align}\label{eq:2DF}
  \hat{\vec{F}}^n_{i+\frac{1}{2},j}&=\frac{1}{\Delta t_n}\int_{t_n}^{t_{n+1}}\int_{\mathbb{R}^3}\vec{\Psi} p^1\hat{f}(x_{i+\frac{1}{2}},y_j,t,\vec{p})d\varXi {dt},\\
    \hat{\vec{G}}^n_{i,j+\frac{1}{2}}&=\frac{1}{\Delta t_n}\int_{t_n}^{t_{n+1}}\int_{\mathbb{R}^3}\vec{\Psi} p^2\hat{f}(x_i,y_{j+\frac{1}{2}},t,\vec{p})d\varXi {dt},
\end{align}
where $\hat{f}(x_{i+\frac{1}{2}},y_j,t,\vec{p})\approx f(x_{i+\frac{1}{2}},y_j,t,\vec{p})$ and
$\hat{f}(x_i,y_{j+\frac{1}{2}},t,\vec{p})\approx f(x_i,y_{j+\frac{1}{2}},t,\vec{p})$.
%
%
Because the derivation of
$\hat{f}(x_i,y_{j+\frac{1}{2}},t,\vec{p})$ is very similar to $\hat{f}(x_{i+\frac{1}{2}},y_j, t,\vec{p})$,
we will mainly derive $\hat{f}(x_{i+\frac{1}{2}},y_j, t,\vec{p})$ with the help of \eqref{eq:2DAWsolu} as follows
\begin{align}
\nonumber \hat{f}(x_{i+\frac{1}{2}},y_j,t,\vec{p})=&\int_{t_n}^{t}g_h(x',y', t',\vec{p})\exp\left(-\int_{t'}^{t}
\frac{U_{\alpha}(x'',y'',t'')p^{\alpha}}{p^{0}\tau}dt''\right)
\frac{U_{\alpha}(x',y',t')p^{\alpha}}{p^{0}\tau}dt'\\
\label{eq:2DAWsoluapprox}&+\exp\left(-\int_{t_n}^{t}
\frac{U_{\alpha}(x',y',t')p^{\alpha}}{\tau p^{0}}dt'\right)f_{h,0}(x_{i+\frac{1}{2}}-v_1\tilde{t},y_j-v_2\tilde{t},\vec{p}),
\end{align}
where $\tilde{t}=t-t_n$, $x'=x_{i+\frac{1}{2}}-v_1(t-t'), y'=y_j-v_2(t-t')$ and $x''=x_{i+\frac{1}{2}}-v_1(t-t''), y''=y_j-v_2(t-t'')$, and $f_{h,0}(x_{i+\frac{1}{2}},y_j-v_1\tilde{t},\vec{p})$ and $g_h(x',y',t',\vec{p})$ are (approximate)   initial   distribution function and equilibrium velocity distribution function, respectively, which will be presented in the following. Similarly,   in order to avoid expensive cost in  getting $U_{\alpha}(x'',y'',t'')$ or $U_{\alpha}(x',y',t')$ along the particle trajectory, $U_{\alpha}(x'',y'',t'')$ and $U_{\alpha}(x',y',t')$ in  \eqref{eq:2DAWsoluapprox} may be taken as a constant $U_{\alpha,{i+\frac{1}{2},j}}^n$, and  $U_{\alpha}(x',y',t')$ in the second term may be replaced with $U_{\alpha,{i+\frac{1}{2},j},L}^n$ or $U_{\alpha,{i+\frac{1}{2},j},R}^n$ which is given in Section \ref{sec:fluxevolution2D}.

\subsubsection{Equilibrium distribution $g_0$ at the point $(x_{i+\frac{1}{2}},y_j, t_n)$}
\label{sec:fluxevolution2D}
Using the cell average values $\{\vec{\bar{W}}^n_{i,j}\}$ reconstructs a  piecewise linear function
\begin{equation}\label{eq-2dreconstruction}
 \vec{W}_h(x,y,t_n)=\sum_{i,j} \vec{W}^n_{i,j}(x,y)\chi_{i,j}(x,y),
\end{equation}
where {$\vec{W}^n_{i,j}(x,y):={\vec{\bar{W}}}^n_{i,j} + \vec{W}^{n,x}_{i,j}(x-x_i) + \vec{W}^{n,y}_{i,j}(y-y_j)$, $\vec{W}^{n,x}_{i,j}$ and $\vec{W}^{n,y}_{i,j}$ are the $x$- and $y$-slopes in the cell $I_{i,j}$, respectively, and $\chi_{i,j}(x,y)$ is the characteristic function of the cell $I_{i,j}$.}
At the point $(x_{i+\frac{1}{2}},y_j)$, the left and right  limiting values of $\vec{W}_h(x,y,t_n)$ are given
by
\begin{equation}
\begin{aligned}
  \vec{W}^n_{i+\frac{1}{2},j,L} &:= \vec{W}_h(x_{i+\frac{1}{2}}-0,y_j,t_n)=\vec{W}^n_{i,j}(x_{i+\frac{1}{2}},y_j),\\
  \vec{W}^n_{i+\frac{1}{2},j,R} &:= \vec{W}_h(x_{j+\frac{1}{2}}+0,y_j,t_n)=\vec{W}^n_{{i+1,j}}(x_{i+\frac{1}{2}},y_j),\\
  \vec{W}^{n,x}_{i+\frac{1}{2},j,L} &:= \frac{d\vec{W}_h}{dx}(x_{i+\frac{1}{2}}-0,y_j,t_n)=\frac{d\vec{W}^n_{i,j}}{dx}(x_{i+\frac{1}{2}},y_j),\\
  \vec{W}^{n,x}_{i+\frac{1}{2},j,R} &:= \frac{d\vec{W}_h}{dx}(x_{i+\frac{1}{2}}+0,y_j,t_n)=\frac{d\vec{W}^n_{{i+1,j}}}{dx}(x_{i+\frac{1}{2}},y_j),\\
  \vec{W}^{n,y}_{i+\frac{1}{2},j,L} &:= \frac{d\vec{W}_h}{dy}(x_{i+\frac{1}{2}}-0,y_j,t_n)=\frac{d\vec{W}^n_{i,j}}{dy}(x_{i+\frac{1}{2}},y_j),\\
  \vec{W}^{n,y}_{i+\frac{1}{2},j,R} &:= \frac{d\vec{W}_h}{dy}(x_{i+\frac{1}{2}}+0,y_j,t_n)=\frac{d\vec{W}^n_{{i+1,j}}}{dy}(x_{i+\frac{1}{2}},y_j).
\end{aligned}
\label{eq:RL2D}
\end{equation}
Similar to the 1D case, with the help of $\vec{W}^n_{i+\frac{1}{2},j,L}$, $\vec{W}^n_{i+\frac{1}{2},j,R}$ and J$\ddot{\text{u}}$ttner distribution \eqref{juttner}, one can get $g_L$ and $g_R$ at $(x_{i+\frac{1}{2}},y_j,{t_n})$.
Then the particle four-flow $N^{\alpha}$ and the energy-momentum tensor $T^{\alpha\beta}$ at $(x_{j+\frac{1}{2}},y_j,t_n)$
can be defined by
\begin{align*}
&(N^{0},T^{01},T^{02},T^{00})_{i+\frac{1}{2},j}^{n,T}:=\int_{\mathbb{R}^{3}\cap{p^{1}    >0}}\vec{\Psi}p^{0}g_{L}d\Xi+\int_{\mathbb{R}^{3}\cap{p^{1}<0}}\vec{\Psi}p^{0}g_{R}d\Xi,\\
&(N^{1},T^{11},T^{21},T^{01})_{i+\frac{1}{2},j}^{n,T}:=\int_{\mathbb{R}^{3}\cap{p^{1}    >0}}\vec{\Psi}p^{1}g_{L}d\Xi+\int_{\mathbb{R}^{3}\cap{p^{1}<0}}\vec{\Psi}p^{1}g_{R}d\Xi,\\
&(N^{2},T^{12},T^{22},T^{02})_{i+\frac{1}{2},j}^{n,T}:=\int_{\mathbb{R}^{3}\cap{p^{1}    >0}}\vec{\Psi}p^{2}g_{L}d\Xi+\int_{\mathbb{R}^{3}\cap{p^{1}<0}}\vec{\Psi}p^{2}g_{R}d\Xi.
\end{align*}
Using those and Theorem \ref{thm:NT}, the macroscopic quantities   $ n ^n_{i+\frac{1}{2},j}, T^n_{i+\frac{1}{2},j}$ and $U_{{\alpha,i+\frac{1}{2},j}}^n$ can be calculated and  then the J$\ddot{\text{u}}$ttner distribution function $g_0$ at $(x_{i+\frac{1}{2}},y_j, t_n)$ is   obtained.

Similarly, in the $y$-direction, $\vec{W}^n_{i,j+\frac{1}{2},L}$ and $\vec{W}^n_{i,j+\frac{1}{2},R}$ can also
be given by \eqref{eq-2dreconstruction} so that one has corresponding left and right equilibrium distributions $\tilde{g}_L$ and
$\tilde{g}_R$. The particle four-flow $N^{\alpha}$ and the energy-momentum tensor $T^{\alpha\beta}$ at $(x_i,y_{j+\frac{1}{2}}, t_n)$ are defined by
\begin{align*}
&(N^{0},T^{01},T^{02},T^{00})_{i,j+\frac{1}{2}}^{n,T}:=\int_{\mathbb{R}^{3}\cap{p^{2}    >0}}\vec{\Psi}p^{0}\tilde{g}_{L}d\Xi+\int_{\mathbb{R}^{3}\cap{p^{2}<0}}\vec{\Psi}p^{0}\tilde{g}_{R}d\Xi,\\
&(N^{1},T^{11},T^{21},T^{01})_{i,j+\frac{1}{2}}^{n,T}:=\int_{\mathbb{R}^{3}\cap{p^{2}    >0}}\vec{\Psi}p^{1}\tilde{g}_{L}d\Xi+\int_{\mathbb{R}^{3}\cap{p^{2}<0}}\vec{\Psi}p^{1}\tilde{g}_{R}d\Xi,\\
&(N^{2},T^{12},T^{22},T^{02})_{i,j+\frac{1}{2}}^{n,T}:=\int_{\mathbb{R}^{3}\cap{p^{2}    >0}}\vec{\Psi}p^{2}\tilde{g}_{L}d\Xi+\int_{\mathbb{R}^{3}\cap{p^{2}<0}}\vec{\Psi}p^{2}\tilde{g}_{R}d\Xi,
\end{align*}
which give   $ n ^n_{i,j+\frac{1}{2}}, T^n_{i,j+\frac{1}{2}}$, $U_{{\alpha,i,j+\frac{1}{2}}}^n$ and $g_0$ at $(x_{i},y_{j+\frac{1}{2}}, t_n)$.

The following will  derive the initial distribution function $f_{h,0}(x,y,\vec{p})$ and equilibrium  distribution $g_h(x,y,t,\vec{p})$, separately.

\subsubsection{Initial  distribution function $f_{h,0}(x,y,\vec{p})$}
\label{Section-Initial-velocity-distribution}
Borrowing the idea in the Chapman-Enskog expansion, $f(x,y,t,\vec{p})$ is supposed to be of the form
\begin{small}
\begin{equation}\label{eq:ceEuler2D000}
  f(x,y,t,\vec{p})=g-\frac{\tau}{U_{\alpha} p^{\alpha}}\left(p^0{g_t+p^1g_x+p^2g_y}\right)+O(\tau^2)=:g\left(1-\frac{\tau}{U_{\alpha} p^{\alpha}}\left(p^0A+p^1a+p^2b\right)\right)+O(\tau^2).
\end{equation}
\end{small}
The conservation constraints \eqref{convc} imply the constraints on $A,a$ and $b$
\begin{equation}
\label{eq:aAcons2D}
\int_{\mathbb{R}^{3}}\vec{\Psi}(p^{0}A+p^{1}a+p^2b)gd\Xi=\int_{\mathbb{R}^{3}}\vec{\Psi}(p^{0}g_{t}+p^{1}g_{x}+p^2g_y)d\Xi
=\frac{1}{\tau}\int_{\mathbb{R}^{3}}\vec{\Psi}U_{\alpha}p^{\alpha}(g-f)d\Xi=0.
\end{equation}
Using the Taylor series expansion of $f$  at the cell interface $(x_{i+\frac{1}{2}},y_j)$ gives
\begin{equation}
\label{eq:fh02D}
f_{h,0}=\left\{\begin{aligned}
&g_{L}\left(1-\frac{\tau}{U_{\alpha,L}p^{\alpha}}(p^{0}A_{L}+p^{1}a_{L}+p^{2}b_{L})+a_{L}\tilde{x} + b_{L}\tilde{y}\right),&\tilde{x}<0,\\
&g_{R}\left(1-\frac{\tau}{U_{\alpha,R}p^{\alpha}}(p^{0}A_{R}+p^{1}a_{R}+p^{2}b_{L})+a_{R}\tilde{x} + b_{R}\tilde{y}\right),&\tilde{x}>0,
\end{aligned}
\right.
\end{equation}
where $\tilde{x}=x-x_{i+\frac{1}{2}}, \tilde{y}=y-y_j$, and
 $(a_{\omega}, b_{\omega}, A_{\omega})$, $\omega=L,R$, are of the form
\begin{align}\label{eq:aAform2D}\begin{aligned}
a_{\omega}=&a_{\omega,1}+a_{\omega,2}p^{1}+a_{\omega,3}p^{2}+a_{\omega,4}p^{0}, \\
b_{\omega}=&b_{\omega,1}+b_{\omega,2}p^{1}+b_{\omega,3}p^{2}+b_{\omega,4}p^{0}, \\
A_{\omega}=&A_{\omega,1}+A_{\omega,2}p^{1}+A_{\omega,3}p^{2}+A_{\omega,4}p^{0}.
\end{aligned}\end{align}
The slopes $a_\omega$ and $b_\omega$ come from the spatial derivative of
J$\ddot{\text{u}}$ttner distribution and have   unique
correspondences with the slopes of the conservative variables $\vec W$ by
 the following linear systems for $a_\omega$ and $b_\omega$
\begin{align*}
<a_{\omega}{p^0}>=\vec{W}_{i+\frac{1}{2},j,\omega}^{n,x},\ \
<b_{\omega}{p^0}>=\vec{W}_{i+\frac{1}{2},j,\omega}^{n,y}, \ \omega=L,R.
\end{align*}
Those linear systems can also be expressed as follows
\begin{align*}
    M_{0}^{\omega}{\vec{a}}_\omega=\vec{W}_{i+\frac{1}{2},j,\omega}^{n,x},\quad
    M_{0}^{\omega}{\vec{b}}_\omega=\vec{W}_{i+\frac{1}{2},j,\omega}^{n,y},
\end{align*}
where the coefficient matrix is defined by
\[
  M_{0}^{\omega}=\int_{\mathbb{R}^{3}}p^{0}g_{\omega}\vec{\Psi}\vec{\Psi}^Td\Xi.
\]

Substituting $a_\omega$ and $b_\omega$ into the conservation constraints \eqref{eq:aAcons2D} gives the linear systems for $A_\omega$  as follows
\[
<a_{\omega}p^{1}+b_{\omega}p^{2}+A_{\omega}p^{0}>=0,\ \omega=L,R,
\]
which can be rewritten as
\begin{equation}\label{eq:lA2D}
M_{0}^{\omega}\vec{A}_{\omega}=-M_{1}^{\omega}\vec{a}_{\omega}-M_{2}^{\omega}\vec{b}_{\omega},
\end{equation}
where
\begin{align*}
M_{1}^{\omega}=\int_{\mathbb{R}^{3}}g_{\omega}p^{1}\vec{\Psi}\vec{\Psi}^Td\Xi,\ \
  M_{2}^{\omega}=\int_{\mathbb{R}^{3}}g_{\omega}p^{2}\vec{\Psi}\vec{\Psi}^Td\Xi, \ \omega=L,R.
\end{align*}
All elements of the  matrices $M_0^{\omega}$, $M_1^{\omega}$ and $M_2^{\omega}$ can also be explicitly presented
by using the coordinate transformation \eqref{eq-polar-tran}. 
If omitting  the superscripts $L$ and $R$, then
the matrices $M_0$, $M_1$, and $M_2$ are
\begin{small}
\begin{align}\notag
  M_0&=\int_{\mathbb{R}^{3}}p^{0}g\vec{\Psi}\vec{\Psi}^Td\Xi:=\begin{pmatrix}
                                                                M^0_{00} &M^0_{01} &M^0_{02} &M^0_{03}\\
                                                                M^0_{10} &M^0_{11} &M^0_{12} &M^0_{13}\\
                                                                M^0_{20} &M^0_{21} &M^0_{22} &M^0_{23}\\
                                                                M^0_{30} &M^0_{31} &M^0_{32} &M^0_{33}
                                                              \end{pmatrix}\\
  &=\begin{pmatrix}
  \int^{\pi}_{-\pi}\int^1_{-1}\Phi d\xi d\varphi    &\int^{\pi}_{-\pi}\int^1_{-1}w^1\Psi d\xi d\varphi &\int^{\pi}_{-\pi}\int^1_{-1}w^2\Psi d\xi d\varphi &\int^{\pi}_{-\pi}\int^1_{-1}\Psi d\xi d\varphi \notag \\
  \int^{\pi}_{-\pi}\int^1_{-1}w^1\Psi d\xi d\varphi    &\int^{\pi}_{-\pi}\int^1_{-1}(w^1)^2\Upsilon d\xi d\varphi &\int^{\pi}_{-\pi}\int^1_{-1}w^1w^2\Upsilon d\xi d\varphi &\int^{\pi}_{-\pi}\int^1_{-1}w^1\Upsilon d\xi d\varphi \notag \\
  \int^{\pi}_{-\pi}\int^1_{-1}w^2\Psi d\xi d\varphi    &\int^{\pi}_{-\pi}\int^1_{-1}w^2w^1\Upsilon d\xi d\varphi &\int^{\pi}_{-\pi}\int^1_{-1}(w^2)^2\Upsilon d\xi d\varphi &\int^{\pi}_{-\pi}\int^1_{-1}w^2\Upsilon d\xi d\varphi \\
  \int^{\pi}_{-\pi}\int^1_{-1}\Psi d\xi d\varphi    &\int^{\pi}_{-\pi}\int^1_{-1}w^1\Upsilon d\xi d\varphi &\int^{\pi}_{-\pi}\int^1_{-1}w^2\Upsilon d\xi d\varphi &\int^{\pi}_{-\pi}\int^1_{-1}\Upsilon d\xi d\varphi \notag
  \end{pmatrix}\\
  &=
  \begin{pmatrix}
   n  U^0               & 4 n  TU^1U^0                        & 4 n  TU^2U^0                        &  n  T (4 U^1U^1 + 4 U^2U^2 + 3)\\
  4 n  TU^1 U^0         & 4 n  T^2 (6U^1U^1 + 1) U^0          & 24 n  T^2U^1 U^2 U^0                & 4 n  T^2 U^1 (6 U^1U^1 + 6 U^2U^2 + 5)\\
  4 n  TU^2 U^0         & 24 n  T^2 U^1 U^2 U^0               & 4 n  T^2(6 U^2U^2 + 1) U^0             & 4 n  T^2 U^2 (6 U^1U^1 + 6 U^2U^2+ 5)\\
  M^0_{03}               & M^0_{13}                             & M^0_{23}                             & 12 n  T^2 U^0 (2 U^1U^1 + 2U^2U^2 + 1)
  \end{pmatrix}, \label{M02D}
\end{align}
\end{small}
\begin{small}
\begin{align}\notag
  M_1&=\int_{\mathbb{R}^{3}}p^{1}g\vec{\Psi}\vec{\Psi}^Td\Xi:=\begin{pmatrix}
                                                                M^1_{00} &M^1_{01} &M^1_{02} &M^1_{03}\\
                                                                M^1_{10} &M^1_{11} &M^1_{12} &M^1_{13}\\
                                                                M^1_{20} &M^1_{21} &M^1_{22} &M^1_{23}\\
                                                                M^1_{30} &M^1_{31} &M^1_{32} &M^1_{33}
                                                           \end{pmatrix}\\
  &=
  \begin{pmatrix}
  \int^{\pi}_{-\pi}\int^1_{-1}w^1\Phi d\xi d\varphi    &\int^{\pi}_{-\pi}\int^1_{-1}(w^1)^2\Psi d\xi d\varphi &\int^{\pi}_{-\pi}\int^1_{-1}w^1w^2\Psi d\xi d\varphi &\int^{\pi}_{-\pi}\int^1_{-1}w^1\Psi d\xi d\varphi \notag \\
  \int^{\pi}_{-\pi}\int^1_{-1}(w^1)^2\Psi d\xi d\varphi    &\int^{\pi}_{-\pi}\int^1_{-1}(w^1)^3\Upsilon d\xi d\varphi &\int^{\pi}_{-\pi}\int^1_{-1}(w^1)^2w^2\Upsilon d\xi d\varphi &\int^{\pi}_{-\pi}\int^1_{-1}(w^1)^2\Upsilon d\xi d\varphi \notag \\
  \int^{\pi}_{-\pi}\int^1_{-1}w^1w^2\Psi d\xi d\varphi    &\int^{\pi}_{-\pi}\int^1_{-1}(w^1)^2w^2\Upsilon d\xi d\varphi &\int^{\pi}_{-\pi}\int^1_{-1}w^1(w^2)^2\Upsilon d\xi d\varphi &\int^{\pi}_{-\pi}\int^1_{-1}w^1w^2\Upsilon d\xi d\varphi \\
  \int^{\pi}_{-\pi}\int^1_{-1}w^1\Psi d\xi d\varphi    &\int^{\pi}_{-\pi}\int^1_{-1}(w^1)^2\Upsilon d\xi d\varphi &\int^{\pi}_{-\pi}\int^1_{-1}w^1w^2\Upsilon d\xi d\varphi &\int^{\pi}_{-\pi}\int^1_{-1}w^1\Upsilon d\xi d\varphi \notag
  \end{pmatrix}\\
  &=
  \begin{pmatrix}
     n  U^1             &  n  T(4U^1U^1 + 1)         & 4 n  TU^1U^2              & 4 n  TU^1U^0\\
     n  T(4U^1U^1 + 1)  & 12 n  T^2 U^1(2U^1U^1 + 1)  & 4 n  T^2 U^2(6U^1U^1 + 1)  & 4 n  T^2 U^0(6U^1U^1 + 1)\\
    4 n  TU^1U^2        & 4 n  T^2 U^2(6U^1U^1 + 1)   & 4 n  T^2 U^1(6U^2U^2 + 1)  & 24 n  T^2 U^1U^2U^0\\
    4 n  TU^1U^0        & 4 n  T^2 U^0(6U^1U^1 + 1)   & 24 n  T^2 U^1U^2U^0        & 4 n  T^2 U^1(6U^1U^1 + 6U^2U^2 + 5)
  \end{pmatrix},\label{M12D}
\end{align}
\end{small}
and
  \begin{small}
\begin{align}\notag
  M_2&=\int_{\mathbb{R}^{3}}p^{2}g\vec{\Psi}\vec{\Psi}^Td\Xi:=\begin{pmatrix}
                                                                M^2_{00} &M^2_{01} &M^2_{02} &M^2_{03}\\
                                                                M^2_{10} &M^2_{11} &M^2_{12} &M^2_{13}\\
                                                                M^2_{20} &M^2_{21} &M^2_{22} &M^2_{23}\\
                                                                M^2_{30} &M^2_{31} &M^2_{32} &M^2_{33}
                                                           \end{pmatrix}\\
  &=
\begin{pmatrix}
  \int^{\pi}_{-\pi}\int^1_{-1}w^2\Phi d\xi d\varphi    &\int^{\pi}_{-\pi}\int^1_{-1}w^2w^1\Psi d\xi d\varphi  & \int^{\pi}_{-\pi}\int^1_{-1}(w^2)^2\Psi d\xi d\varphi &\int^{\pi}_{-\pi}\int^1_{-1}w^2\Psi d\xi d\varphi \notag \\
  \int^{\pi}_{-\pi}\int^1_{-1}w^2w^1\Psi d\xi d\varphi    &\int^{\pi}_{-\pi}\int^1_{-1}w^2(w^1)^2\Upsilon d\xi d\varphi &\int^{\pi}_{-\pi}\int^1_{-1}w^1(w^2)^2\Upsilon d\xi d\varphi &\int^{\pi}_{-\pi}\int^1_{-1}w^1w^2\Upsilon d\xi d\varphi \notag \\
  \int^{\pi}_{-\pi}\int^1_{-1}(w^2)^2\Psi d\xi d\varphi    &\int^{\pi}_{-\pi}\int^1_{-1}(w^2)^2w^1\Upsilon d\xi d\varphi &\int^{\pi}_{-\pi}\int^1_{-1}(w^2)^3\Upsilon d\xi d\varphi &\int^{\pi}_{-\pi}\int^1_{-1}(w^2)^2\Upsilon d\xi d\varphi \notag\\
  \int^{\pi}_{-\pi}\int^1_{-1}w^2\Psi d\xi d\varphi    &\int^{\pi}_{-\pi}\int^1_{-1}w^1w^2\Upsilon d\xi d\varphi &\int^{\pi}_{-\pi}\int^1_{-1}(w^2)^2\Upsilon d\xi d\varphi &\int^{\pi}_{-\pi}\int^1_{-1}w^2\Upsilon d\xi d\varphi \notag
\end{pmatrix}\\
  &=
  \begin{pmatrix}
     n  U^2                &  4 n  TU^1U^2             &  n  T(4U^2U^2 + 1)          & 4 n  TU^2U^0\\
    4 n  TU^1U^2         & 4 n  T^2U^2(6U^1U^1 + 1)  & 4 n  T^2U^1(6U^2U^2 + 1)    & 24 n  T^2U^1U^2U^0\\
     n  T(4U^2U^2 + 1)   & 4 n  T^2U^1(6U^2U^2 + 1)  & 12 n  T^2U^2(2U^2U^2 + 1)   & 4 n  T^2U^0(6U^2U^2 + 1)\\
    4 n  TU^2U^0         & 24 n  T^2U^1U^2U^0        & 4 n  T^2U^0(6U^2U^2 + 1)    & 4 n  T^2U^2(6U^1U^1 + 6U^2U^2 + 5)
  \end{pmatrix},\label{M22D}
\end{align}
\end{small}
where  $w^1=\xi, w^2=\sqrt{1-\xi^2}\sin\varphi, w^3=\sqrt{1-\xi^2}\cos\varphi$, and
\begin{align}
  \Phi(x,y,\xi,\varphi) &=\frac{1}{4\pi}\frac{ n (x,y)}{(U^0(x,y)-w^1U^1(x,y)-w^2U^2(x,y))^3},\notag\\
  \Psi(x,y,\xi,\varphi) &=\frac{3}{4\pi}\frac{( n  T)(x,y)}{(U^0(x,y)-w^1U^1(x,y)-w^2U^2(x,y))^4},\\
  \Upsilon(x,y,\xi,\varphi) &= \frac{3}{\pi}\frac{( n  T^2)(x,y)}{(U^0(x,y)-w^1U^1(x,y)-w^2U^2(x,y))^5}\notag.
\end{align}

\subsubsection{Equilibrium velocity distribution $g_h(x,y,t,\vec p)$}
\label{sec:equi2DEuler}
Using $\vec{W}_{0}:=\vec{W}_{i+\frac{1}{2},j}^{n}$ derived in Section \ref{sec:fluxevolution2D} and the cell averages $\vec{\bar{W}}_{i+1,j}$ and $\vec{\bar{W}}_{i,j}$  reconstructs a  linear polynomial
\[
\vec{W}_{0}(x)=\vec{W}_{0}+\vec{W}_{0}^{x}(x-x_{i+\frac{1}{2}}) + \vec{W}_{0}^{y}(y-y_j),
\]
where $\vec{W}_{0}^{x}=\frac{1}{\Delta x}(\vec{\bar{W}}_{i+1,j}-\vec{\bar{W}}_{i,j})$ and $\vec{W}_{0}^{y}=\frac{1}{2\Delta y}(\vec{W}^{n}_{i+\frac12,j+1}-\vec{W}^{n}_{i+\frac12,j-1})$.
Again using the Taylor series expansion of $g$ at the cell interface $(x_{i+\frac{1}{2}},y_j)$ gives
\begin{equation}
\label{eq:gh2D}
g_{h}(x,y,t,\vec{p})=g_{0}(1+a_{0}(x-x_{i+\frac{1}{2}})+b_0(y-y_j)+A_{0}(t-t^n)),
\end{equation}
where $(a_0,b_0,A_0)$ are the values of $(a,b,A)$ at the point $(x_{i+\frac{1}{2}},y_j, t_n)$.
Similarly, 
the linear systems for $a_0, b_0$ and $A_0$ can be derived as follows
\[
<a_{0}{p^0}>=\vec{W}_{0}^{x},\quad <b_{0}{p^0}>=\vec{W}_{0}^{y},\quad <A_{0}p^{0}+a_{0}p^{1}+b_{0}p^{2}>=0,
\]
or
\begin{equation}\label{eq:abAforg}
M_{0}^{0}\vec{a}_{0}=\vec{W}_{0}^{x},\quad M_{0}^{0}\vec{b}_{0}=\vec{W}_{0}^{y}, \quad M_{0}^{0}\vec{A}_{0}= -M_{1}^{0}\vec{a}_{0} -M_{2}^{0}\vec{b}_{0},
\end{equation}
where the elements of $M_0^0, M_1^0$ and $M_2^0$ are given by  \eqref{M02D}, \eqref{M12D}, and \eqref{M22D} with $ n ,T,U^{\alpha}$   instead of $ n ^n_{i+\frac{1}{2},j}, T^n_{i+\frac{1}{2},j}$ and $U^{n,\alpha}_{i+\frac{1}{2},j}$.

Up to now, the initial gas distribution function $f_{h,0}$ and the equilibrium state $g_h$
have been given. Substituting \eqref{eq:fh02D} and \eqref{eq:gh2D} into \eqref{eq:2DAWsoluapprox}
 gives
\begin{small}
\begin{align}
 &\hat{f}(x_{i+\frac{1}{2}},y_j,t,\vec{p})
 =g_0\left(1-\exp\left(-\frac{U_{\alpha,i+\frac{1}{2},j}^np^{\alpha}}{p^0\tau}\tilde{t}\right)\right)\notag\\
&+g_0a_0v_1\left(\left(\tilde{t}+\frac{p^0\tau}{U_{\alpha,i+\frac{1}{2},j}^np^{\alpha}}\right)\exp\left(-\frac{U_{\alpha,i+\frac{1}{2},j}^np^{\alpha}}{p^0\tau}\tilde{t}\right)-\frac{p^0\tau}{U_{\alpha,i+\frac{1}{2},j}^np^{\alpha}}\right)\notag\\
&+g_0b_0v_2\left(\left(\tilde{t}+\frac{p^0\tau}{U_{\alpha,i+\frac{1}{2},j}^np^{\alpha}}\right)\exp\left(-\frac{U_{\alpha,i+\frac{1}{2},j}^np^{\alpha}}{p^0\tau}\tilde{t}\right)-\frac{p^0\tau}{U_{\alpha,i+\frac{1}{2},j}^np^{\alpha}}\right)\notag\\
&+g_0A_0\left(\tilde{t}-\frac{p^0\tau}{U_{\alpha,i+\frac{1}{2},j}^np^{\alpha}}\left(1-\exp\left(-\frac{U_{\alpha,i+\frac{1}{2},j}^np^{\alpha}}{p^0\tau}\tilde{t}\right)\right)\right)\notag\\
&+H[v_1]g_L\left(1-\frac{\tau}{U_{\alpha,i+\frac{1}{2},j,L}^np^{\alpha}}(p^0A_L+p^1a_L+p^2b_L)-a_Lv_1\tilde{t}-b_Lv_2\tilde{t}\right)\exp\left(-\frac{U_{\alpha,i+\frac{1}{2},j,L}^np^{\alpha}}{p^0\tau}\tilde{t}\right)\notag\\
&+(1-H[v_1])g_R\left(1-\frac{\tau}{U_{\alpha,i+\frac{1}{2},j,R}^np^{\alpha}}(p^0A_R+p^1a_R+p^2a_R)-a_Rv_1\tilde{t}-b_Rv_2\tilde{t}\right)\exp\left(-\frac{U_{\alpha,i+\frac{1}{2},j,R}^np^{\alpha}}{p^0\tau}\tilde{t}\right),\notag
\end{align}
\end{small}
where $\tilde{t}=t-t_n$.
Combining this $\hat{f}(x_{i+\frac{1}{2}},y_j,t,\vec{p})$ with \eqref{eq:2DF}  can get the numerical flux $\hat{\vec{F}}^n_{i+\frac{1}{2},j}$. The numerical flux $\hat{\vec{G}}^n_{i,j+\frac{1}{2}}$ can be obtained in the same procedure.

\subsection{2D Navier-Stokes equations}
\label{sec:GKSNS}
Because  the previous simple expansion \eqref{eq:ceEuler} or \eqref{eq:ceEuler2D000}
  cannot give the  Navier-Stokes equations \eqref{eq-NS01}-\eqref{eq-NS03},
  one has to use the complicate Chapman-Enskog expansion  \eqref{ce}-\eqref{dev}
 to design the genuine BGK schemes for the Navier-Stokes equations.
On the other hand, for the  Navier-Stokes equations,
calculating the macroscopic quantities $n, U^{\alpha}$, and $p$
needs the value of the fluxes $\vec F^k$
besides  $\vec W$. More specially,
one has to first calculate the {energy-momentum}
tensor $T^{\alpha\beta}$ and {particle} four-flow $N^{\alpha}$ from the kinetic level and
then use  Theorem \ref{thm:NT}  to calculate $n, U^{\alpha}$, and $p$.
It shows that there exists a very big difference between
the genuine BGK schemes for the Euler and Navier-Stokes equations.


In order to obtain $T^{\alpha\beta}$ and $N^{\alpha}$ at $t=t_{n+1}$ from the kinetic level, multiplying \eqref{2DAW} by
$p^k/p^0$  gives
\begin{align}
\label{eq:boltzmannNS2}
      p^k\frac{\partial f}{\partial t} + \frac{p^kp^1}{p^0}\frac{\partial f}{\partial x} + \frac{p^kp^2}{p^0}\frac{\partial f}{\partial y} = \frac{p^kU_{\alpha}p^{\alpha}(g-f)}{p^0\tau}, \ k=1,2.
\end{align}
Taking the moments of \eqref{2DAW} and \eqref{eq:boltzmannNS2}  and integrating them over the space-time domain $I_{i,j}\times[t_n,t_{n+1})$ , respectively, yield
\begin{equation}
  \label{eq:moment1}
  \vec{\bar{W}}^{n+1}_{\alpha,i,j} = \vec{\bar{W}}^{n}_{\alpha,i,j} - \frac{\Delta t_n}{\Delta x}(\hat{\vec{F}}^{n}_{\alpha,i+\frac{1}{2},j} - \hat{\vec{F}}^n_{\alpha,i-\frac{1}{2},j}) - \frac{\Delta t_n}{\Delta y}(\hat{\vec{G}}^n_{\alpha,i,j+\frac{1}{2}} - \hat{\vec{G}}^n_{\alpha,i,j-\frac{1}{2}}) +\vec{S}^n_{\alpha,i,j} ,\ \alpha=0,1,2,
\end{equation}
where
\begin{equation}
\begin{aligned}
\label{eq:relationNS}
    &\vec{\bar{W}}^{n}_{\alpha,i,j} = (N^\alpha,T^{1\alpha},T^{2\alpha},T^{0\alpha})^{n,T}_{i,j},\\
    &\hat{\vec{F}}^{n}_{\alpha,i+\frac{1}{2},j} = \frac{1}{\Delta t_n}\int_{\mathbb{R}^3}\int_{t_n}^{t_{n+1}}\vec{\Psi}\frac{p^1p^\alpha}{p^0}\hat{f}(x_{i+\frac{1}{2}},y_j,t)dtd\varXi,\\
    &\hat{\vec{G}}^n_{\alpha,i,j+\frac{1}{2}} = \frac{1}{\Delta t_n}\int_{\mathbb{R}^3}\int_{t_n}^{t_{n+1}}\vec{\Psi}\frac{p^2p^\alpha}{p^0}\hat{f}(x_i,y_{j+\frac{1}{2}},t)dtd\varXi,\\
    &\vec{S}^n_{0,i,j} =0,\
    \vec{S}^n_{k,i,j} = \int_{\mathbb{R}^3}\int_{t_n}^{t_{n+1}}\vec{\Psi}\frac{p^kU_{\alpha}p^{\alpha}}{p^0\tau}(g(x_i,y_j,t)-\hat{f}(x_i,y_j,t))dtd\varXi, \ k=1,2.
\end{aligned}
\end{equation}
Our task is to get the approximate distributions $\hat{f}(x_{i+\frac{1}{2}},y_j,t)$ and $\hat{f}(x_{i},y_{j+\frac{1}{2}},t)$ for the numerical fluxes
and $\hat{f}(x_{i},y_j,t)$ and $g(x_i,y_j,t)$ for the source terms.
The following will focus on the derivation of  $\hat{f}(x_{i+\frac{1}{2}},y_j,t)$ with the help of  the analytical solution \eqref{eq:2DAWsolu}
of the 2D Anderson-Witting model.

\subsubsection{Initial   distribution function $f_{h,0}(x,y,t,\vec p)$}
This section derives the initial   distribution function $f_{h,0}$ for $\hat{f}(x_{i+\frac{1}{2}},y_j,t)$. The Chapman-Enskog expansion  \eqref{ce}-\eqref{dev} is rewritten as follows
\begin{equation}\label{eq:CENS}
   f(x,y,t,\vec{p})=g\left(1-\frac{\tau}{U_{\alpha}p^{\alpha}}\left(A^{ce}p^0
   +a^{ce}p^1+b^{ce}p^2+c^{ce}p^3\right)\right)+O(\tau^2),
\end{equation}
where
$A^{ce}=A_{\beta}^{ce}p^{\beta}+A^{ce}_4$, $a^{ce} =a^{ce}_{\beta}p^{\beta}+a^{ce}_4$,
$b^{ce} =b^{ce}_{\beta}p^{\beta}+b^{ce}_4$, $c^{ce} =c^{ce}_{\beta}p^{\beta}+c^{ce}_4$,
and
\begin{align}\label{cecoef}\begin{aligned}
  A^{ce}_{\beta} &= -\frac{1}{T}\nabla^{<0}U^{\beta>} + \frac{U_{\beta}}{T^2}(\nabla^{0}T-\frac{T}{ n  h}\nabla^{0}p),\quad
  A^{ce}_4 = -\frac{h}{T^2}(\nabla^{0}T-\frac{T}{ n  h}\nabla^{0}p),
  \\
  a^{ce}_{\beta} &= \frac{1}{T}\nabla^{<1}U^{\beta>} - \frac{U_{\beta}}{T^2}(\nabla^{1}T-\frac{T}{ n  h}\nabla^{1}p),\quad
  a^{ce}_4 = \frac{h}{T^2}(\nabla^{1}T-\frac{T}{ n  h}\nabla^{1}p),\
  \\
  b_{\beta} &= \frac{1}{T}\nabla^{<2}U^{\beta>} - \frac{U_{\beta}}{T^2}(\nabla^{2}T-\frac{T}{ n  h}\nabla^{2}p),\quad
  b^{ce}_4 = \frac{h}{T^2}(\nabla^{2}T-\frac{T}{ n  h}\nabla^{2}p),\\
  c^{ce}_{\beta} &= \frac{1}{T}\nabla^{<3}U^{\beta>} - \frac{U_{\beta}}{T^2}(\nabla^{3}T-\frac{T}{ n  h}\nabla^{3}p),\quad
  c^{ce}_4 = \frac{h}{T^2}(\nabla^{3}T-\frac{T}{ n  h}\nabla^{3}p).
\end{aligned}\end{align}
It is observed from those expressions of $A^{ce}, a^{ce}, b^{ce}$, and $c^{ce}$
that one has to compute the time derivatives, which are not required  in the Euler case.
Those time derivatives are {approximately} computed by using the following second-order extrapolation method:
for any smooth function $h(t)$, the first order derivative at $t=t_n$ is numerically obtained by
\begin{small}
\begin{equation}
\label{eq:extra}
    h_t(t_{n}) = \frac{h(t_{n-2})(t_{n-1}-t_n)^2-h(t_{n-1})(t_{n-2}-t_n)^2-h(t_{n})((t_{n-1}-t_n)^2-(t_{n-2}-t_n)^2)}
    {(t_{n-2}-t_{n})(t_{n-1}-t_n)^2-(t_{n-1}-t_{n})(t_{n-2}-t_n)^2}.
\end{equation}
\end{small}
Using  the Chapman-Enskog expansion \eqref{eq:CENS} and  the Taylor series expansion in terms of $x$
gives the initial velocity distribution
\begin{small}
\begin{equation}
\label{eq:fh0NS}
f_{h,0}(x,y,t^n,\vec{p})=\left\{\begin{aligned}
&g_{L}\left(1-\frac{\tau}{U_{\alpha,L}p^{\alpha}}(p^{0}A_{L}^{ce}+p^{1}a_{L}^{ce}+p^{2}b_{L}^{ce}+p^{3}c_{L}^{ce})+a_{L}\tilde{x}+b_{L}\tilde{y}\right),\tilde{x}<0,\\
&g_{R}\left(1-\frac{\tau}{U_{\alpha,R}p^{\alpha}}(p^{0}A_{R}^{ce}+p^{1}a_{R}^{ce}+p^{2}b_{R}^{ce}+p^{3}c_{R}^{ce})+a_{R}\tilde{x}+b_{R}\tilde{y}\right),\tilde{x}>0,
\end{aligned}
\right.
\end{equation}
\end{small}
where $\tilde{x}=x-x_{i+\frac{1}{2}}, \tilde{y}=y-y_{j}$, $g_{L}$ and $g_{R}$ denote the left and right J$\ddot{\text{u}}$ttner distributions at   $x_{i+\frac{1}{2}}$ with  $y=y_j,t=t_n$, the Taylor expansion coefficients $(a_{L},b_{L})$ and $(a_{R},b_{R})$ are calculated by using the same procedure as in the Euler case,
while the Chapman-Enskog expansion coefficients  $a_{L}^{ce}, a_{R}^{ce}, b_{L}^{ce}, b_{R}^{ce}, c_{L}^{ce}, c_{R}^{ce}$ and  $A_{L}^{ce},A_{R}^{ce}$ are calculated by \eqref{cecoef}.

\subsubsection{Equilibrium distribution functions $g_h(x,y,t,\vec p)$}
In order to obtain the equilibrium  distribution functions $g_h(x,y,t,\vec p)$ for $\hat{f}(x_{i+\frac{1}{2}},y_j,t)$, the particle four-flow $N^{\alpha}$ and the energy-momentum tensor $T^{\alpha\beta}$ at $(x_{i+\frac{1}{2}},y_j)$ and $t=t^n$  are defined by
\begin{align*}
(N^{\alpha},T^{\alpha1},T^{\alpha2},T^{\alpha0})_{i+\frac{1}{2},j}^{n,T}:=\int_{\mathbb{R}^{3}\cap{p^{1}>0}}\vec{\Psi}p^{\alpha}f_{L}d\Xi+\int_{\mathbb{R}^{3}\cap{p^{1}<0}}\vec{\Psi}p^{\alpha}f_{R}d\Xi,\ \alpha=0,1,2,
\end{align*}
where $f_L$ and $f_R$ are  the left and right limits of $f_{h,0}$ with $y=y_j$ 
at   $x=x_{i+\frac{1}{2}}$.
Using those definitions and Theorem \ref{thm:NT}, the macroscopic quantities   $ n ^n_{i+\frac{1}{2},j}, T^n_{i+\frac{1}{2},j}$ and {$U_{\alpha,i+\frac{1}{2},j}^n$} can be obtained, and then one gets the J$\ddot{\text{u}}$ttner distribution function $g_0$ at $(x_{i+\frac{1}{2}},y_j,t_n)$.
Similar to Section \ref{sec:equi2DEuler}, we reconstruct a cell-vertex based linear polynomial
and do the first-order Taylor series expansion of $g$ at the cell interface $(x_{i+\frac{1}{2}},y_j)$, see \eqref{eq:gh2D}. However, it is different from the Euler case that
 $A_0$ is obtained by
\[M^0_0\vec{A}_0=\vec{W}^t_0,\]
where $\vec{W}^t_0$ is calculated by using the second-order extrapolation \eqref{eq:extra}.
After those, substituting $f_{h,0}$ and $g_h$   into \eqref{eq:2DAWsoluapprox} gets $\hat{f}(x_{i+\frac{1}{2}},y_j,t)$. The distribution $\hat{f}(x_{i},y_{j+\frac{1}{2}},t)$
can be similarly obtained.

\subsubsection{Derivation of  source terms $\vec{S}_{{1},i,j}$ and $\vec{S}_{{2},i,j}$}
The rest is to calculate $\hat{f}(x_i,y_j,t)$ and $g(x_i,y_j,t)$ for the source terms $\vec{S}_{{1},i,j}$ and $\vec{S}_{{2},i,j}$.
The procedure is the same as the above except for taking the first-order Taylor series expansion at the cell-center $(x_i,y_j)$.   
To be more specific, $g$ and $f_0$ in the analytical solution  \eqref{eq:2DAWsoluapprox} of 2D Anderson-Witting model are replaced  with 
\begin{equation}
\label{eq:ghsource}
g_{h}(x,y,t,\vec{p})=g_{0}(1+a_{0}(x-x_{i})+b_0(y-y_j)+A_{0}(t-t_n)),
\end{equation}
and
\begin{equation}
  f_{h,0}({x,y},\vec{p})=
  g_{0}\left(1-\frac{\tau}{U_{\alpha,0}p^{\alpha}}(A_0^{ce}p^0+a_0^{ce}p^1+b_0^{ce}p^2+c_0^{ce}p^3)+a_{0}\tilde{x}+b_{0}\tilde{y}\right),
\end{equation}
where $(a_0,b_0,A_0)$ are {the Taylor expansion coefficients} at $(x_{i},y_j,t_n)$  calculated by the same procedure as that for $\hat{f}(x_{i+\frac{1}{2}},y_j,t)$,
 $\tilde{x}=x-x_{i}$, $\tilde{y}=y-y_{j}$, $g_{0}$ denotes the J$\ddot{\text{u}}$ttner distribution at $(x_i,y_j, t_n)$, $a_{0}^{ce}, b_{0}^{ce}, c_{0}^{ce}$ and  $A_{0}^{ce}$ are the Chapman-Enskog expansion coefficients at $(x_i,y_j,t_n)$.
 It is worth noting that since $f_{h,0}$ is continuous at  $(x_i, y_j)$,
 there is no need to consider whether the left or right states should be taken here.
 The subroutine for the coefficients  in \eqref{eq:fh0NS} can be used to get those in $f_{h,0}({x,y},\vec{p})$.

In order to define the equilibrium state $g(x_i,y_j,t)$ in the source term, firstly we need to figure out the corresponding macroscopic quantities such as $N^{\alpha}$ and $T^{\alpha\beta}$ which can be obtained by taking the moments of $\hat{f}(x_i,y_j,t)$. Using the Theorem \ref{thm:NT}, the macroscopic quantities such as $ n , T$ and ${\vec{u}}$ can be obtained. Thus the J$\ddot{\text{u}}$ttner distribution function at cell center $(x_{i},y_j)$ is derived according to the definition.

Until now, all distributions are derived and the second-order accurate genuine BGK scheme \eqref{eq:moment1} is
developed for the 2D ultra-relativistic  Navier-Stokes equations.

\section{Numerical experiments}
\label{sec:test}
This section will solve several 1D and 2D problems on the ultra-relativistic fluid flow
to demonstrate the accuracy and effectiveness of the present genuine BGK schemes, which will be compared to  the second-order accurate BGK-type  and   KFVS schemes \cite{abdel2013,qamar2003}.
The collision time $\tau$  is taken as
\[
\tau=\tau_m+C_{2}\Delta t^{\alpha}_{n}\frac{|P_{L}-P_{R}|}{P_{L}+P_{R}},
\]
with $\tau_m = \frac{5\mu}{4p}$ for the viscous flow and $\tau_m=C_{1}\Delta t^{\alpha}_n$ for the inviscid flow, $C_1$, $C_2$ and
$\alpha$ are three constants, $P_L, P_R$ are the left and right limits of the pressure at the cell interface, respectively. Unless specifically stated, this section takes $C_1=0.001, C_2=1.5$ and $\alpha=1$,
the time step-size $\Delta t_n$ is determined by the CFL condition \eqref{TimeStep1D} or \eqref{TimeStep2D} with the CFL number of 0.4, and the characteristic variables are reconstructed with the van Leer limiter.

\subsection{1D Euler case}
\begin{example}[Accuracy test]\label{ex:accurary1D}\rm To check the accuracy of our BGK method, we first solve a smooth problem which describes a sine wave propagating  periodically in the domain $\Omega=[0,1]$. The initial conditions are taken as
 \[ n (x,0)= 1+0.5\sin(2\pi x), \quad u_1(x,0)=0.2, \quad p(x,0)=1,\]
 and corresponding exact solutions are given by
 \[ n (x,t)= 1+0.5\sin(2\pi (x-0.2t)), \quad u_1(x,t)=0.2, \quad p(x,t)=1.\]
The computational domain $\Omega$ is divided into $N$ uniform cells and the periodic boundary conditions are specified at $x=0,1$.
 \scriptsize
 \begin{table}[H]
   \setlength{\abovecaptionskip}{0.cm}
   \setlength{\belowcaptionskip}{-0.cm}
   \caption{Example \ref{ex:accurary1D}: Numerical errors of $n$ in $l^1, l^2$-norms and convergence rates at $t = 0.2$ with or without limiter.}\label{accuracy}
   \begin{center}
     \begin{tabular}{*{9}{c}}
       \toprule
       \multirow{2}*{$N$} &\multicolumn{4}{c}{With limiter} &\multicolumn{4}{c}{Without limiter}\\
       \cmidrule(lr){2-5}\cmidrule(lr){6-9}
       & $l^1$ error  & $l^1$ order  &  $l^2$ error  &  $l^2$ order & $l^1$ error  & $l^1$ order  &  $l^2$ error  &  $l^2$ order\\
       \midrule
       25   & 1.6793e-03 &  --       &2.5667e-03 & --     &6.0337e-04 &-       &6.7007e-04  &-  \\
       50   & 4.9516e-04 &  1.7619 &8.2151e-04 &1.6436  &1.5275e-04   &1.9819  &1.6965e-04  &1.9818\\
       100  & 1.3012e-04 &  1.9281 &2.6823e-04 &1.6148  &3.8305e-05   &1.9956  &4.2559e-05  &1.9950\\
       200  & 3.4917e-05 &  1.8978 &8.5622e-05 &1.6474  &9.5628e-06   &2.0020  &1.0621e-05  &2.0025\\
       400  & 8.2820e-06 &  2.0759 &2.6141e-05 &1.7117  &2.3904e-06   &2.0002  &2.6550e-06  &2.0001\\
       \bottomrule
     \end{tabular}
   \end{center}
 \end{table}
\end{example}
Table \ref{accuracy} gives the $l^1$- and $l^2$-errors at $t=0.2$ and corresponding convergence rates for the BGK scheme
with $\alpha = 2$ and $C_1=C_2=1$.
The results show that a second-order rate of convergence can be obtained for our BGK scheme although
the van Leer limiter loses  slight accuracy.

\begin{example}[Riemann problem I]\rm \label{ex:1DRP1} 
 This is a Riemann problem with the following initial data
 \begin{equation}
   \label{exeq:1DRP1}
  (n,u_1,p)(x,0)=
   \begin{cases}
     (1.0,1.0,3.0),& x<0.5,\\
     (1.0,-0.5,2.0),& x>0.5.
   \end{cases}
 \end{equation}
\end{example}
The initial discontinuity will evolve as a left-moving shock wave,  a right-moving contact discontinuity,
and a right-moving shock wave. Fig. \ref{fig:1DRP1}
displays the numerical results at $t=0.5$ and their close-ups obtained by using
our BGK scheme (``{$\circ$}"), the BGK-type scheme (``{$\times$}"),
and the KFVS scheme (``{+}")
with 400 uniform cells in the domain $[0,1]$, where the solid lines denote the exact solutions.
It can be seen that our BGK scheme resolves the contact discontinuity better than the second-order accurate BGK-type and
KFVS schemes,
and they can well capture such wave configuration.

\begin{figure}[htbp]
 \centering
 \subfigure[$ n $]{
 \includegraphics[width=0.3\textwidth]{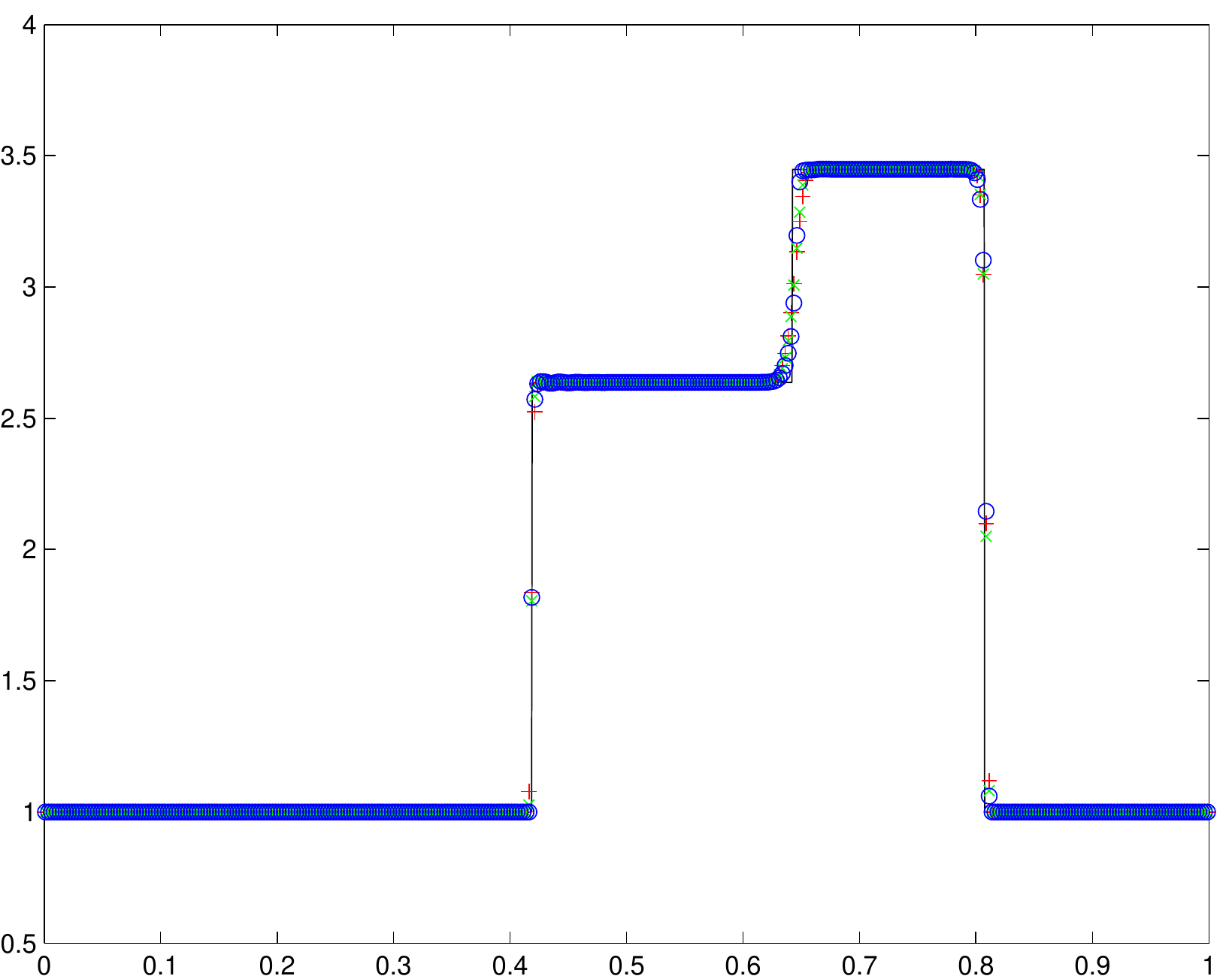}
 }
 \subfigure[$u_1$]{
 \includegraphics[width=0.3\textwidth]{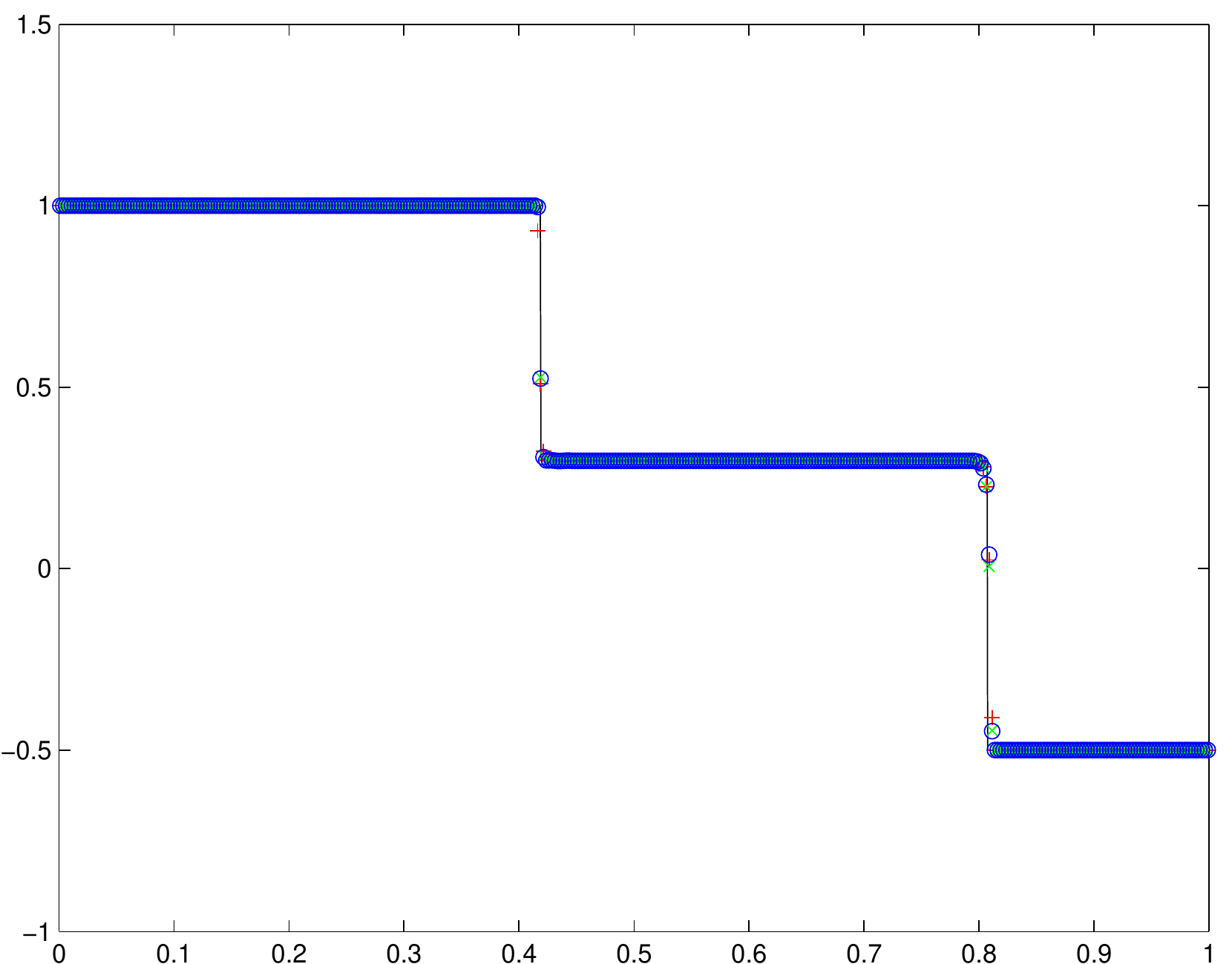}
 }
 \subfigure[$p$]{
 \includegraphics[width=0.3\textwidth]{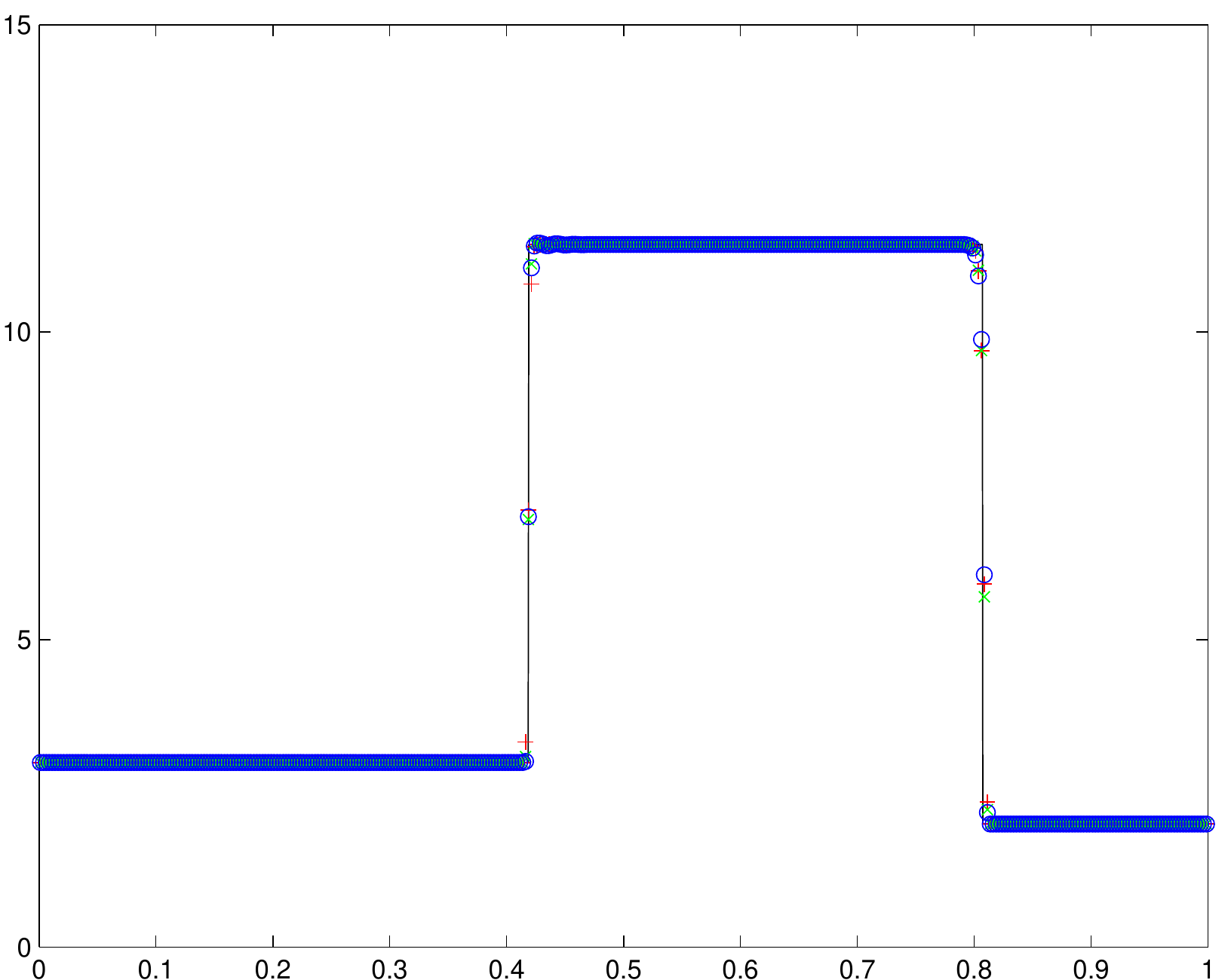}
 }
 \subfigure[close-up of $ n $]{
  \includegraphics[width=0.3\textwidth]{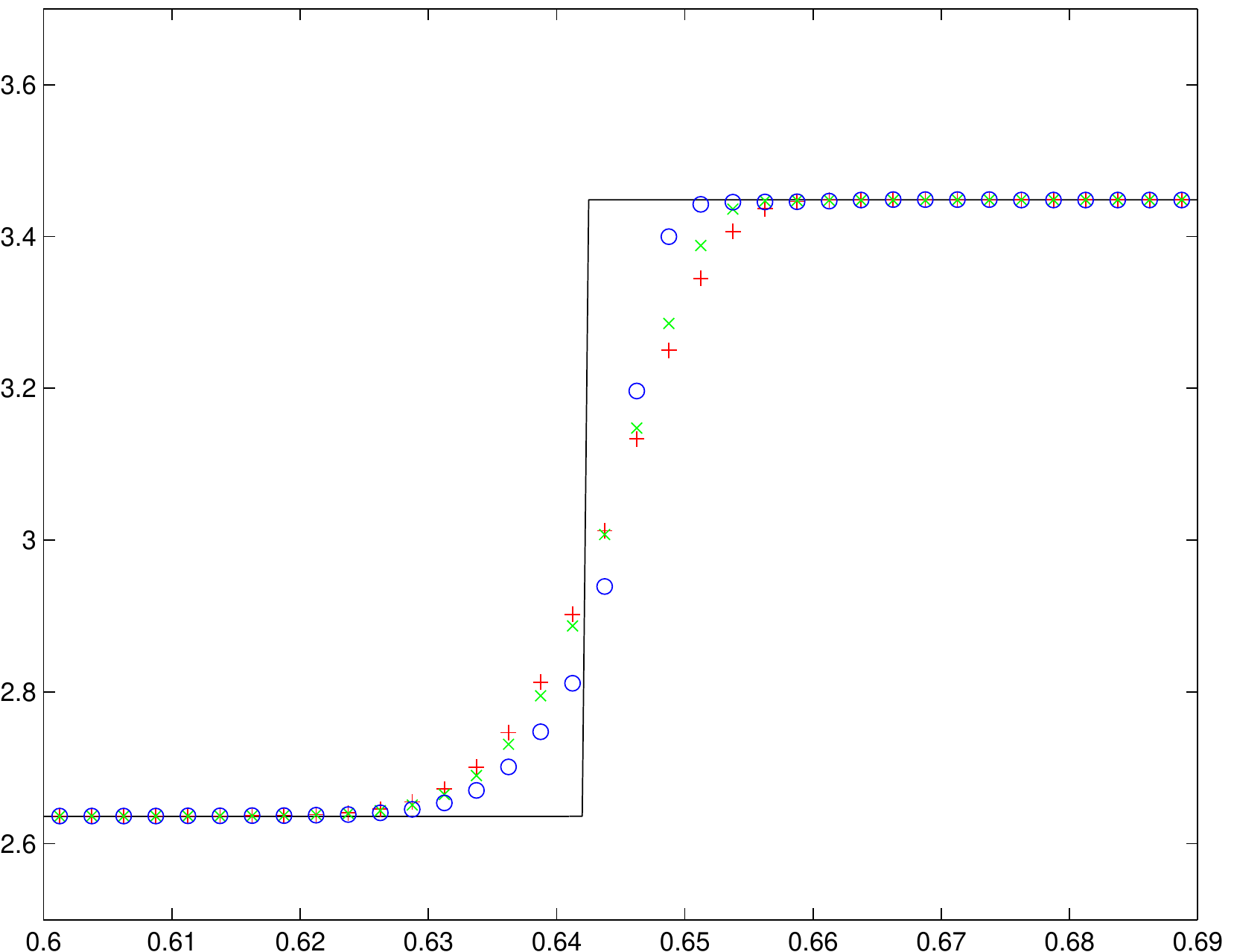}
 }
 \subfigure[close-up of $u_1$]{
   \includegraphics[width=0.3\textwidth]{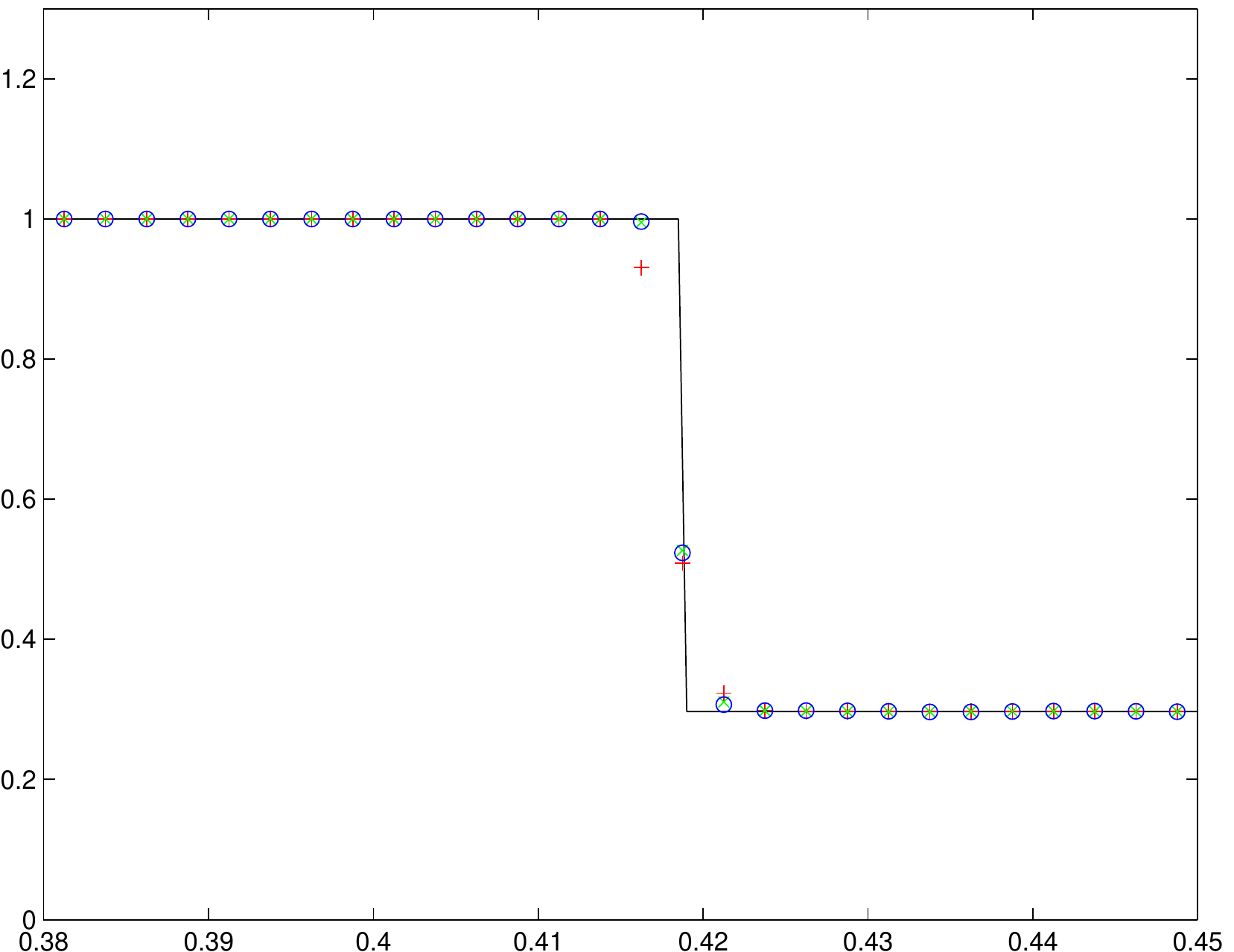}
 }
 \subfigure[close-up of $p$]{
   \includegraphics[width=0.3\textwidth]{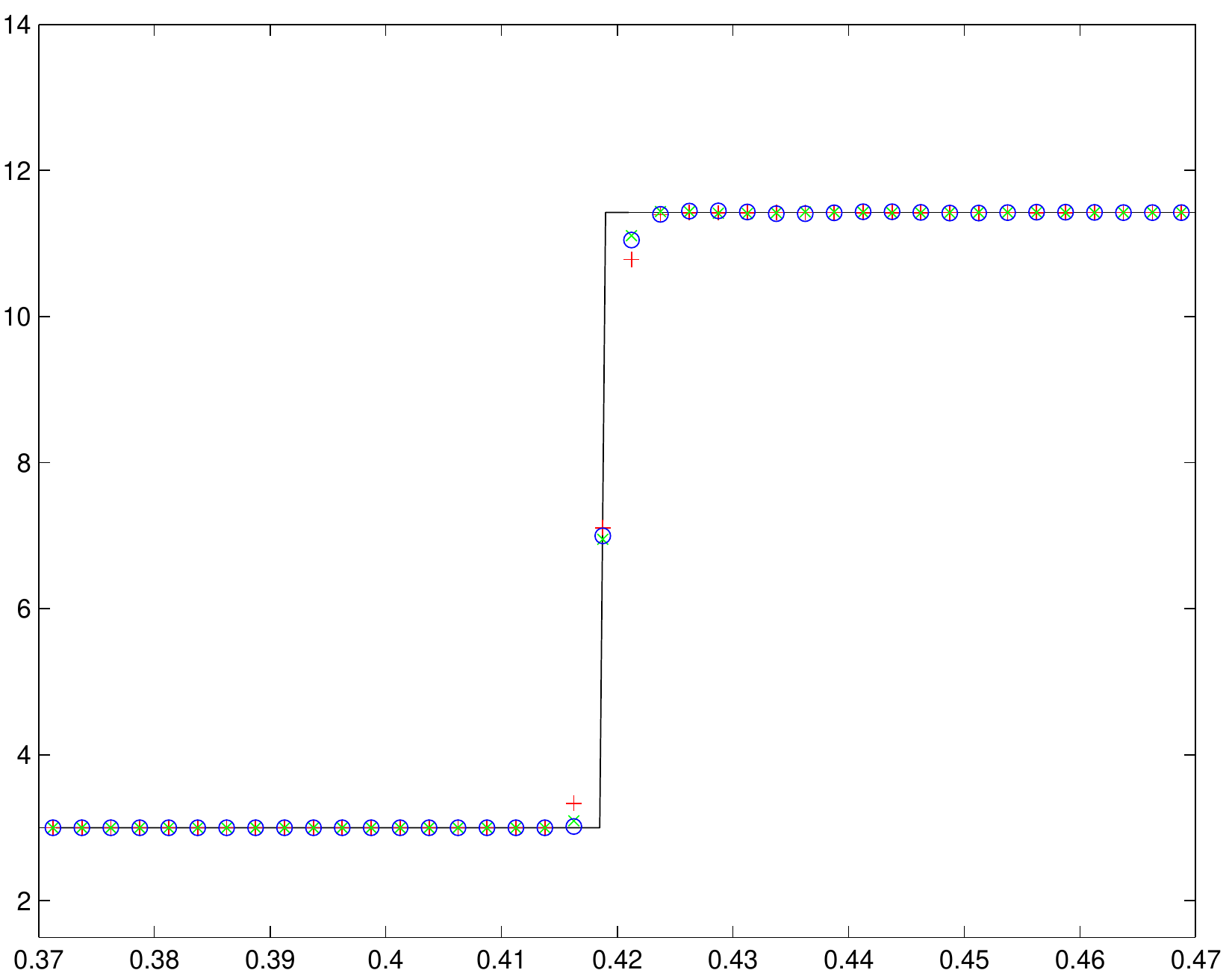}
 }
 \caption{Example \ref{ex:1DRP1}: The solutions {at} $t=0.5$  and their close-ups  obtained by using our BGK scheme (``{$\circ$}"), the BGK-type scheme (``{$\times$}"),
and the KFVS scheme (``{+}") with 400 uniform cells. }
 \label{fig:1DRP1}
\end{figure}

\begin{example}[Riemann problem II]\label{ex:1DRP2}\rm
 The initial conditions of the second Riemann problem are
 \begin{equation}
   \label{exeq:1DRP2}
  (n,u_1,p)(x,0)=
   \begin{cases}
     (5.0,0.0,10.0),& x<0.5,\\
     (1.0,0.0,0.5),& x>0.5.
   \end{cases}
 \end{equation}
\end{example}
\begin{figure}[htbp]
\centering
\subfigure[$ n $]{
 \includegraphics[width=0.3\textwidth]{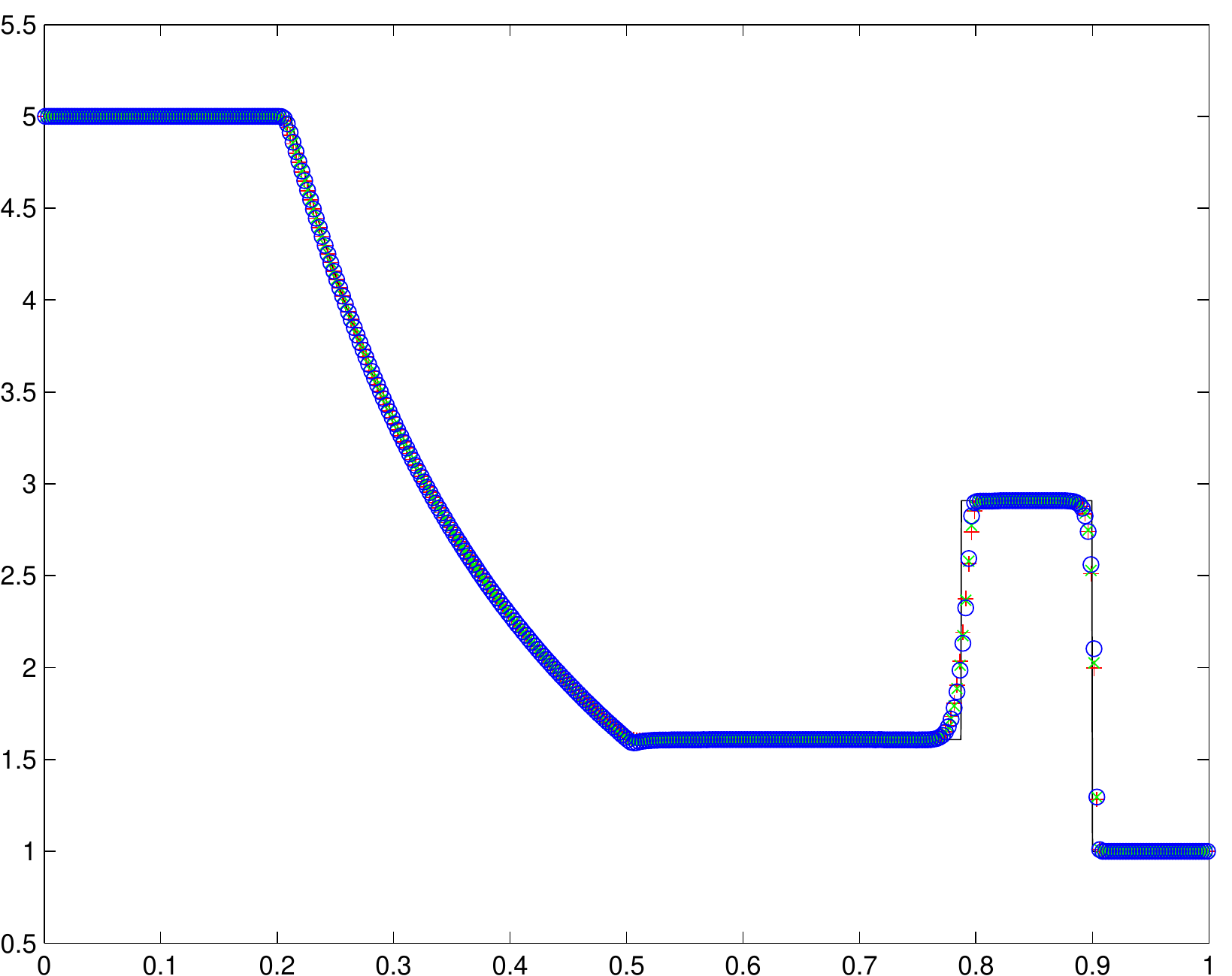}
}
\subfigure[$u_1$]{
 \includegraphics[width=0.3\textwidth]{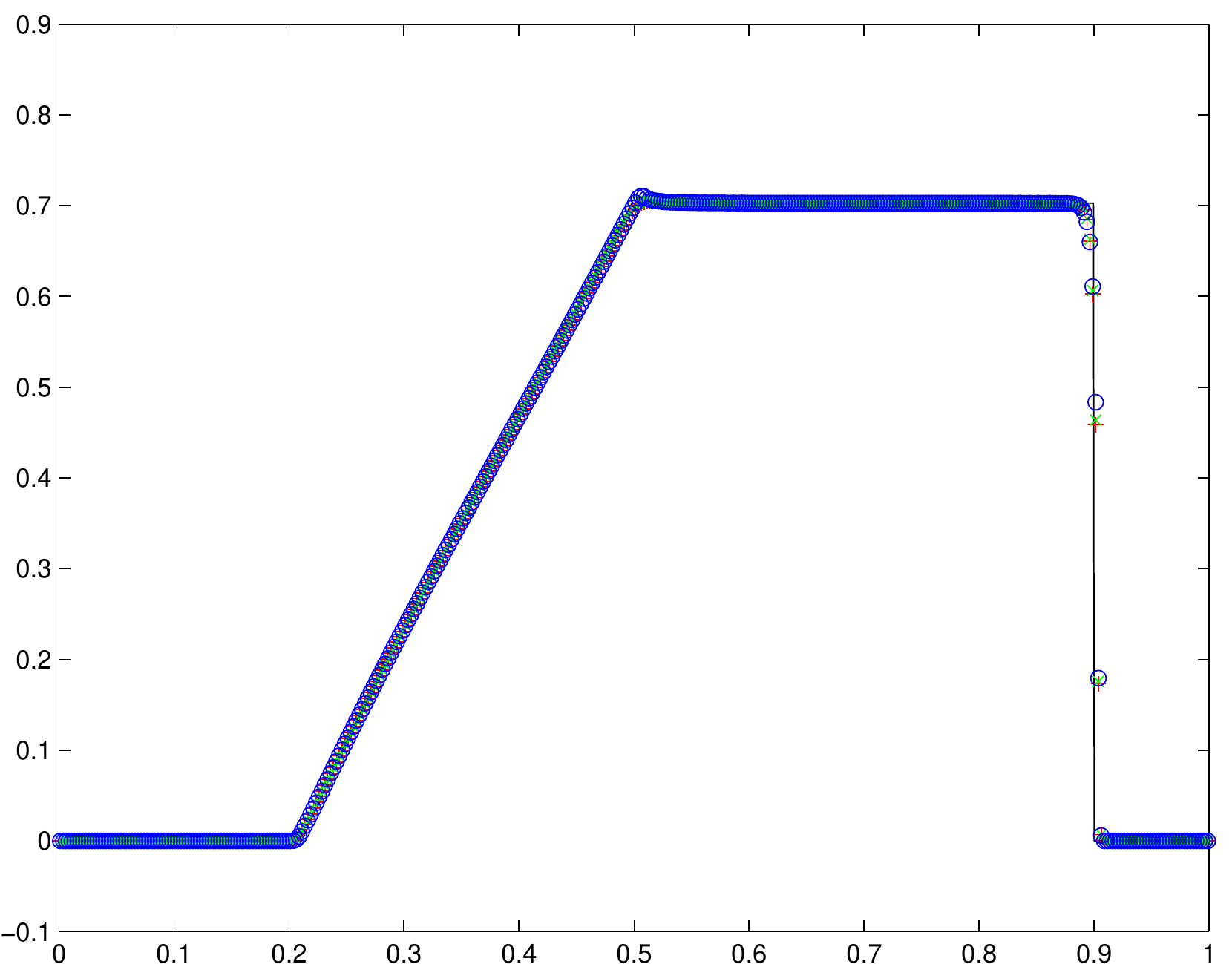}
}
\subfigure[$p$]{
 \includegraphics[width=0.3\textwidth]{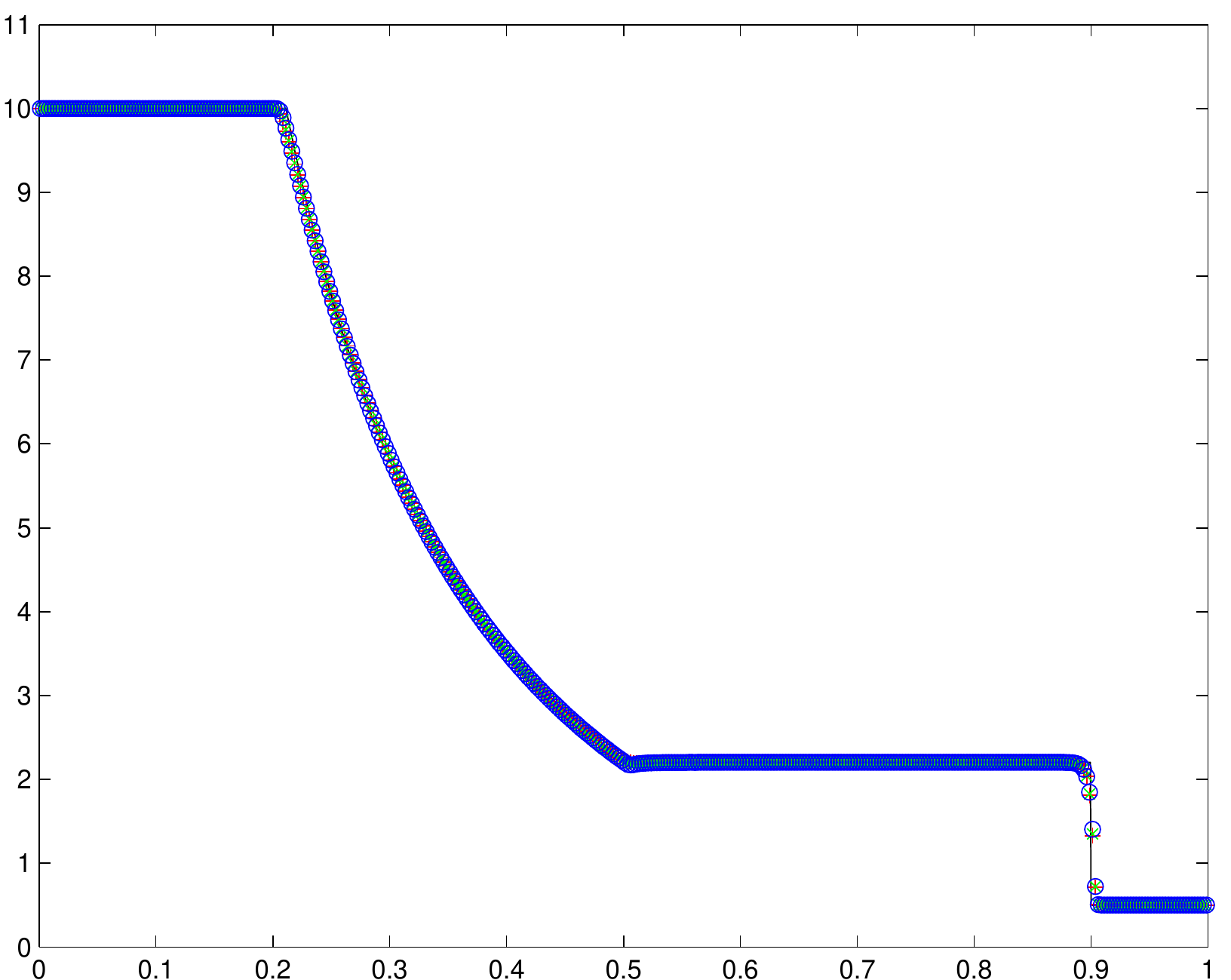}
}
\subfigure[close-up of $ n $]{
 \includegraphics[width=0.3\textwidth]{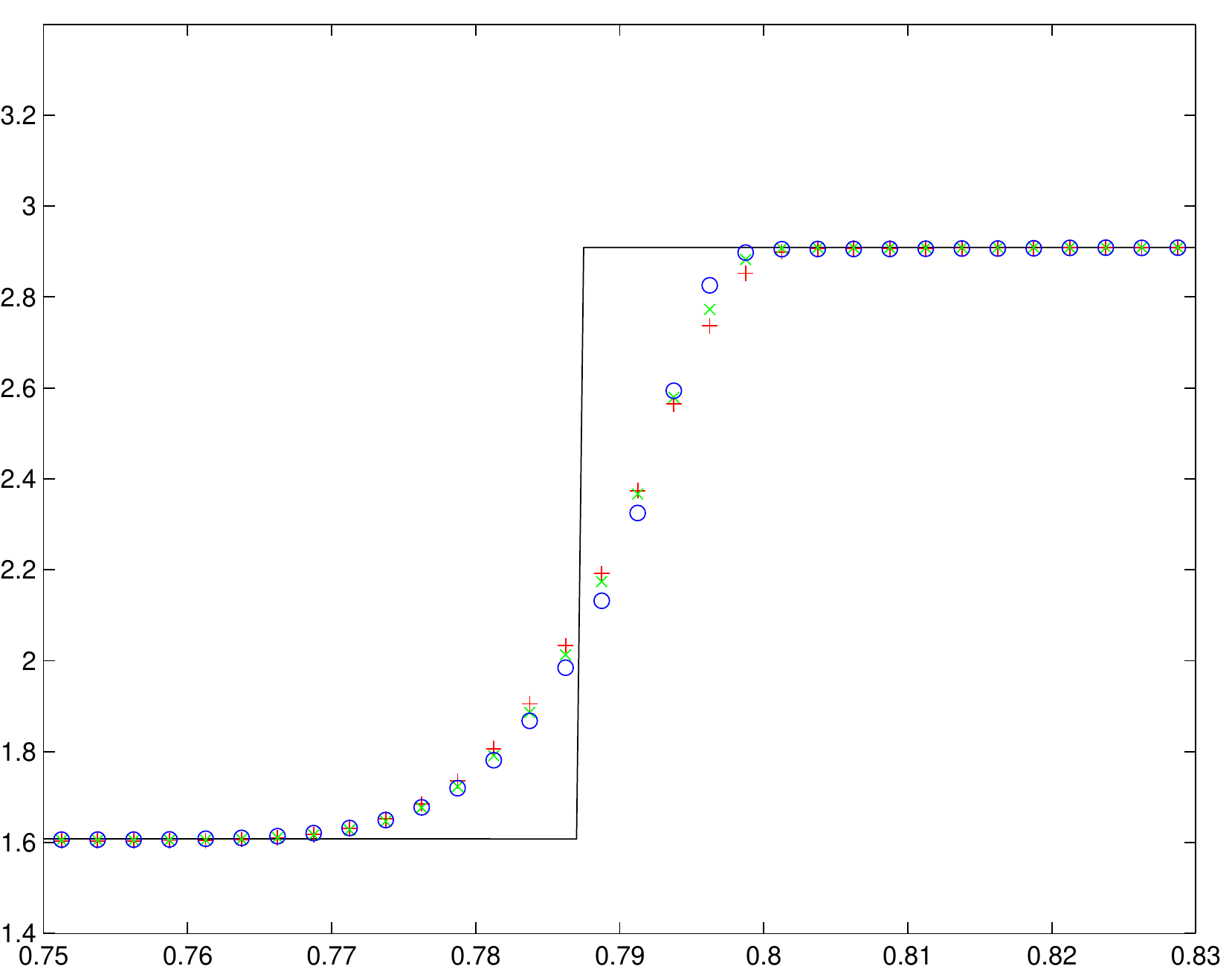}
}
\subfigure[close-up of $u_1$]{
 \includegraphics[width=0.3\textwidth]{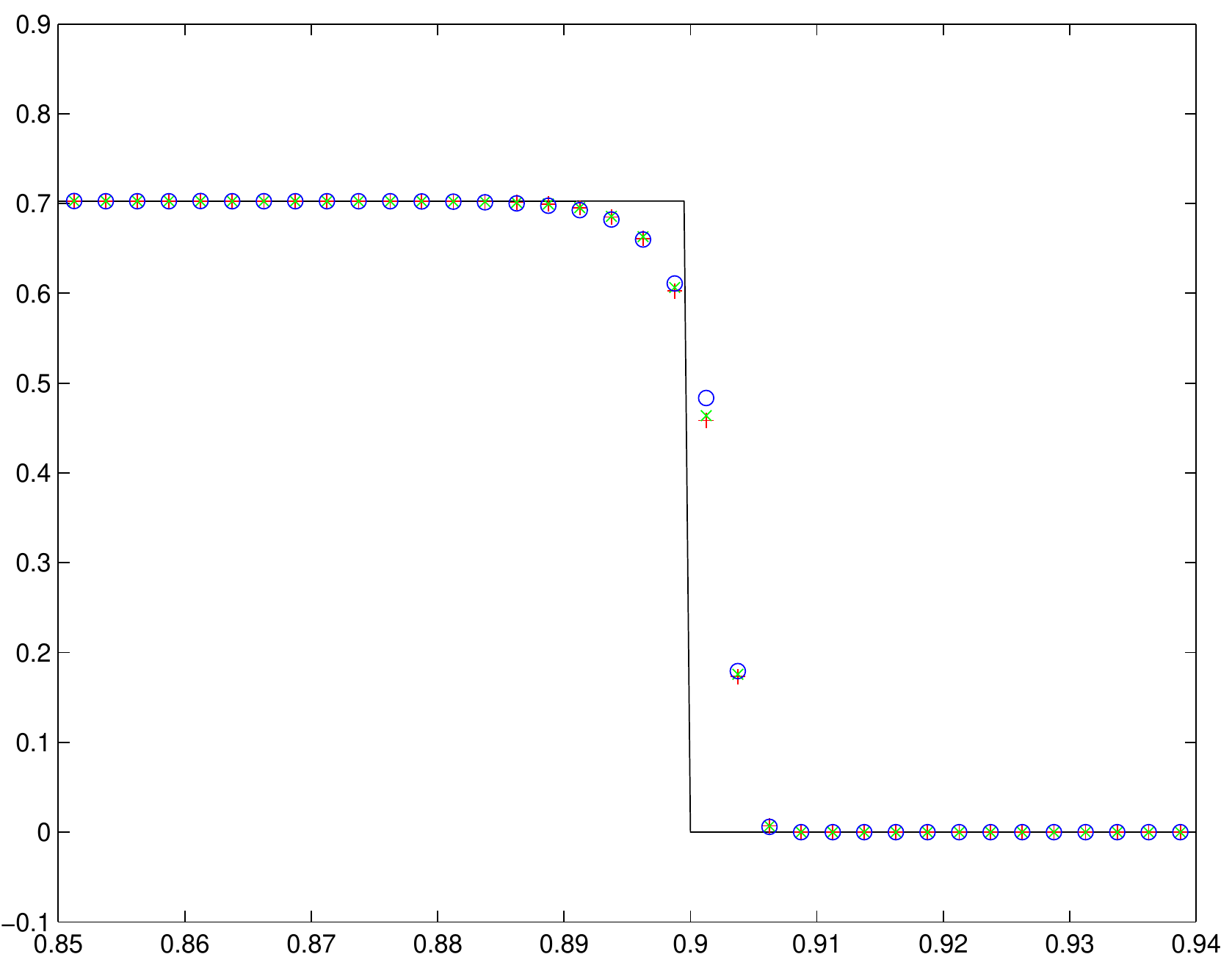}
}
\subfigure[close-up of $p$]{
 \includegraphics[width=0.3\textwidth]{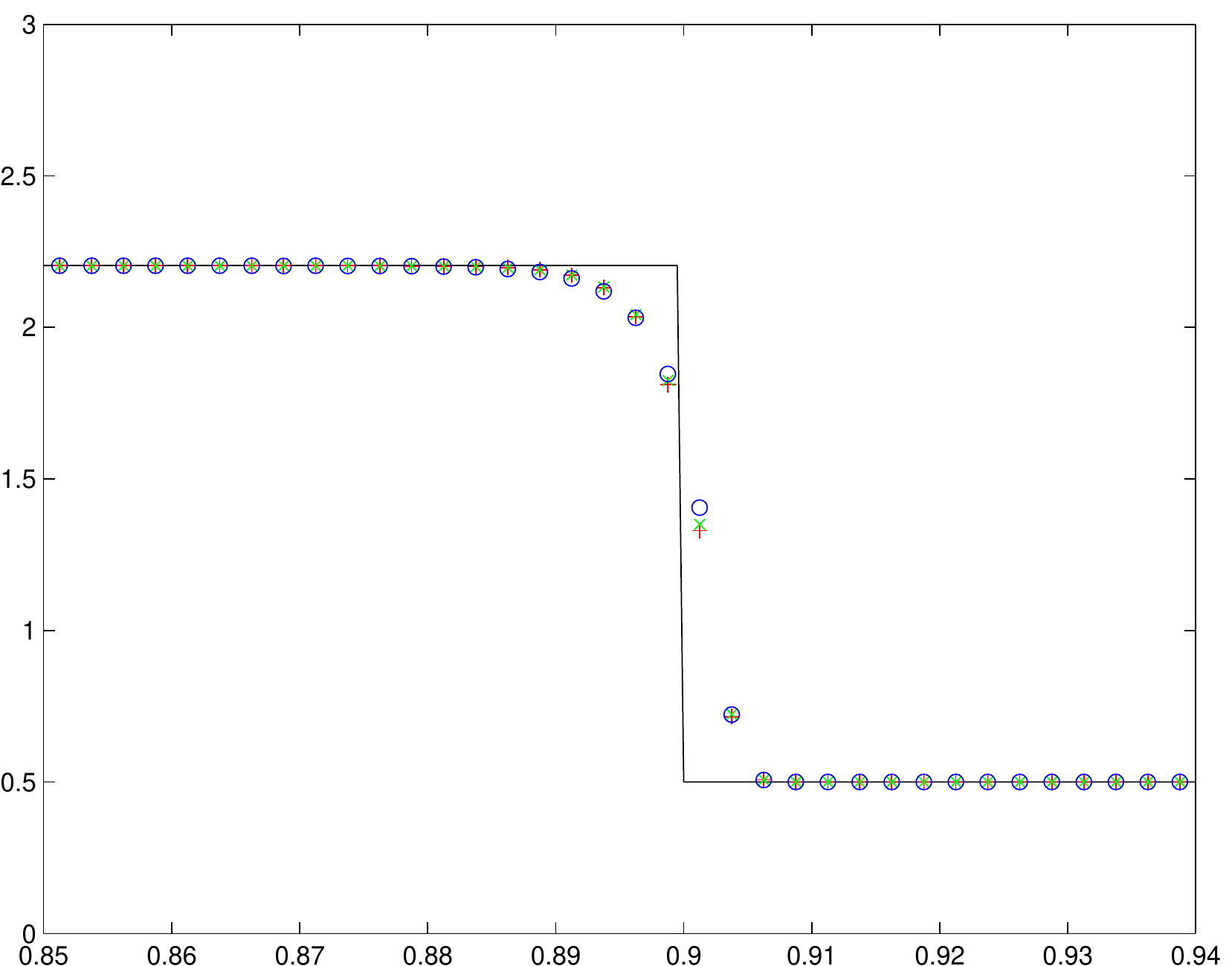}
}
\caption{Example \ref{ex:1DRP2}: The number density $ n $,  velocity $u_1$ and  pressure $p$ and their close-ups at $t=0.5$ obtained by using our BGK scheme (``{$\circ$}"), the BGK-type scheme (``{$\times$}"),
and the KFVS scheme (``{+}") with 400 uniform cells.}
\label{fig:1DRP2}
\end{figure}

Fig. \ref{fig:1DRP2}
shows the numerical solutions at $t=0.5$ obtained by using our BGK scheme (``{$\circ$}"), the BGK-type scheme (``{$\times$}"),
and the KFVS scheme (``{+}") with 400 uniform cells within the domain $[0,1]$, where the solid line denotes the exact solution. It is seen that the solutions
consist of a left-moving rarefaction wave, a contact discontinuity, and a right-moving shock wave,
the computed solutions well accord with the exact solutions, and the
rarefaction and shock waves are well resolved.
Moreover, our BGK scheme exhibits better resolution of the contact discontinuity
than the BGK-type and KFVS schemes.

\begin{example}[Riemann problem III]\label{ex:1DRP3}\rm
 The initial data are
 \begin{equation}
   \label{eqex:1DRP3}
(n,u_1,p)(x,0)=
   \begin{cases}
     (1.0,-0.5,2.0),& x<0.5,\\
     (1.0,0.5,2.0),& x>0.5.
   \end{cases}
 \end{equation}
 The initial discontinuity will evolve as a left-moving rarefaction wave, a   stationary contact discontinuity, and a right-moving rarefaction wave.

Fig. \ref{fig:1DRP3} plots the numerical results at $t=0.5$ obtained by using our BGK scheme (``{$\circ$}"), the BGK-type scheme (``{$\times$}"),
and the KFVS scheme (``{+}") with 400 uniform cells in the domain $[0,1]$,
where the solid line denotes the exact solution.
It is seen that there is a undershoot near the contact discontinuity in the number density
 which usually happens in the non-relativistic cases.
\end{example}

\begin{figure}[htbp]
 \centering
 \subfigure[$ n $]{
   \includegraphics[width=0.3\textwidth]{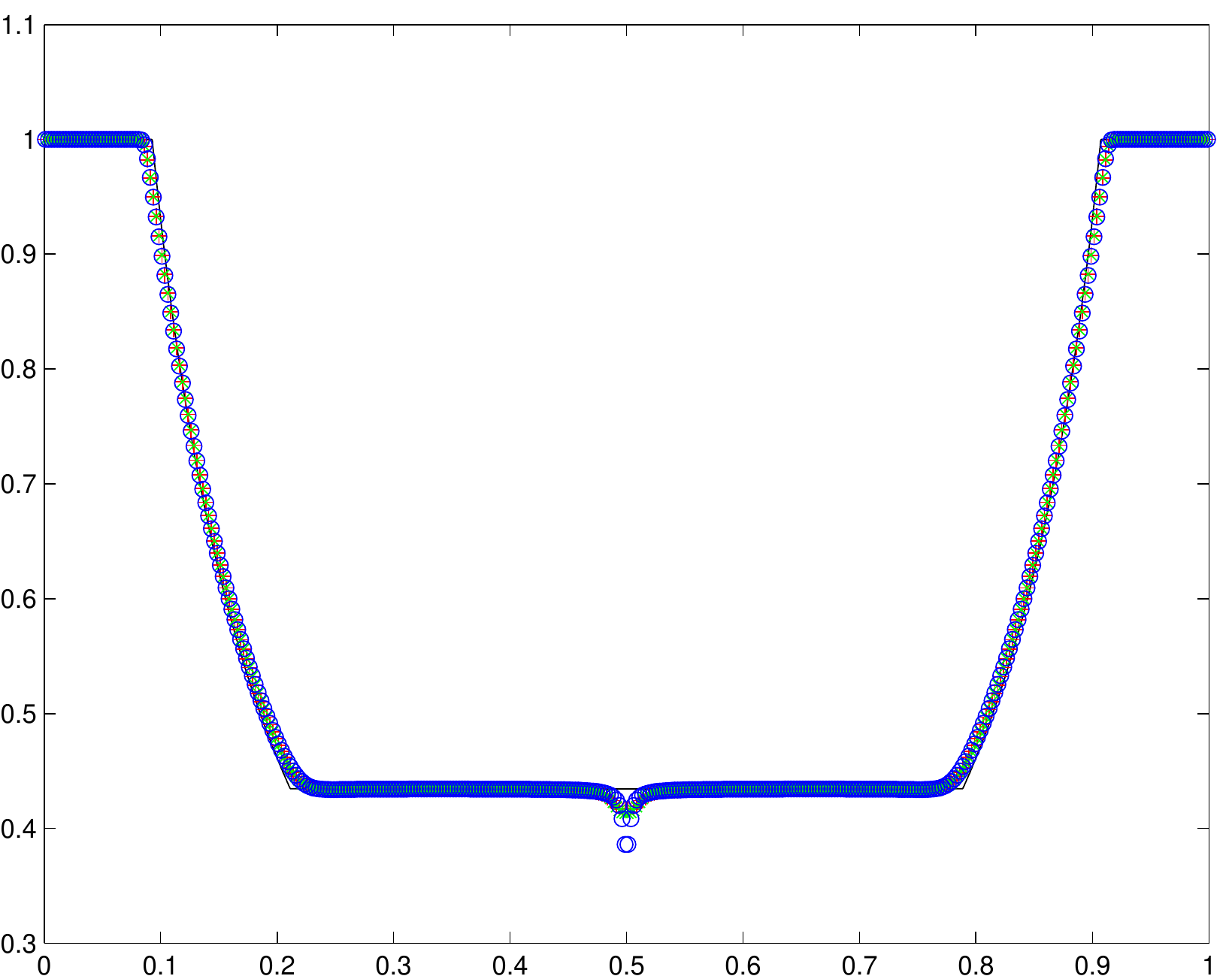}
 }
 \subfigure[$u_1$]{
   \includegraphics[width=0.3\textwidth]{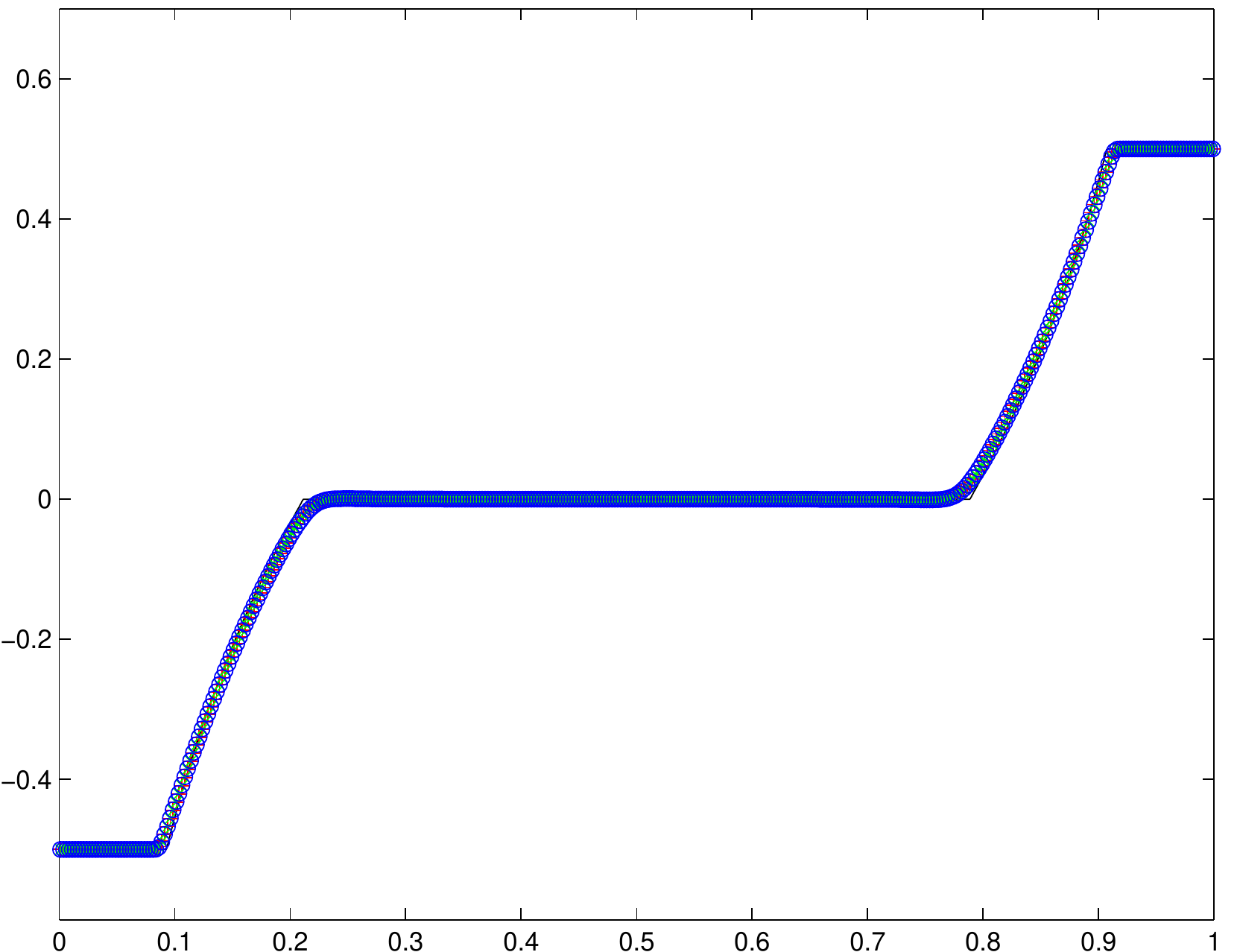}
 }
 \subfigure[$p$]{
   \includegraphics[width=0.3\textwidth]{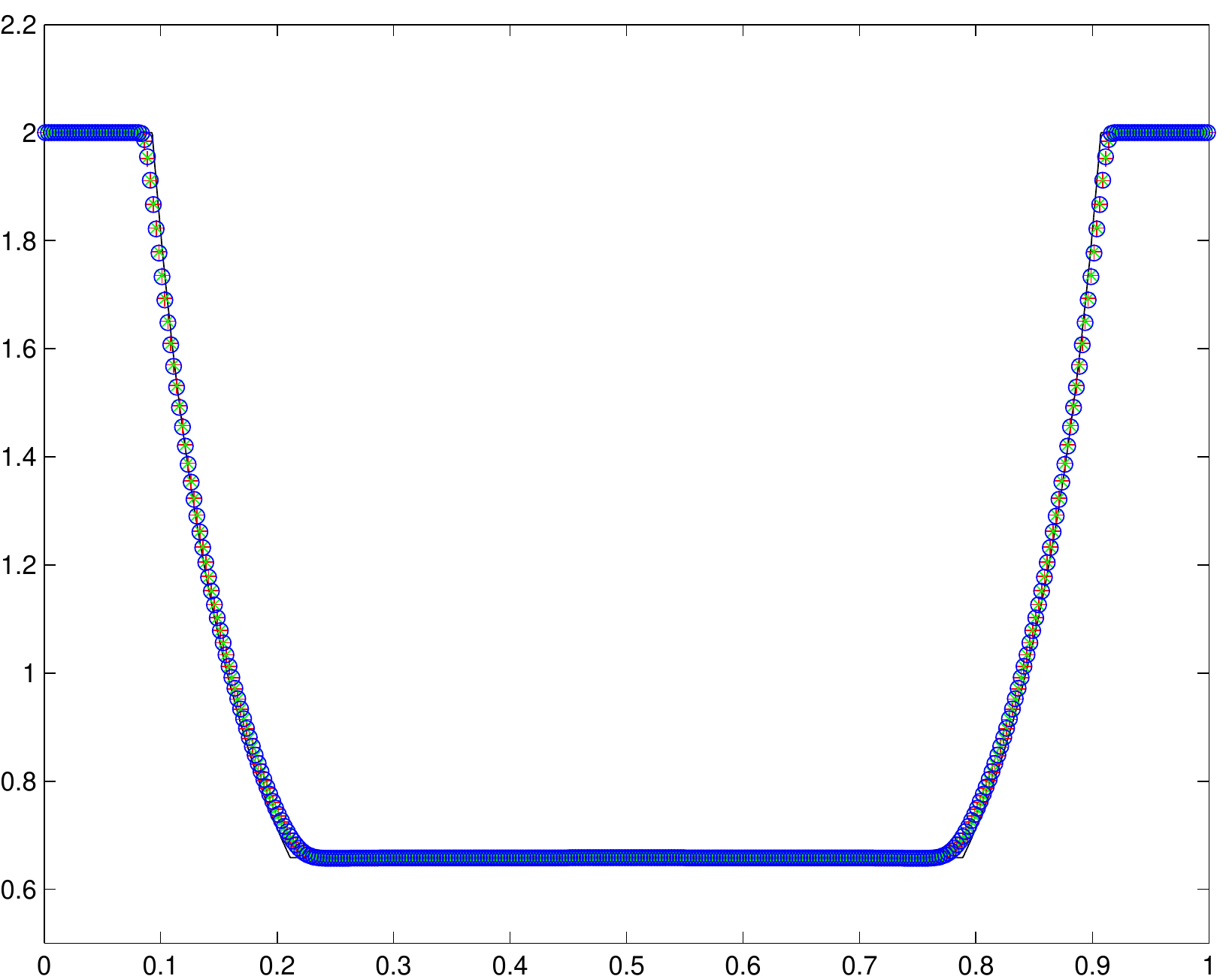}
 }
 \subfigure[close-up of $ n $]{
   \includegraphics[width=0.3\textwidth]{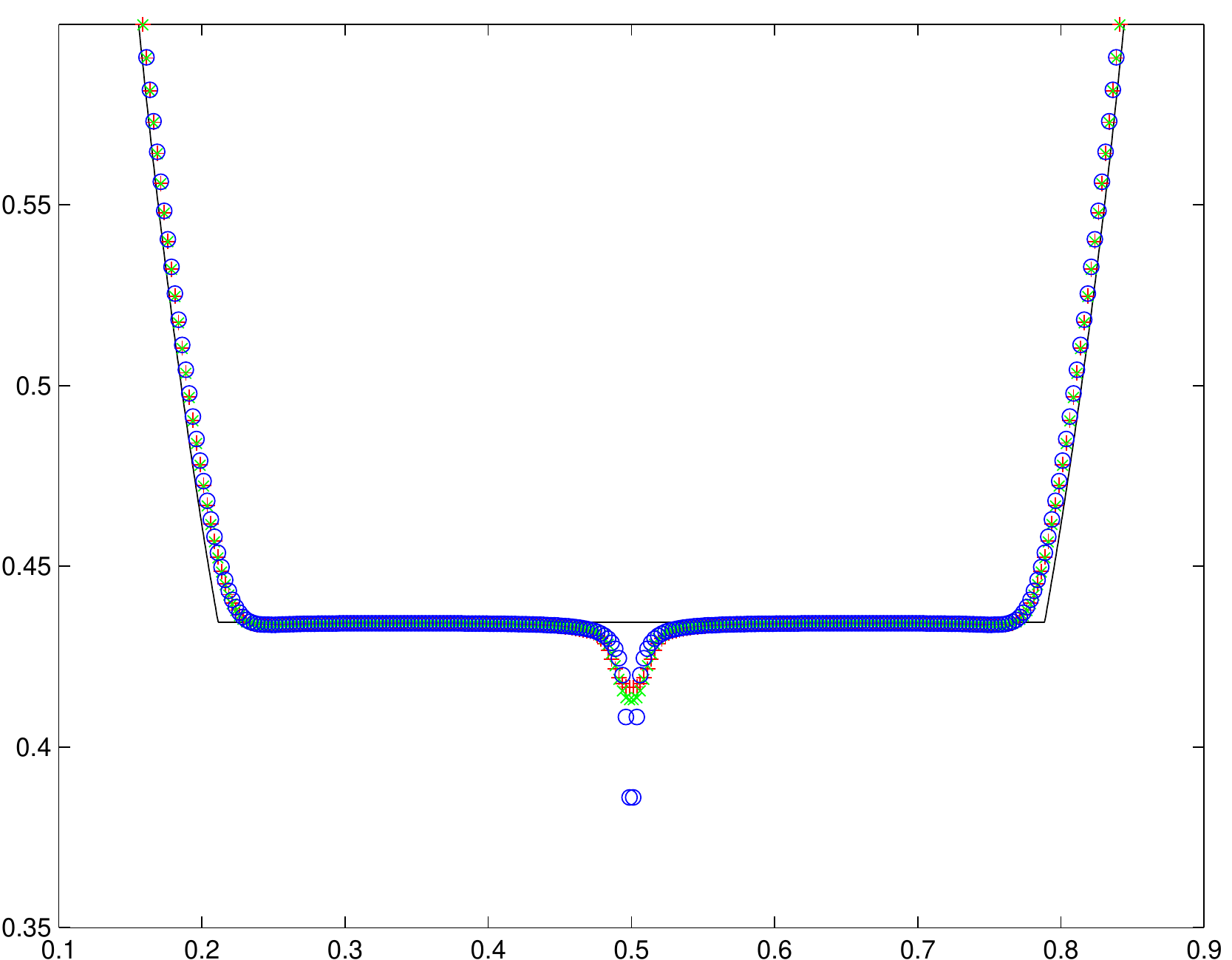}
 }
 \subfigure[close-up of $u_1$]{
   \centering
   \includegraphics[width=0.3\textwidth]{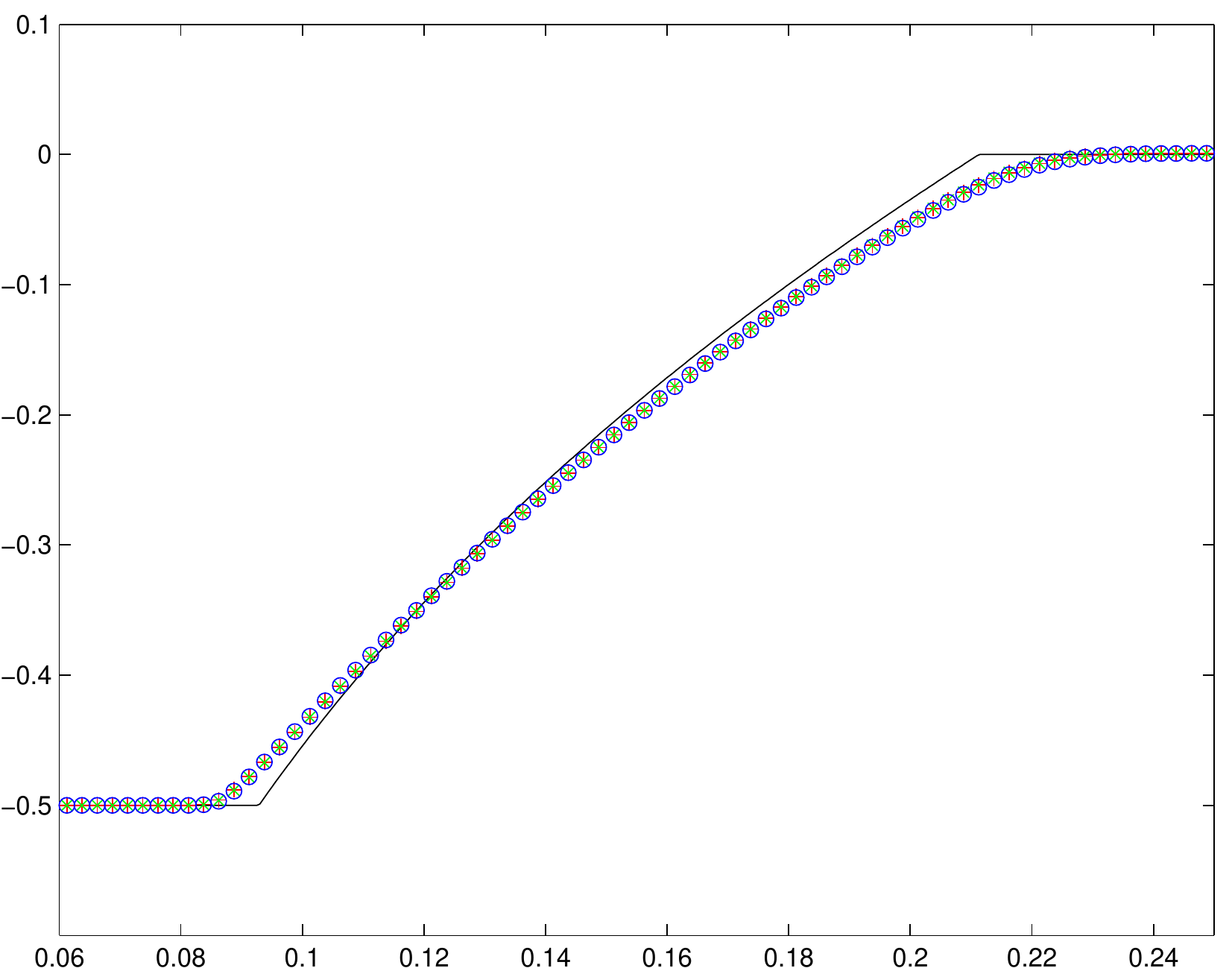}
 }
 \subfigure[close-up of $p$]{
   \includegraphics[width=0.3\textwidth]{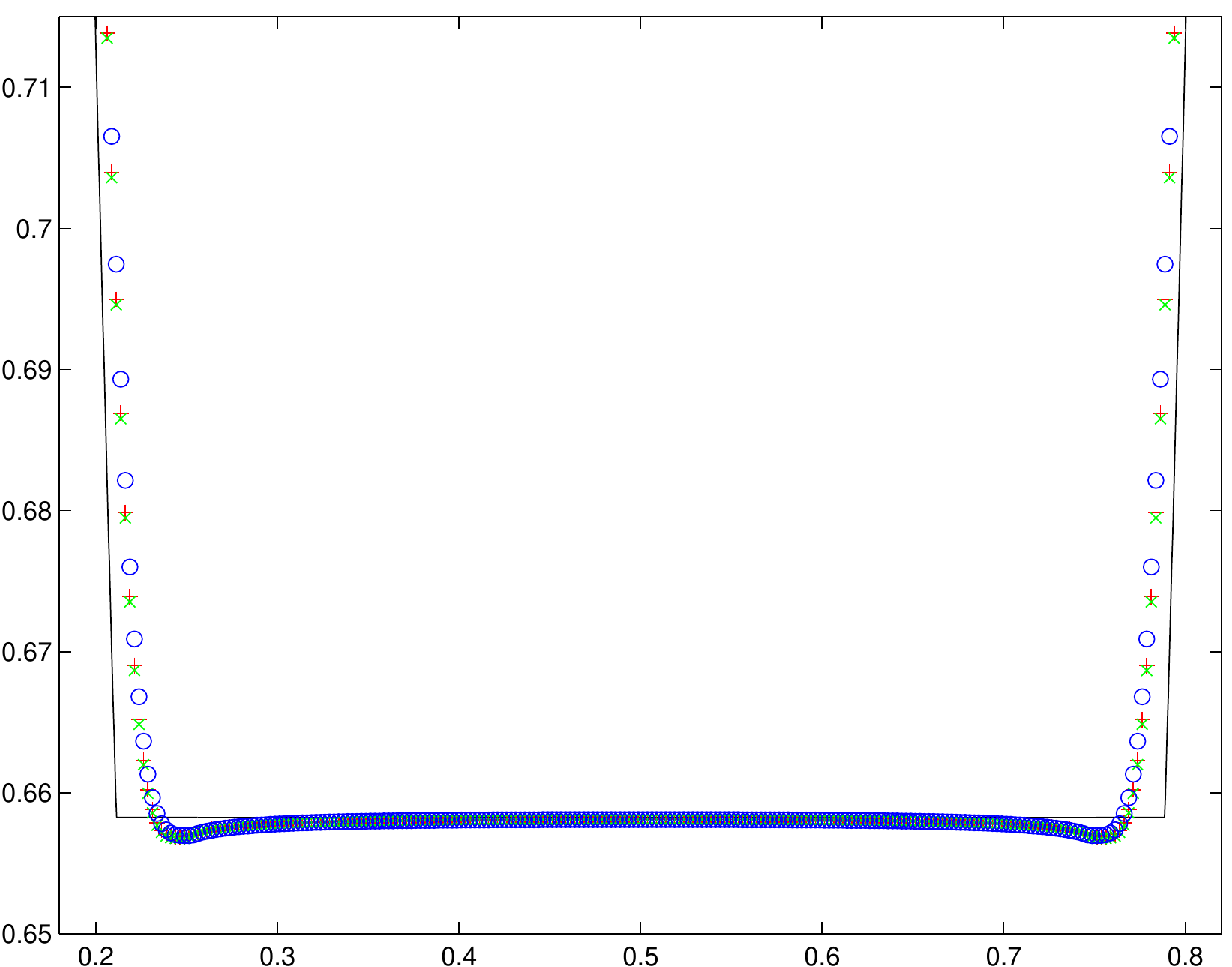}
 }
 \caption{Example \ref{ex:1DRP3}: The solutions and their close-ups at $t=0.5$ obtained by using our BGK scheme (``{$\circ$}"), the BGK-type scheme (``{$\times$}"),
and the KFVS scheme (``{+}") with 400 uniform cells.}
 \label{fig:1DRP3}
\end{figure}

\begin{example}[Perturbed shock tube problem] \label{ex:sinewave}\rm
 The initial data are
 \begin{equation}
 \label{exeq:sinewave}
(n,u_1,p)(x,0)=
 \begin{cases}
   (1.0,0.0,1.0),& x<0.5,\\
   ( n _r,0.0,0.1),& x>0.5,
 \end{cases}
\end{equation}
where $ n _r=0.125 - 0.0875\sin(50(x-0.5))$.
It is a perturbed shock tube problem, which has widely been used to test the ability of the shock-capturing schemes in resolving small-scale flow features in the non-relativistic flow.
\end{example}

\begin{figure}[h]
 \centering
 \subfigure[$ n $]{
   \includegraphics[width=0.4\textwidth]{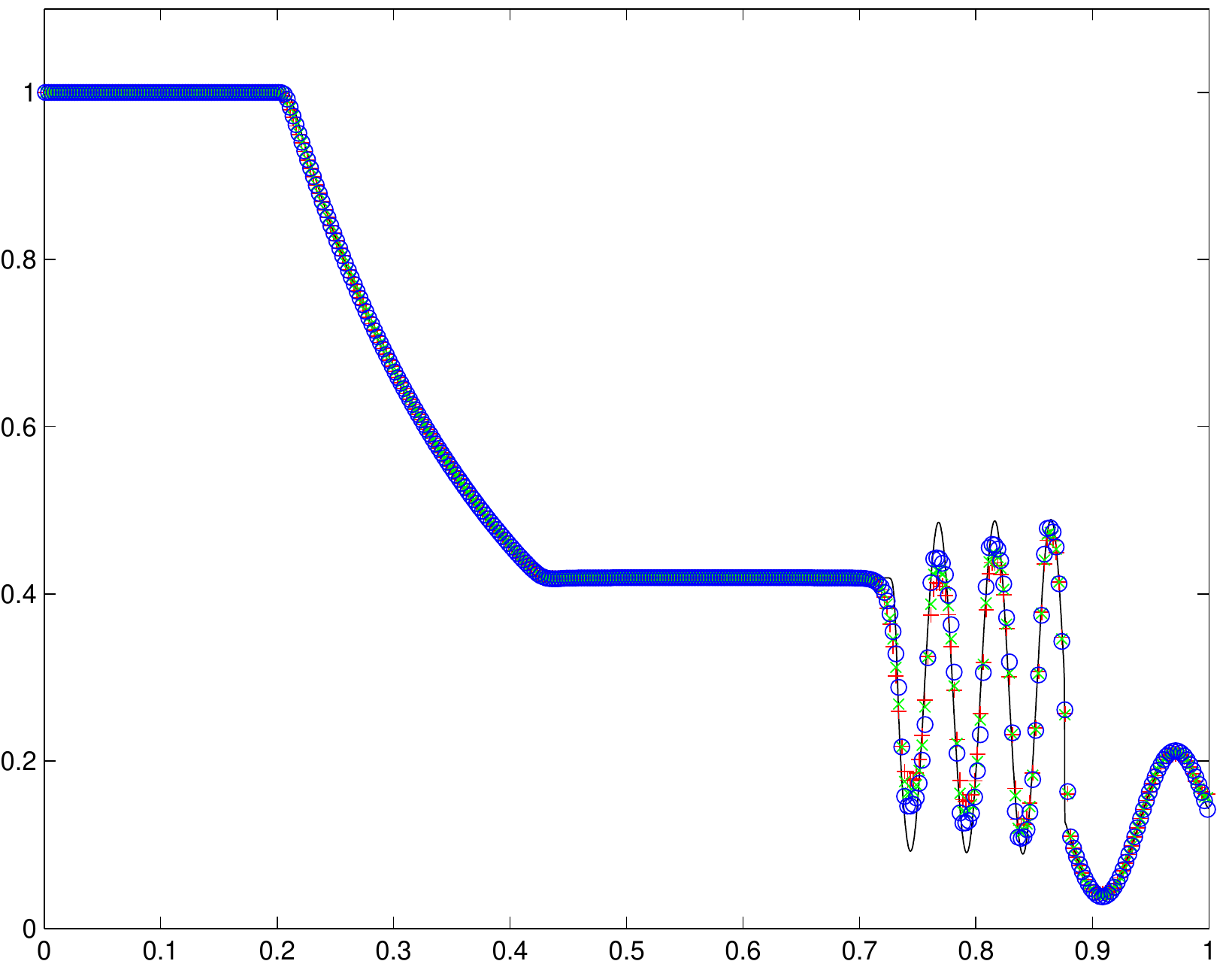}
 }
 \subfigure[close-up of $ n $]{
   \includegraphics[width=0.4\textwidth]{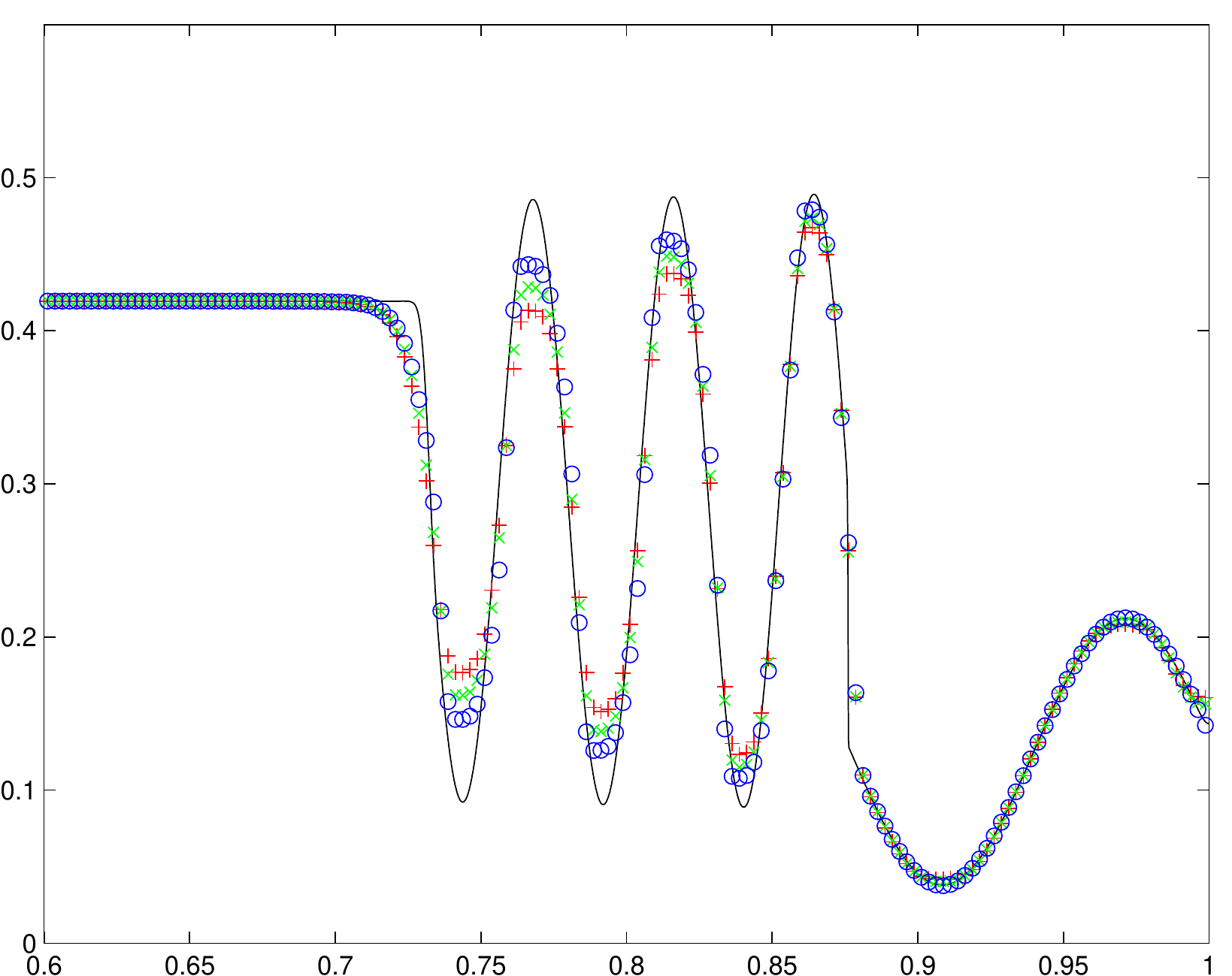}
 }
 \subfigure[$u_1$]{
   \includegraphics[width=0.4\textwidth]{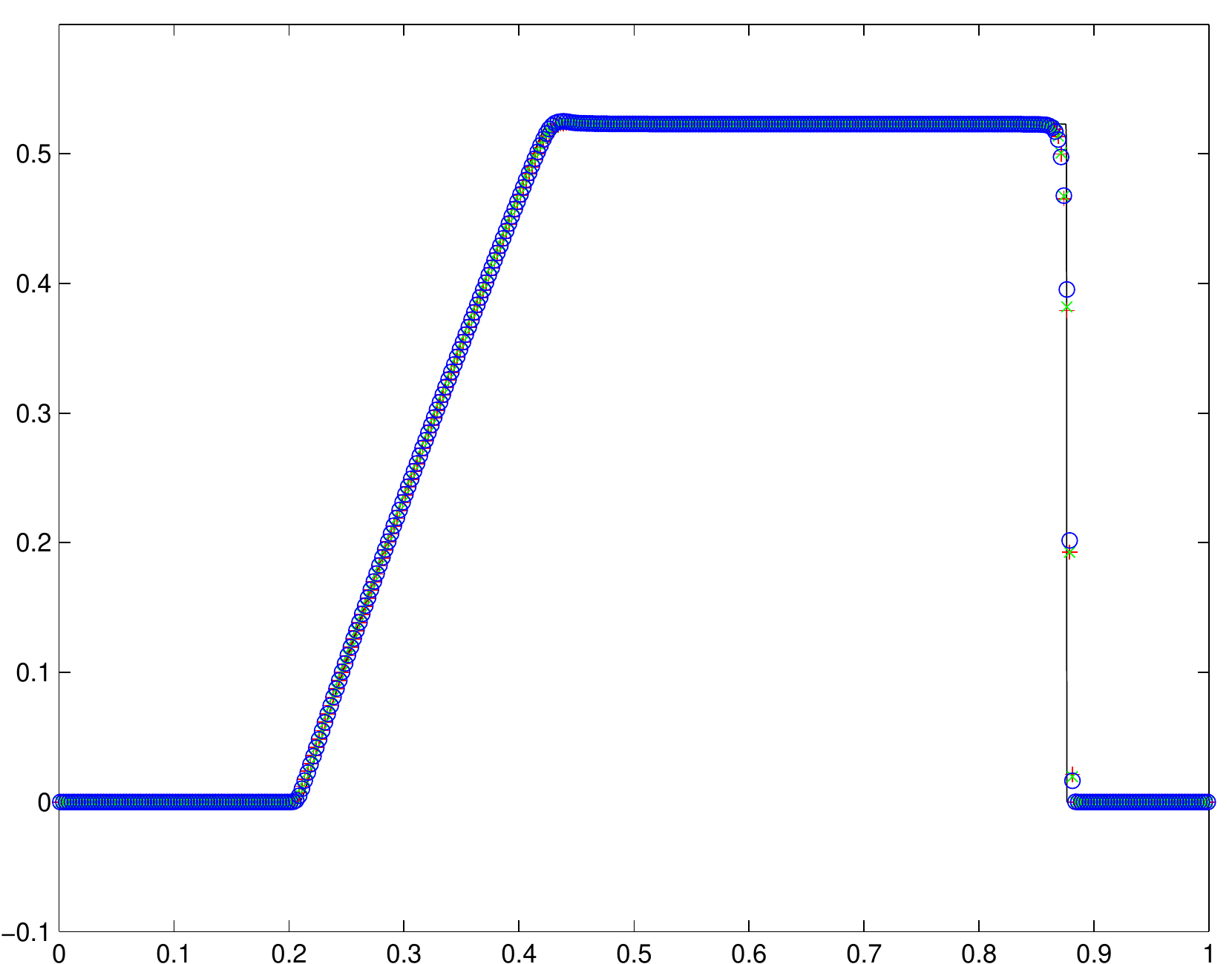}
 }
 \subfigure[$p$]{
   \includegraphics[width=0.4\textwidth]{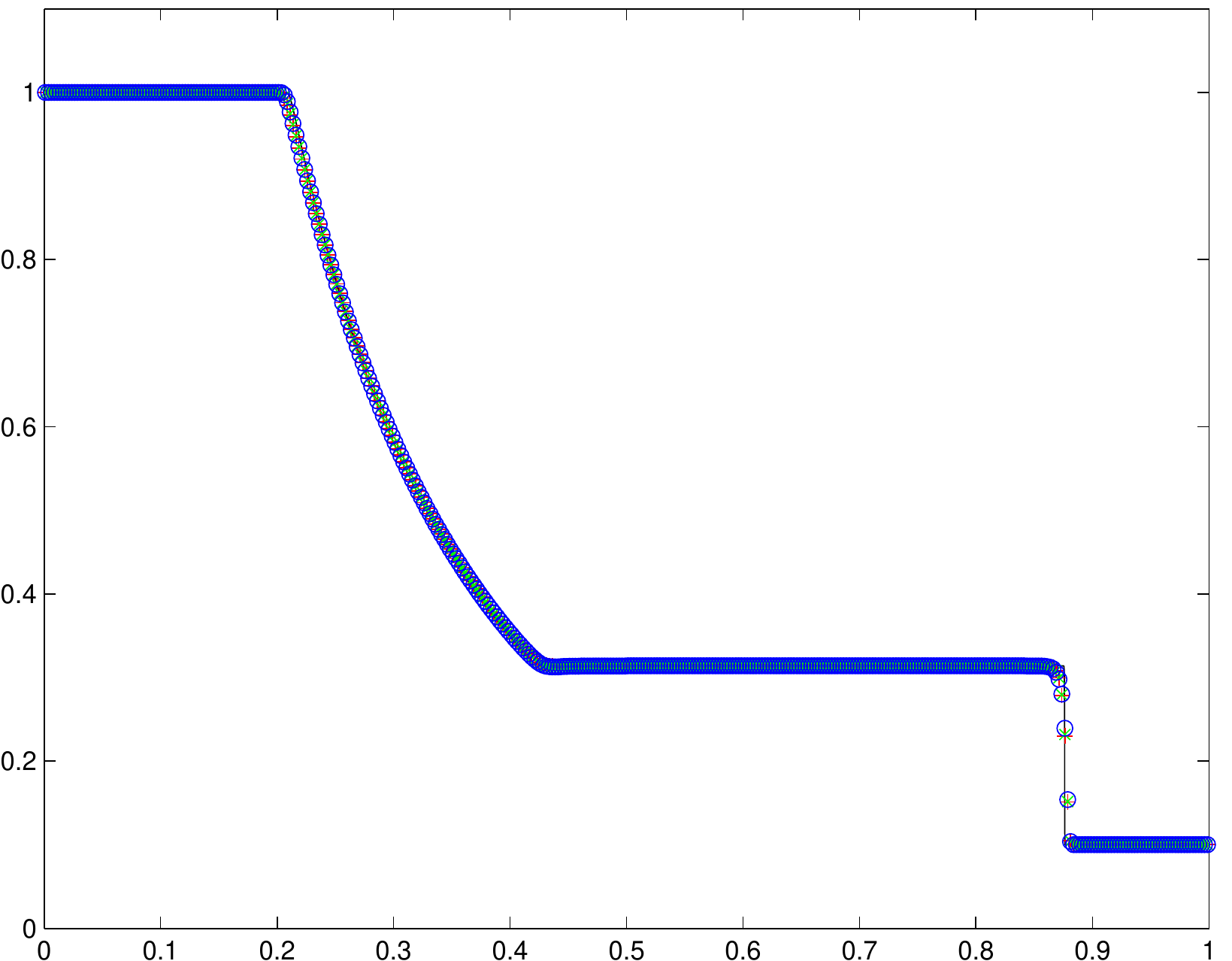}
 }
 \caption{Example \ref{ex:sinewave}: The numerical solutions at $t=0.5$ obtained by our BGK scheme (``{$\circ$}"), the BGK-type scheme (``{$\times$}"),
and the KFVS scheme (``{+}") with 400 uniform cells.}
 \label{fig:sinewave}
\end{figure}
Fig. \ref{fig:sinewave} plots the numerical results at $t=0.5$ in the computational domain $\Omega=[0,1]$ obtained by using our BGK scheme (``{$\circ$}"), the BGK-type scheme (``{$\times$}"),
and the KFVS scheme (``{+}") with 400 uniform cells.
Those are compared with the reference solution (the solid line) obtained by using the KFVS scheme with a finer
mesh of 10000 uniform cells. It is seen that
 the shock wave is moving into a sinusoidal density field,
some complex but smooth structures are generated at the left hand side of the shock wave  when the shock wave interacts with the sine wave,
and our BGK scheme is obviously better than the BGK-type and KFVS schemes
in resolving those complex structures.
Since the continuity equation in the Euler equations decouples from   other equations for the pressure and velocity, one does not see the effect of perturbation in the pressure \cite{kunik2004bgktype}.

\begin{example}[Collision of blast waves]\label{ex:blast}\rm
It is about the collision of blast waves and  simulated to evaluate the performance of the genuine BGK scheme and the BGK-type and KFVS schemes for the flow with strong discontinuities.
 The initial data are taken as follows
 \begin{equation}
   \label{eqex:blast}
(n,u_1,p)(x,0)=
   \begin{cases}
     (1.0,0.0,100.0),& 0<x<0.1,\\
     (1.0,0.0,0.06),& 0.1<x<0.9,\\
     (1.0,0.0,10.0),& 0.9<x<1.0.
   \end{cases}
 \end{equation}
\end{example}
Reflecting boundary conditions are specified at the two ends of the unit interval $[0,1]$.

Fig. \ref{fig:blast} plots the numerical results at $t=0.75$ obtained by using our BGK scheme (``{$\circ$}"), the BGK-type scheme (``{$\times$}"),
and the KFVS scheme (``{+}") with 700 uniform cells within the domain $[0,1]$.
It is found that the solutions at $t=0.75$ are bounded by two shock waves and those schemes can well resolve those shock waves. However, the genuine BGK scheme exhibits better resolution of  the contact discontinuity
than the BGK-type and KFVS schemes.
\begin{figure}[htbp]
 \centering
 \subfigure[$ n $]{
 \includegraphics[width=0.3\textwidth]{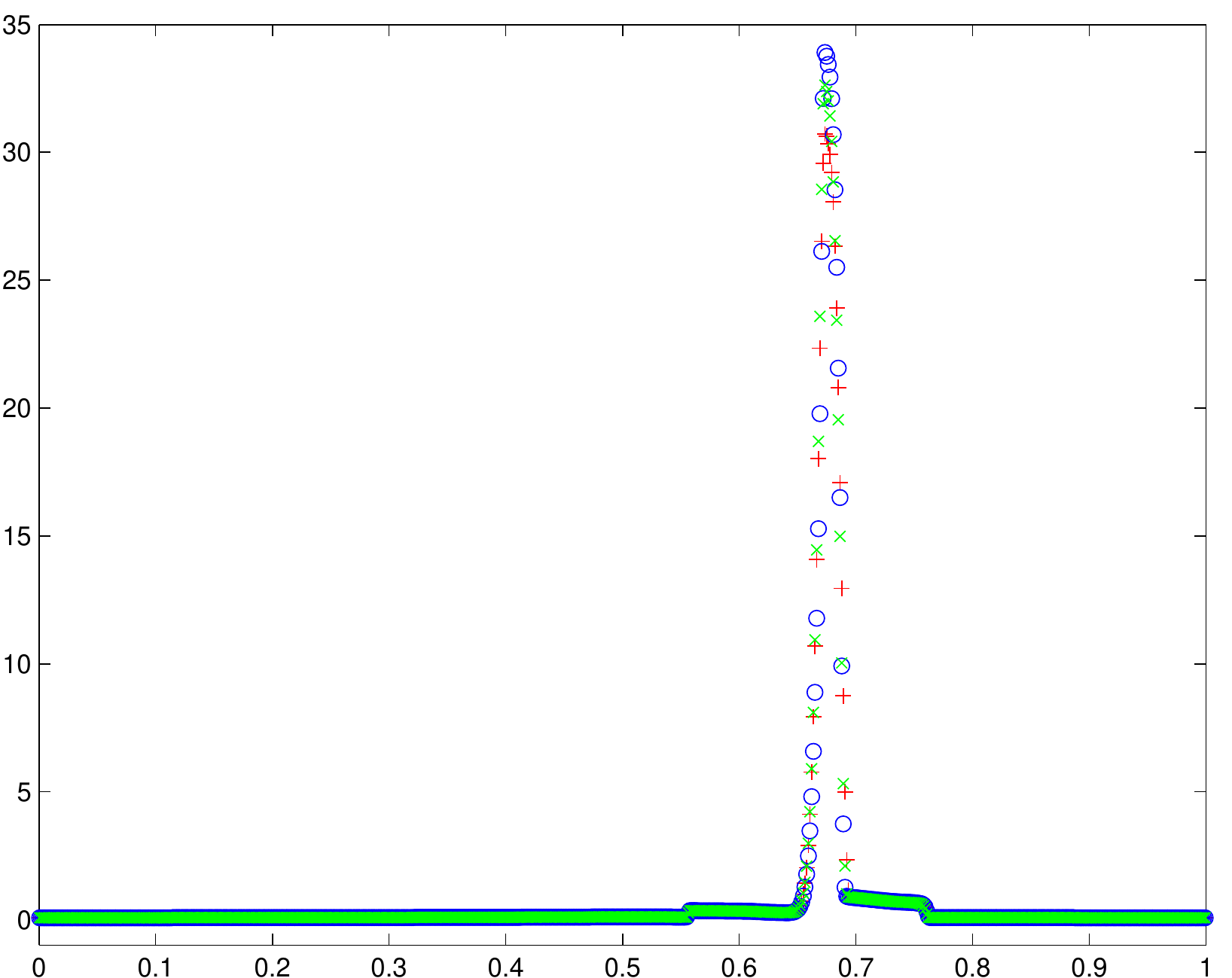}
 }
 \subfigure[$u_1$]{
 \includegraphics[width=0.3\textwidth]{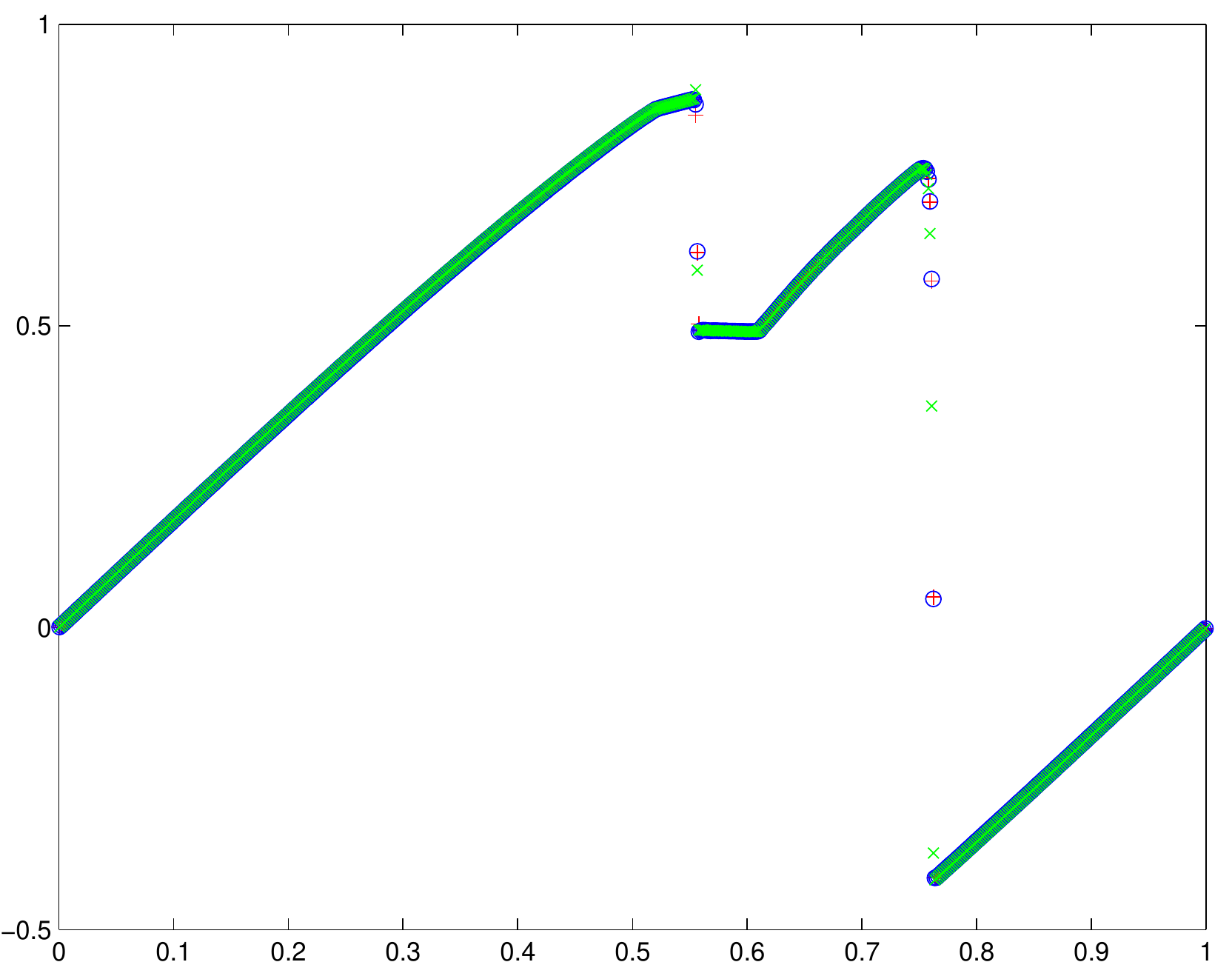}
 }
 \subfigure[$p$]{
 \includegraphics[width=0.3\textwidth]{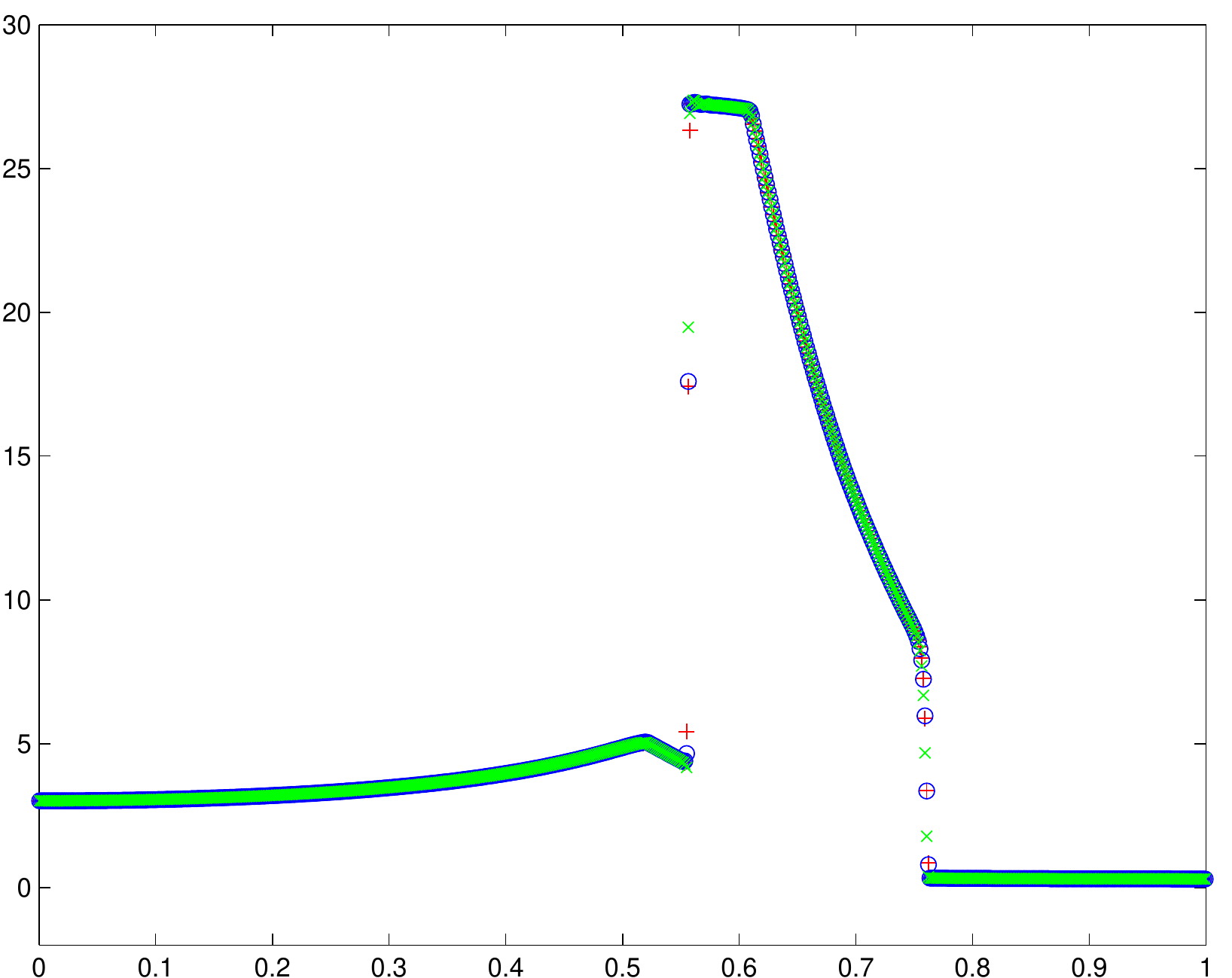}
 }
 \caption{Example \ref{ex:blast}: The number density $ n $,  velocity $u_1$ and  pressure $p$
 at $t=0.75$ obtained by using our BGK scheme (``{$\circ$}"), the BGK-type scheme (``{$\times$}"),
and the KFVS scheme (``{+}") with 700 uniform cells.}
 \label{fig:blast}
\end{figure}

\subsection{2D Euler case}

\begin{example}[Accuracy test]\label{ex:accurary2D}\rm To check the accuracy of our BGK scheme,
 we   solve a smooth problem which describes a sine wave propagating periodically in the domain $\Omega=[0,1]\times[0,1]$ at an angle $\alpha={45}^{\circ}$ with the $x$-axis.
 The initial conditions are taken as follows
\[ n (x,y,0)= 1+0.5\sin(2\pi (x+y)), u_1(x,y,0)=u_2(x,y,0)=0.2, p(x,y,0)=1,\]
so that the exact solution can be given by
\[ n (x,y,t)= 1+0.5\sin(2\pi (x-0.2t + y-0.2t)), u_1(x,y,t)=u_2(x,y,t)=0.2, p(x,y,t)=1.\]
The computational domain $\Omega$ is divided into $N\times N$ uniform cells and the periodic boundary conditions are specified.
\end{example}
\begin{table}[H]
 \setlength{\abovecaptionskip}{0.cm}
 \setlength{\belowcaptionskip}{-0.cm}
 \caption{Example \ref{ex:accurary2D}: Numerical errors at $t = 0.1$  in $l^1, l^2$-norms and convergence rates with or without limiter.}\label{tab:accuracy2D}
 \begin{center}
   \begin{tabular}{*{9}{c}}
     \toprule
     \multirow{2}*{N} &\multicolumn{4}{c}{With limiter} &\multicolumn{4}{c}{Without limiter}\\
     \cmidrule(lr){2-5}\cmidrule(lr){6-9}
     & $l^1$ error  & $l^1$ order  &  $l^2$ error  &  $l^2$ order & $l^1$ error  & $l^1$ order  &  $l^2$ error  &  $l^2$ order\\
     \midrule
     25   & 1.7316e-03 &  -       &2.5820e-03 & -       &6.1369e-04   &-       &6.8214e-04  &-  \\
     50   & 5.3784e-04 &  1.6869   &9.1457e-04 &1.4974  &1.5610e-04   &1.9751  &1.7340e-04  &1.9759\\
     100  & 1.4248e-04 &  1.9164   &2.8992e-04 &1.6574  &3.9584e-05   &1.9795  &4.3962e-05  &1.9798\\
     200  & 3.8119e-05 &  1.9022   &9.5759e-05 &1.5982  &9.8942e-06   &2.0003  &1.0989e-05  &2.0002\\
     400  & 1.0923e-05 &  1.8031   &3.1779e-05 &1.5914  &2.4837e-06   &1.9941  &2.7586e-06  &1.9941\\
     \bottomrule
   \end{tabular}
 \end{center}
\end{table}

Table \ref{tab:accuracy2D} gives the $l^1$- and $l^2$- errors at $t=0.1$ and corresponding convergence rates for the BGK scheme with $\alpha = 2$ and $C_1=C_2=1$.
The results show that the 2D BGK scheme is second-order accurate  and  the van Leer limiter affects the accuracy.

To verify the capability of our genuine BGK scheme in capturing the complex 2D relativistic wave configurations, we will solve three inviscid problems: explosion in a box, cylindrical explosion, and ultra-relativistic jet problems.

\begin{example}[Implosion in a box]\rm\label{ex:implosion}
 This example considers a 2D Riemann problem inside a squared domain $[0,2]\times[0,2]$ with reflecting walls. A square with side
 length of 0.5 embedded in the center of the outer box of side length of 2. The number density is 4 and the pressure is 10 inside the small box while both the density
 and the pressure are 1 outside of the small box. The fluid velocities are zero everywhere.
\end{example}

Figs.~\ref{fig:implosion3} and  \ref{fig:implosion12} give the contours of the density, pressure and velocities
at time $t=3$ and $12$ obtained by our BGK scheme on the uniform mesh of $400\times400$  cells, respectively.
The results show that the genuine BGK scheme captures the complex wave interaction well. Fig. \ref{fig:implosioncompare} gives a comparison of the numerical densities along the line $y=1$ calculated  by using the genuine BGK  scheme (``{$\circ$}"), { BGK-type scheme (``{$\times$}"),  and KFVS scheme (``{+}") respectively. Obviously, the genuine BGK  scheme  resolves the complex wave structure  better than the BGK-type and KFVS schemes.

\begin{figure}[htbp]
 \centering
 \subfigure[$ n $]{
   \includegraphics[width=4.5cm,height=4.5cm]{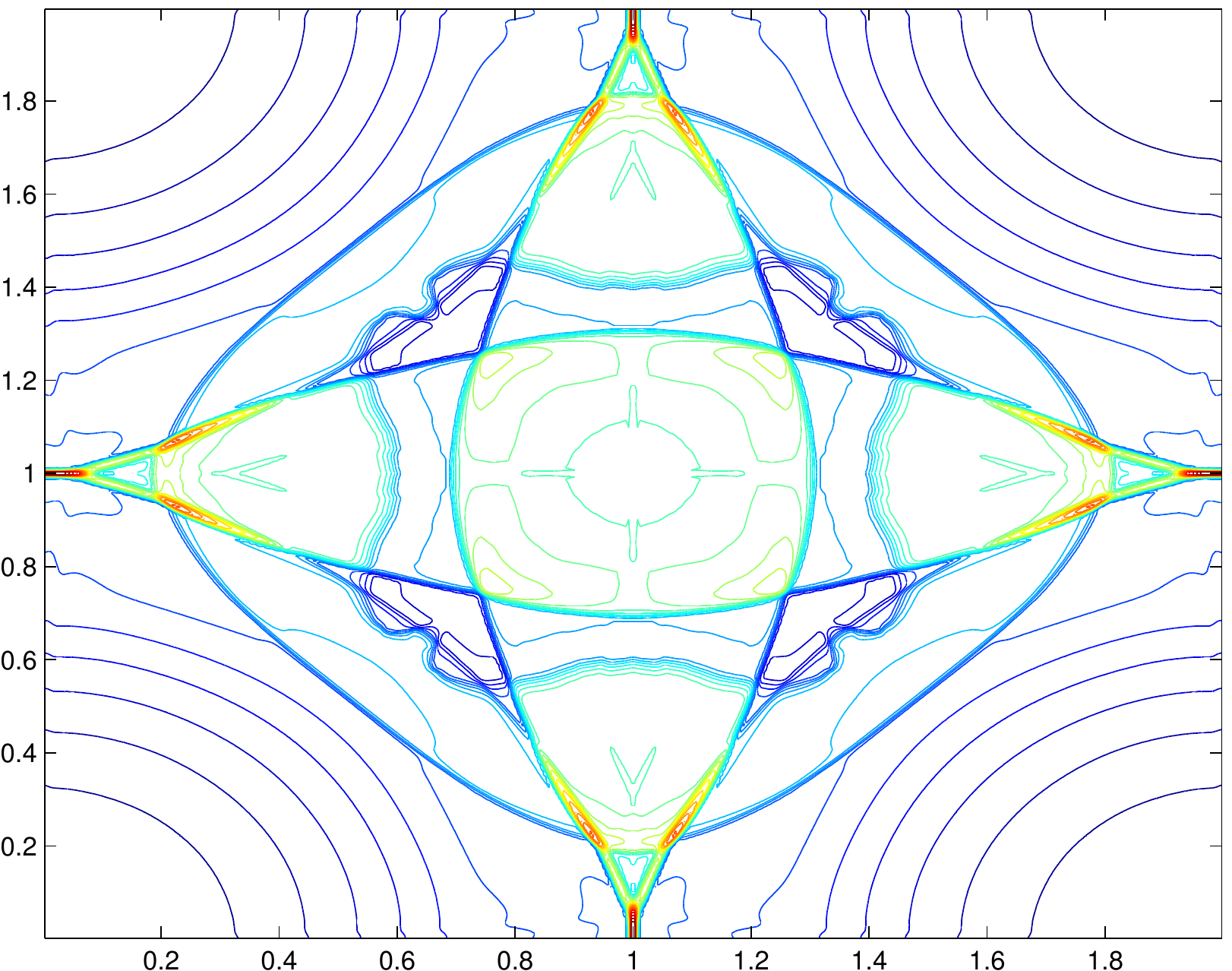}
 }
 \subfigure[$p$]{
   \includegraphics[width=4.5cm,height=4.5cm]{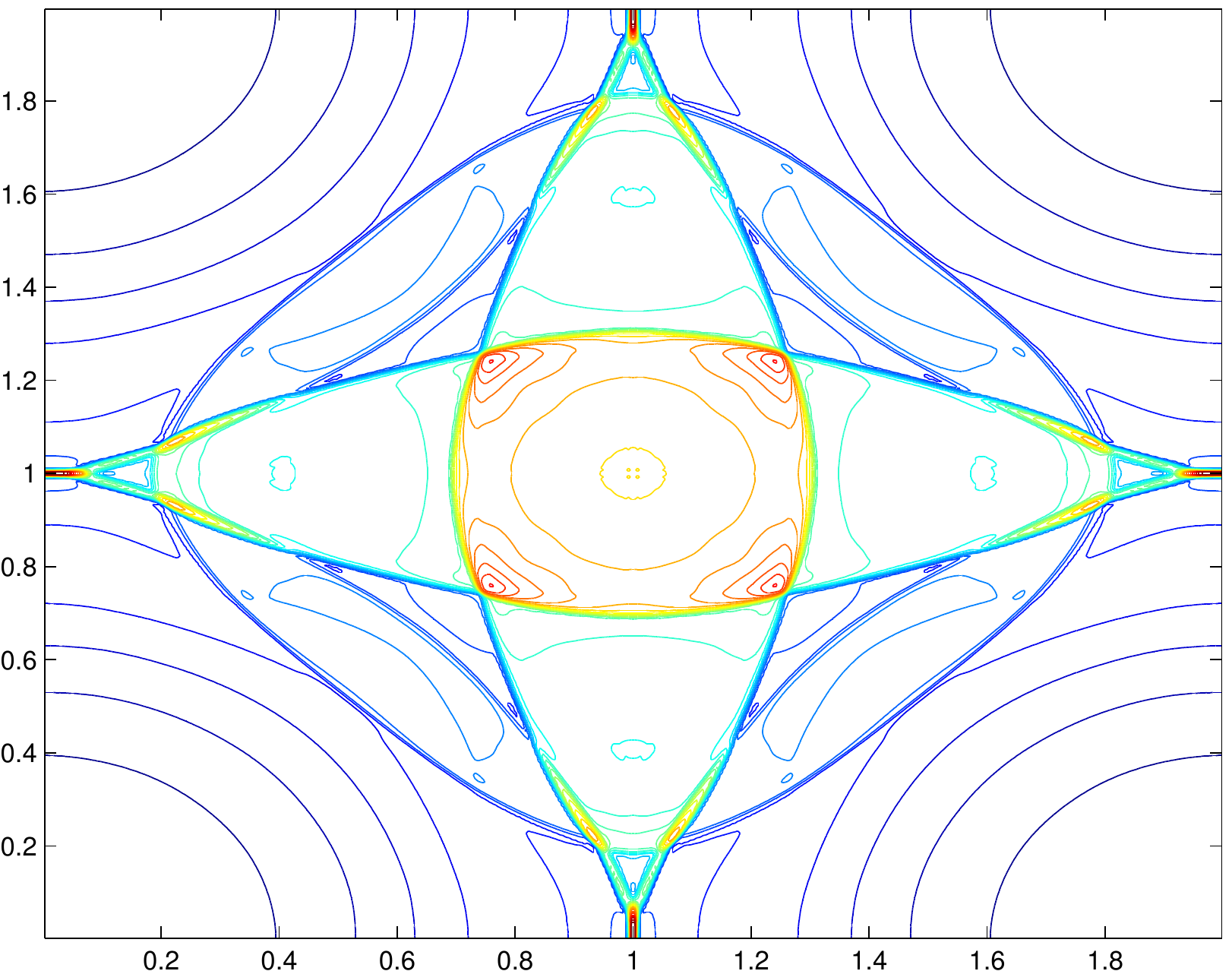}
 }

 \subfigure[$u_1$]{
   \includegraphics[width=4.5cm,height=4.5cm]{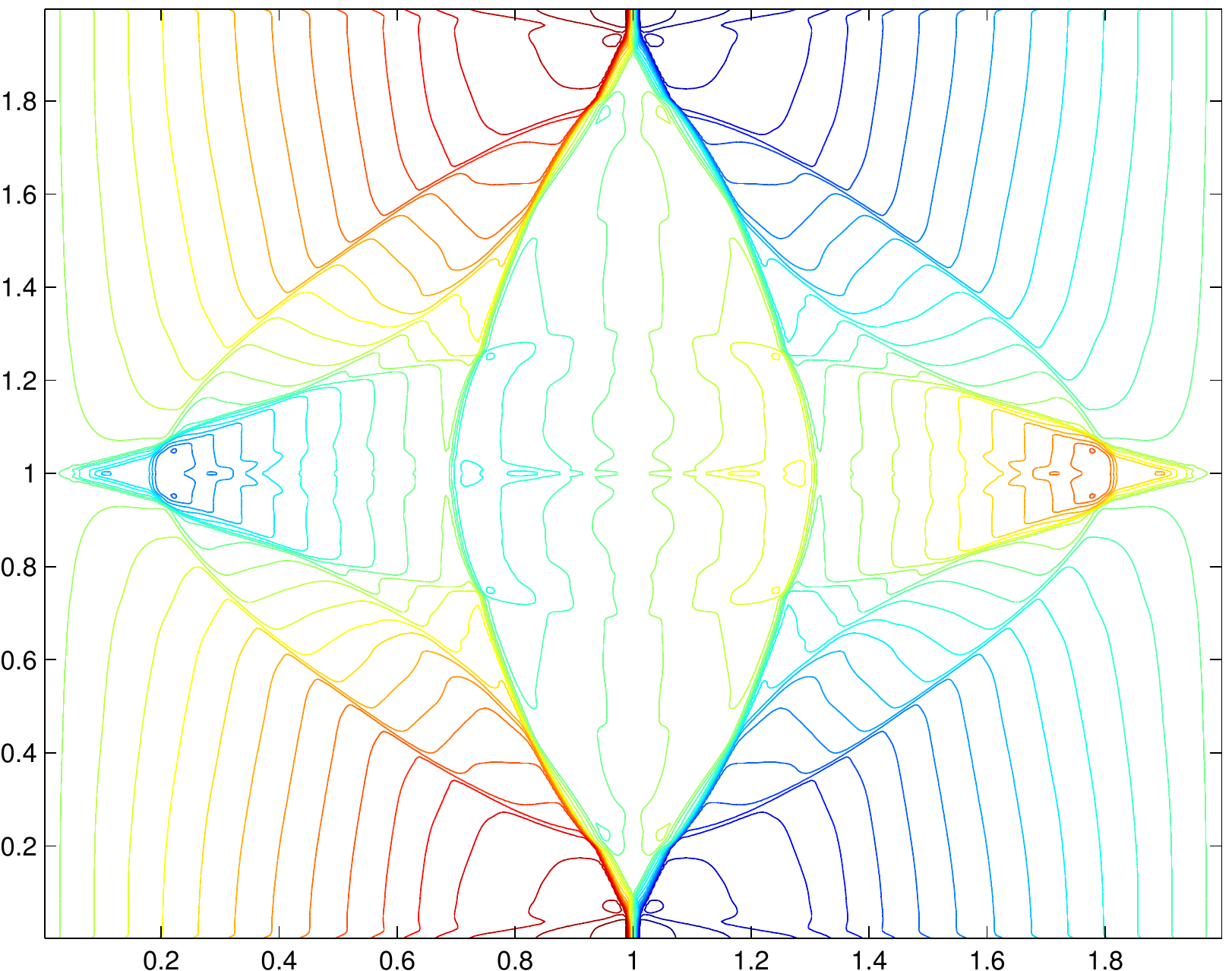}
 }
 \subfigure[$u_2$]{
   \includegraphics[width=4.5cm,height=4.5cm]{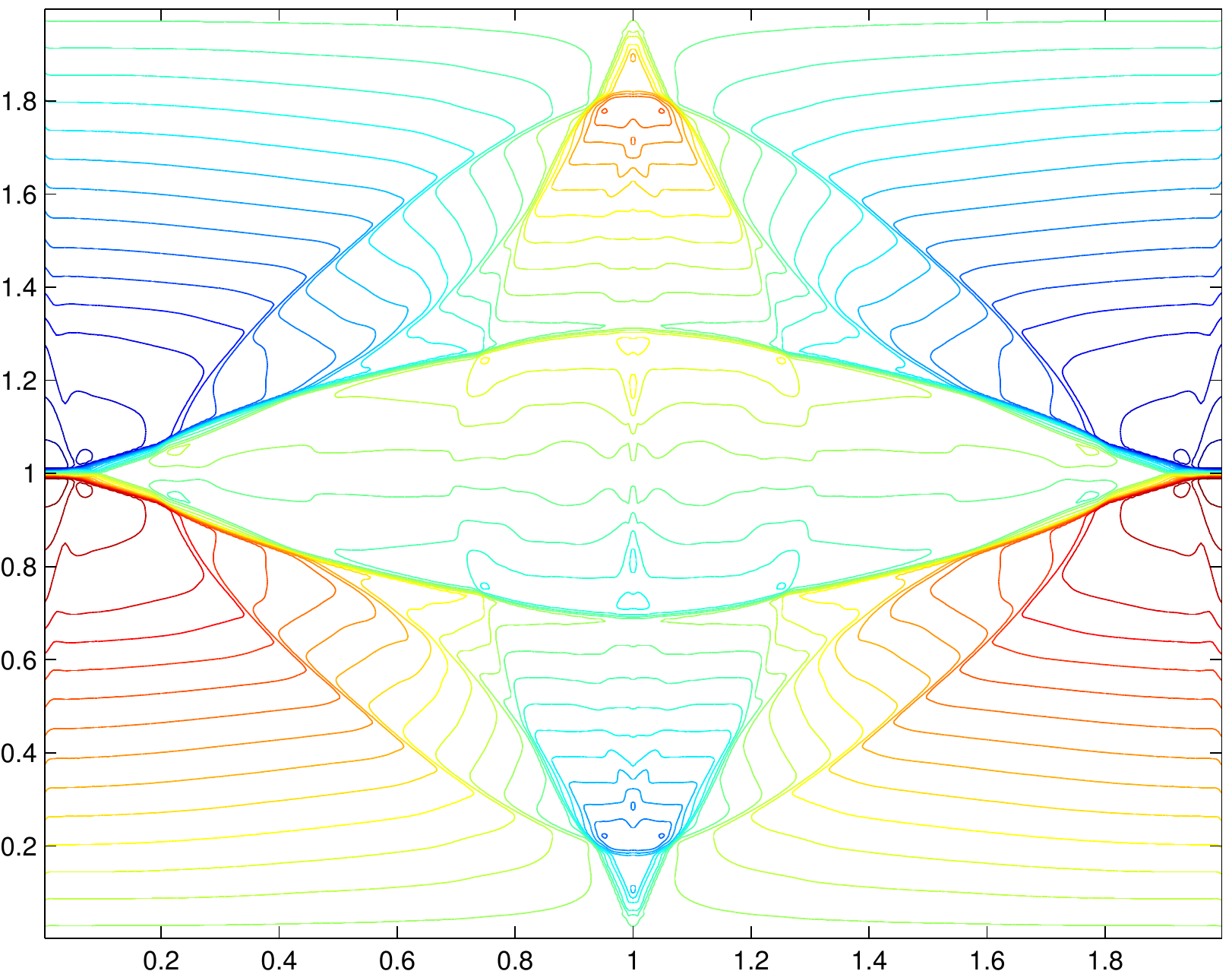}
 }
 \caption{Example \ref{ex:implosion}: The contours of the number density $ n $,   pressure $p$,  and velocities $u_1$ and  $u_2$ at $t = 3$ obtained by the BGK scheme with   $400\times400$  uniform cells.
 30 equally spaced contour lines are used.}
 \label{fig:implosion3}
\end{figure}

\begin{figure}[htbp]
 \centering
 \subfigure[$ n $]{
   \includegraphics[width=4.5cm,height=4.5cm]{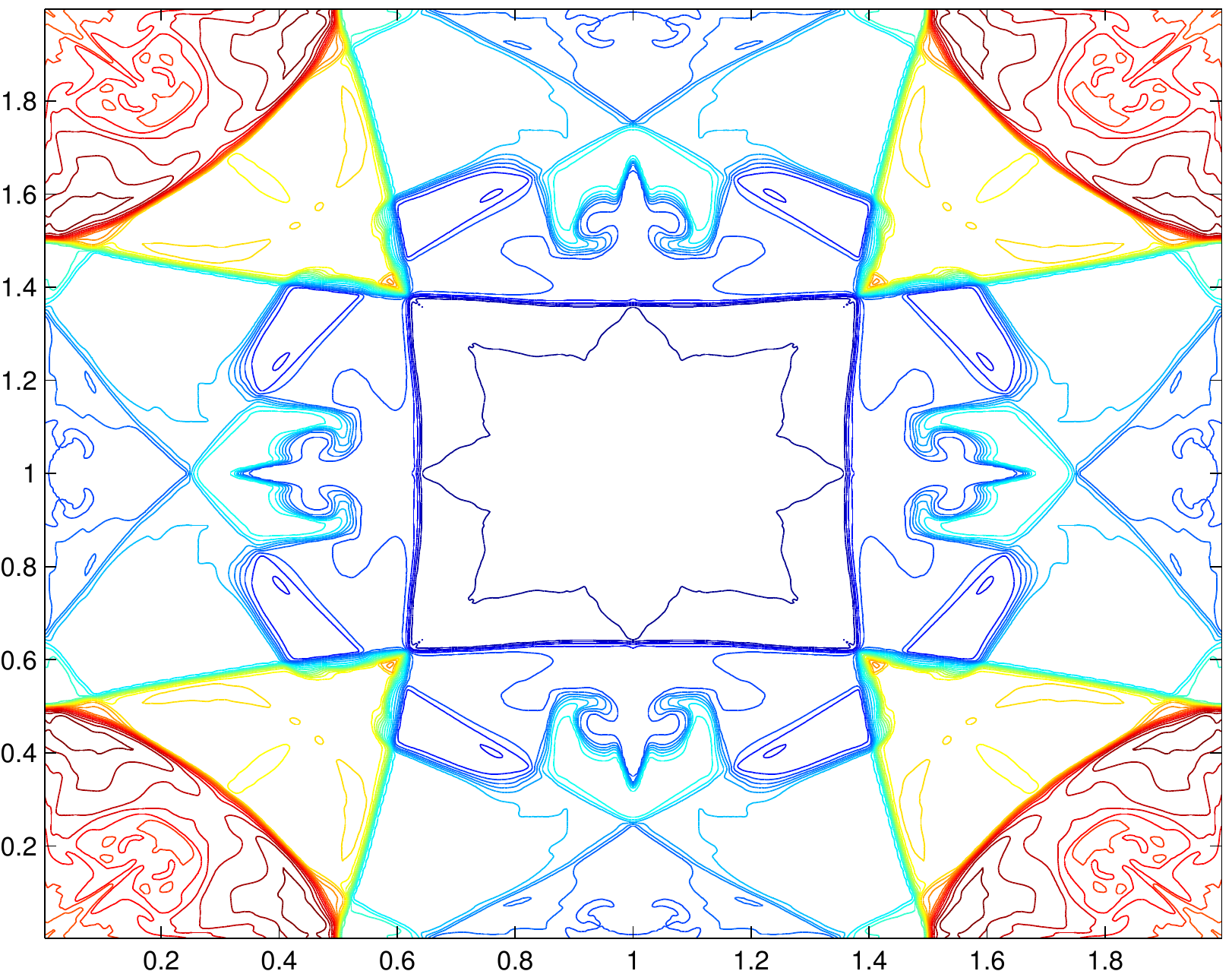}
 }
 \subfigure[$p$]{
   \includegraphics[width=4.5cm,height=4.5cm]{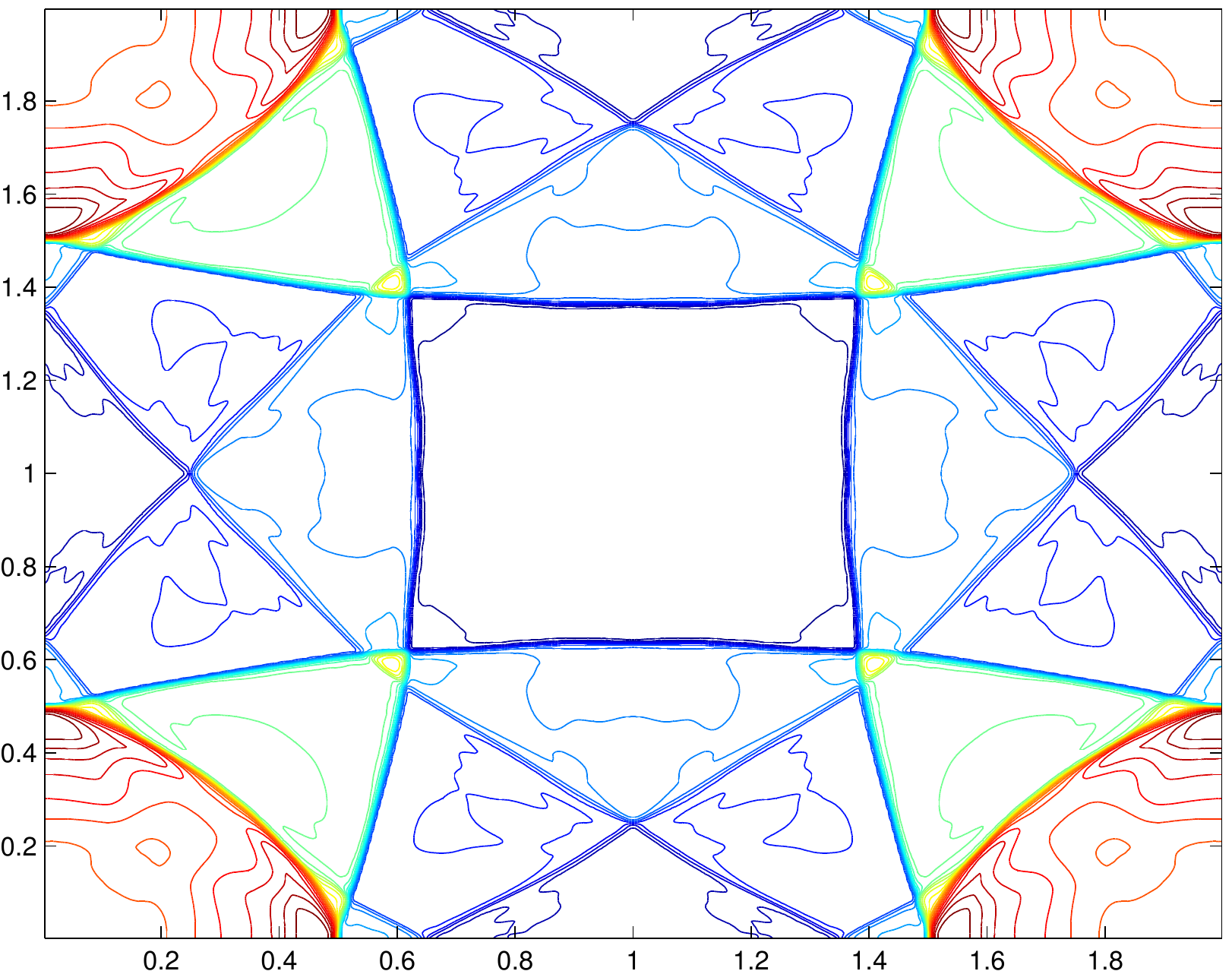}
 }

 \subfigure[$u_1$]{
   \includegraphics[width=4.5cm,height=4.5cm]{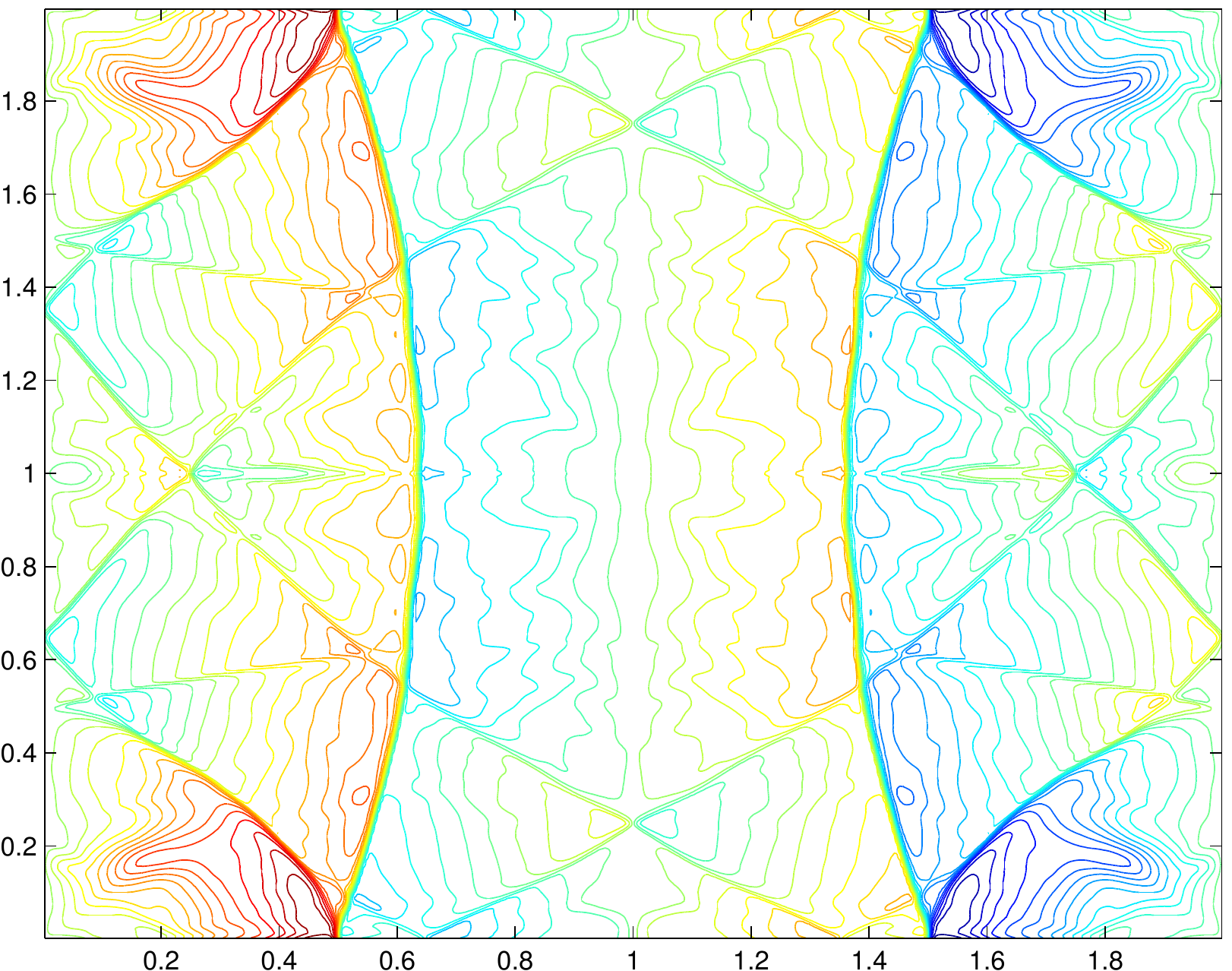}
 }
 \subfigure[$u_2$]{
   \includegraphics[width=4.5cm,height=4.5cm]{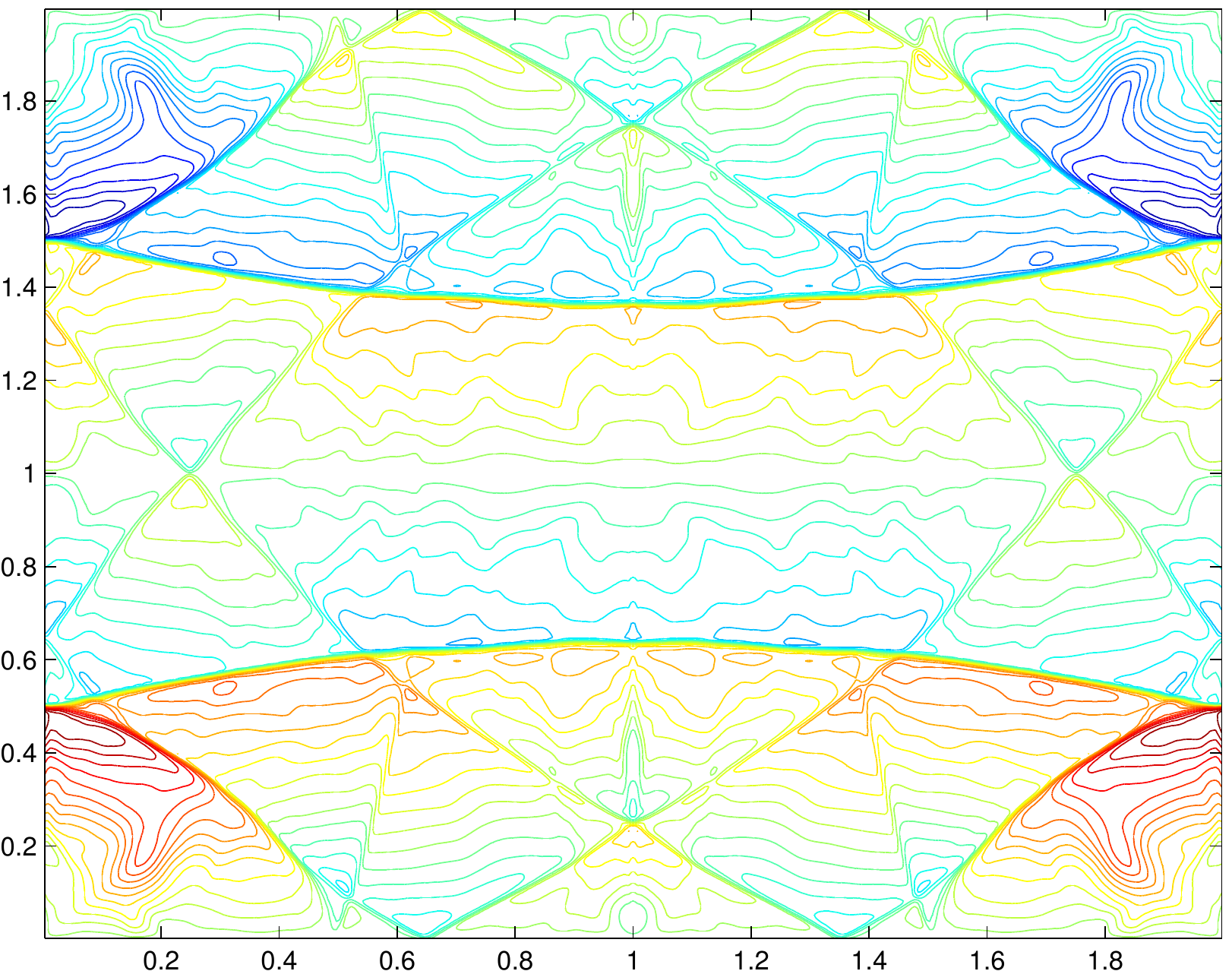}
 }
 \caption{Example \ref{ex:implosion}: Same as Fig. \ref{fig:implosion3} except for  $t=12$.}
 \label{fig:implosion12}
\end{figure}

\begin{figure}
 \centering
 \subfigure[$t=3$]{
   \includegraphics[width=0.4\textwidth]{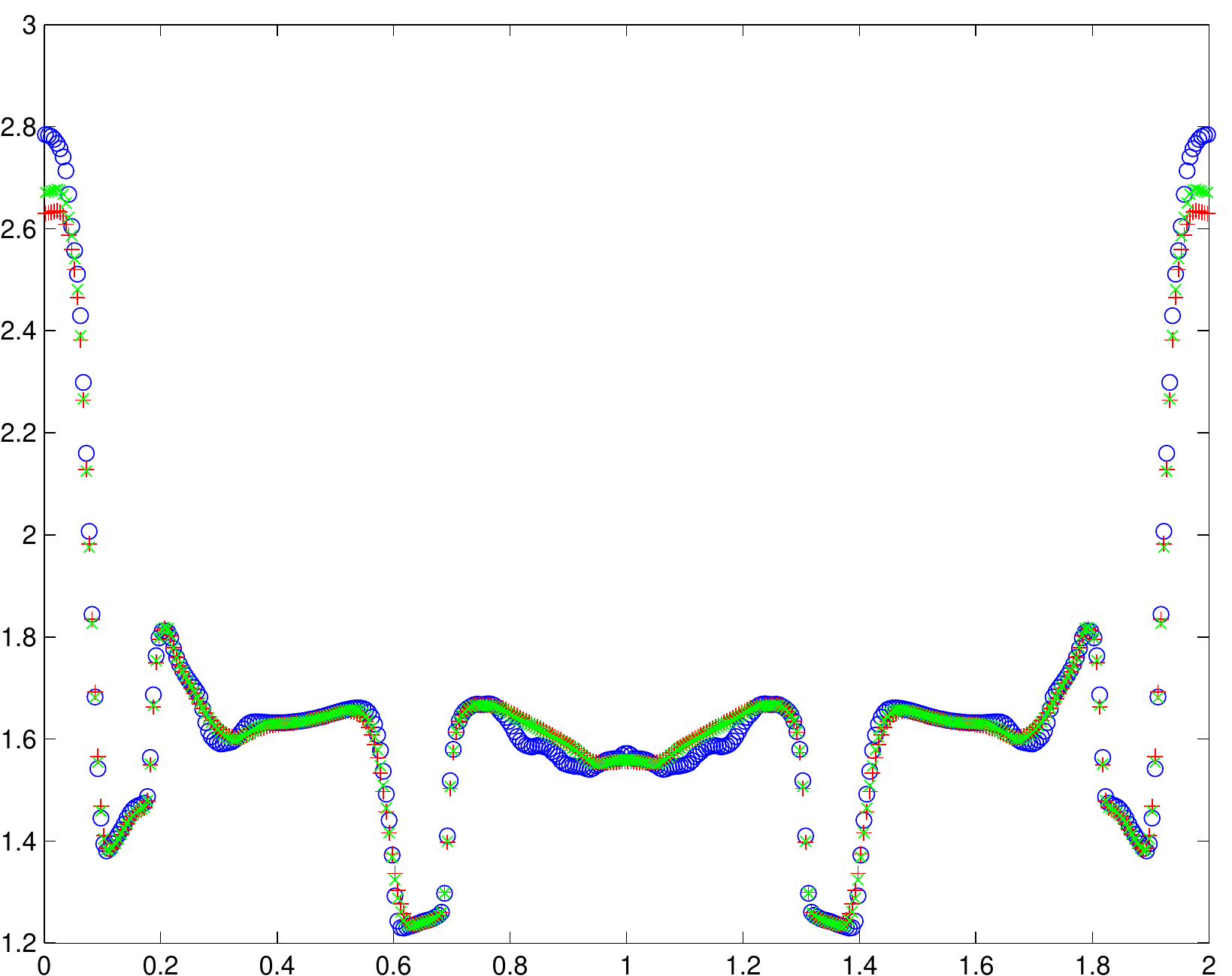}
 }
 \subfigure[$t=12$]{
   \includegraphics[width=0.4\textwidth]{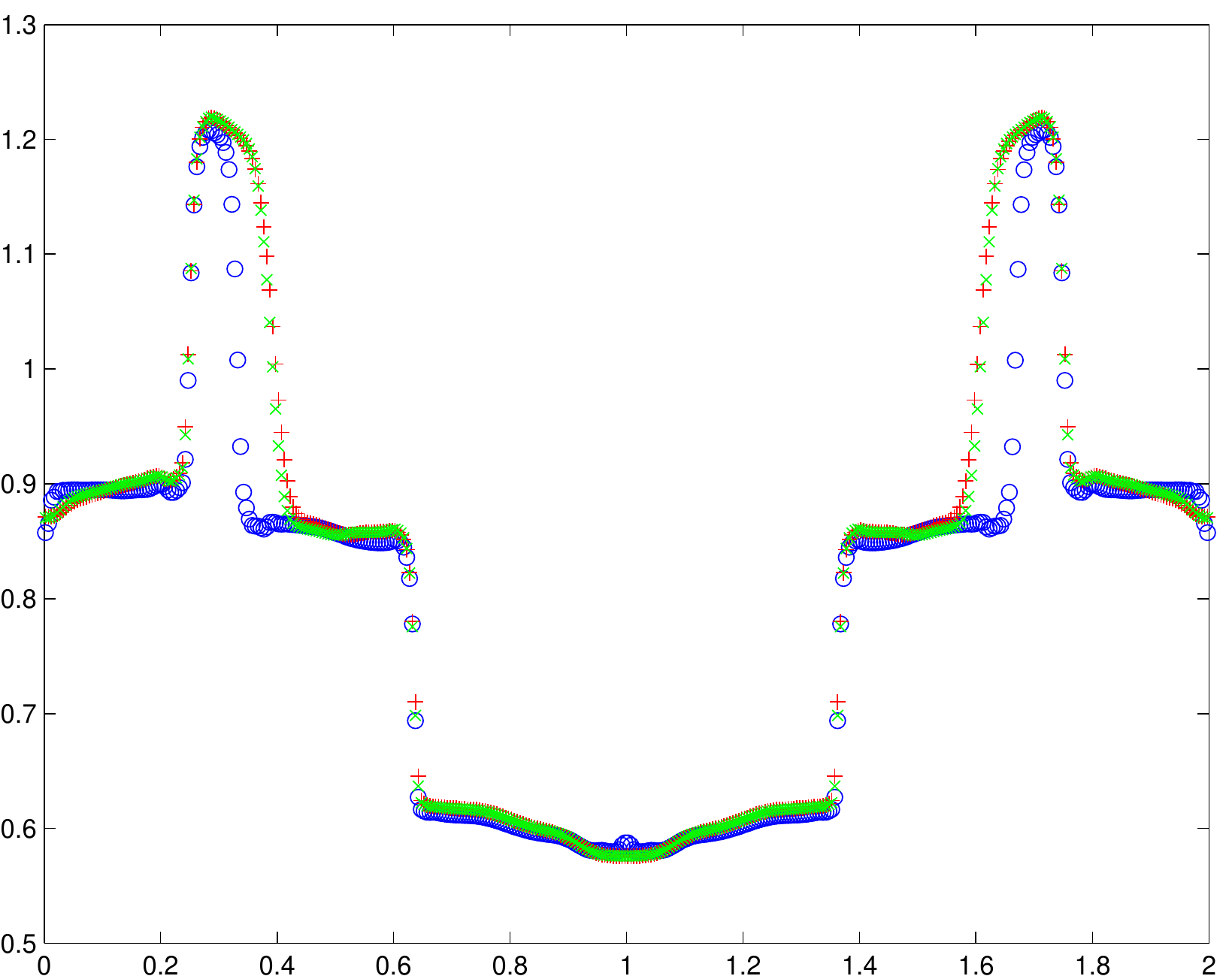}
 }
 \caption{Example \ref{ex:implosion}: Comparison of the number density $ n $ along the line $y = 1$.
 The symbols ``$\circ$", {``$\times$"} and ``{+}" denote the solutions obtained by the BGK, {BGK-type}, and KFVS schemes on the uniform mesh of $400\times400$  cells, respectively.}
 \label{fig:implosioncompare}
\end{figure}

\begin{example}[Cylindrical explosion problem]\rm\label{ex:cylindrical}
Initially, there is  a high-density, high-pressure circle with a
 radius of 0.2 embedded in a low density, low pressure medium within
  a squared domain $[0,1]\times[0,1]$.
   Inside the circle, the number density is 2 and the pressure is 10, while outside the circle
   the number density and pressure are 1 and  0.3, respectively. The velocities are zero everywhere.
\end{example}

Fig. \ref{fig:ccontour} displays the the contour plots  at $t = 0.2$ obtained by using
the BGK scheme on the mesh of $200\times200$ uniform cells.
The results show that a circular shock wave and a circular discontinuity travel away from the center, and a circular rarefaction wave propagates toward the center of the circle.
Fig. \ref{fig:ccompare} gives a comparison of the number density and pressure along the line $y=0.5$
obtained by the BGK, BGK-type, and KFVS schemes, respectively. The symbols ``{$\circ$}" , {``$\times$"} and ``{+}" denote the solutions obtained by using the BGK, {BGK-type} and KFVS schemes. It can be observed that
all of them give closer results. However, the BGK scheme resolves the discontinuities better than the KFVS.

\begin{figure}
 \centering
 \subfigure[$ n $]{
   \includegraphics[width=4.5cm,height=4.5cm]{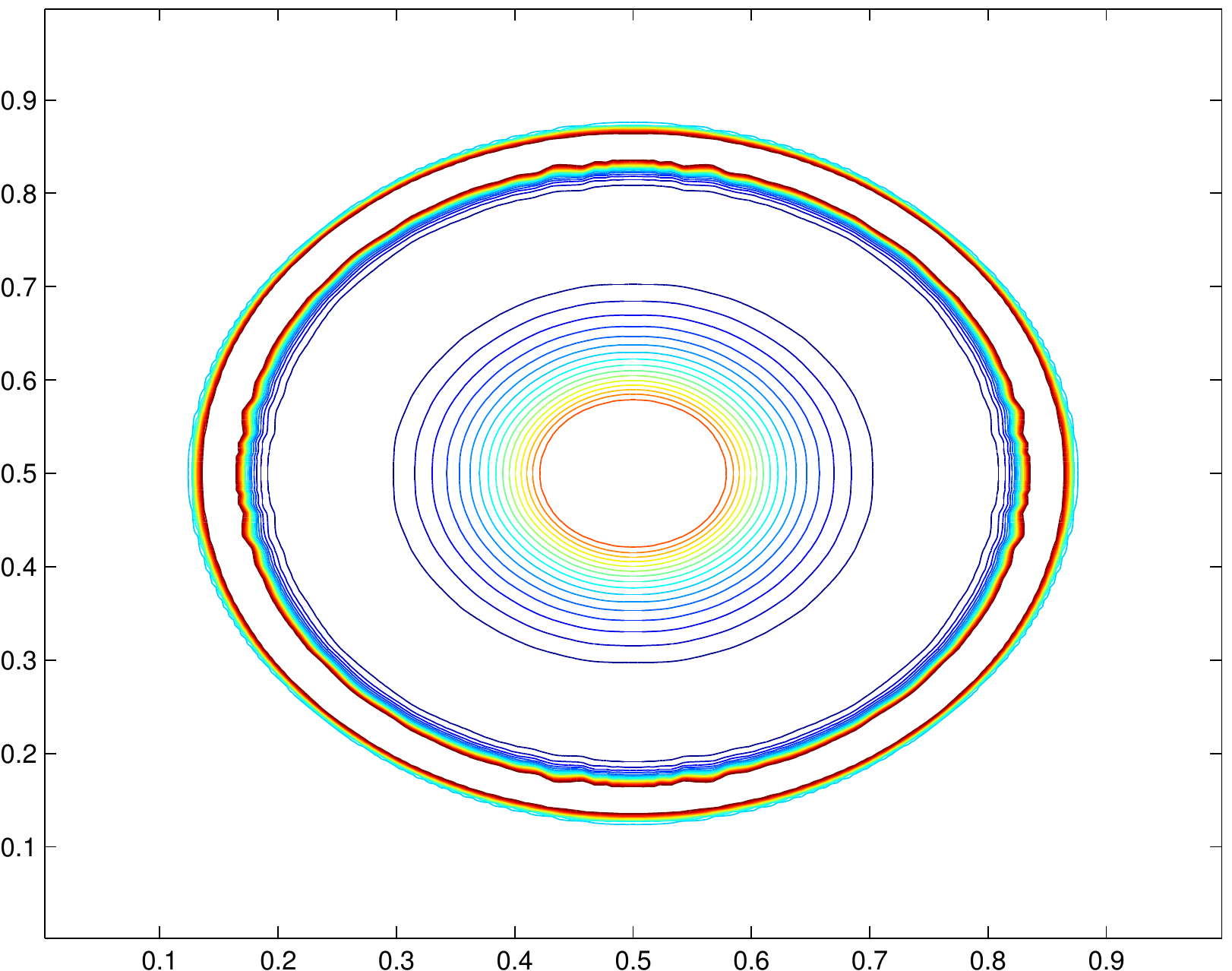}
 }
 \subfigure[$p$]{
   \includegraphics[width=4.5cm,height=4.5cm]{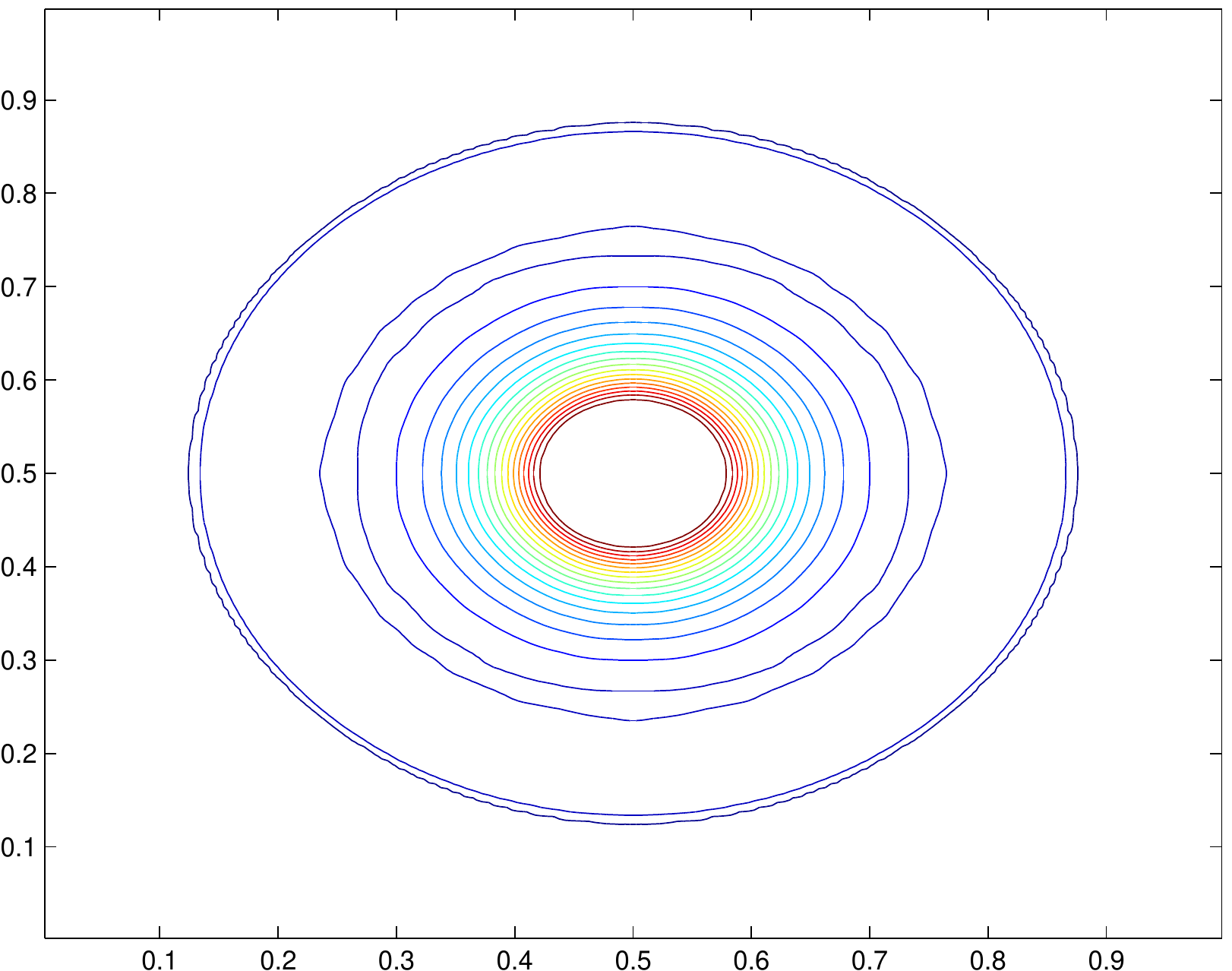}
 }

 \subfigure[$u_1$]{
   \includegraphics[width=4.5cm,height=4.5cm]{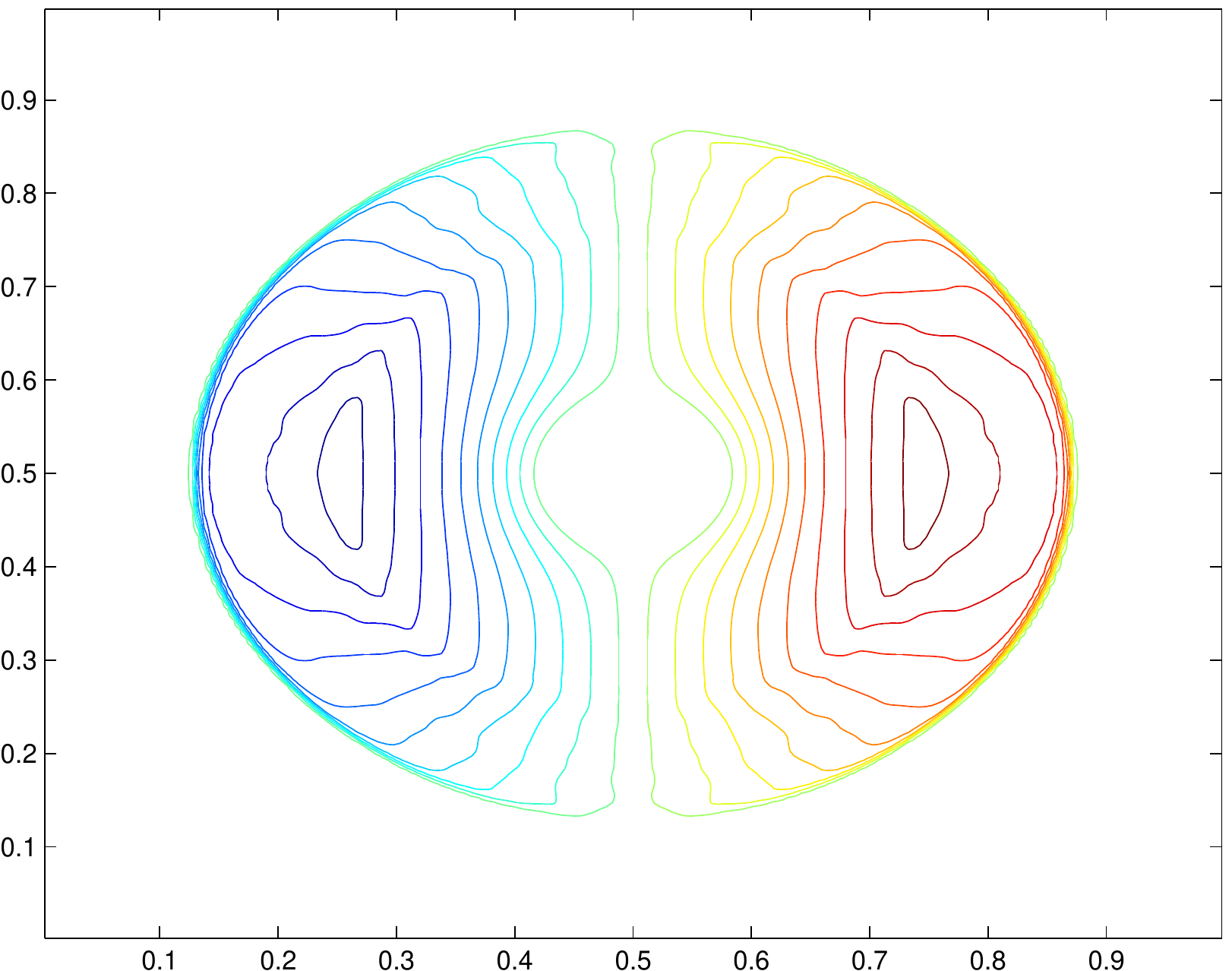}
 }
 \subfigure[$u_2$]{
   \includegraphics[width=4.5cm,height=4.5cm]{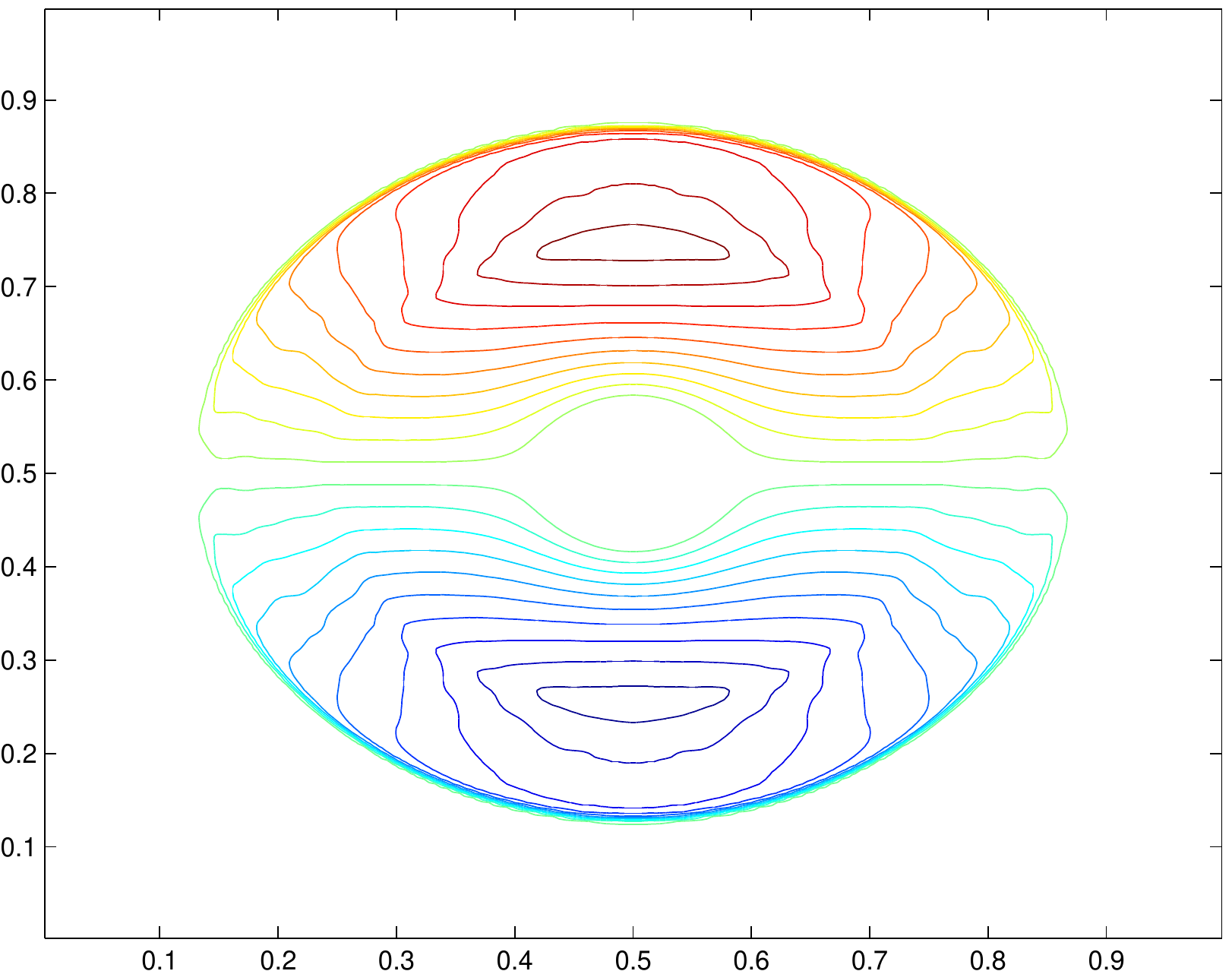}
 }
 \caption{Example \ref{ex:cylindrical}: The contours of the number density $ n $,
  pressure $p$, and  velocities $u_1$ and $u_2$ at $t = 0.2$ obtained by the BGK scheme
 on the mesh of $200\times200$ uniform cells. 20 equally spaced contour lines are used.}
 \label{fig:ccontour}
\end{figure}

\begin{figure}
 \centering
 \subfigure[$ n $]{
   \includegraphics[width=0.4\textwidth]{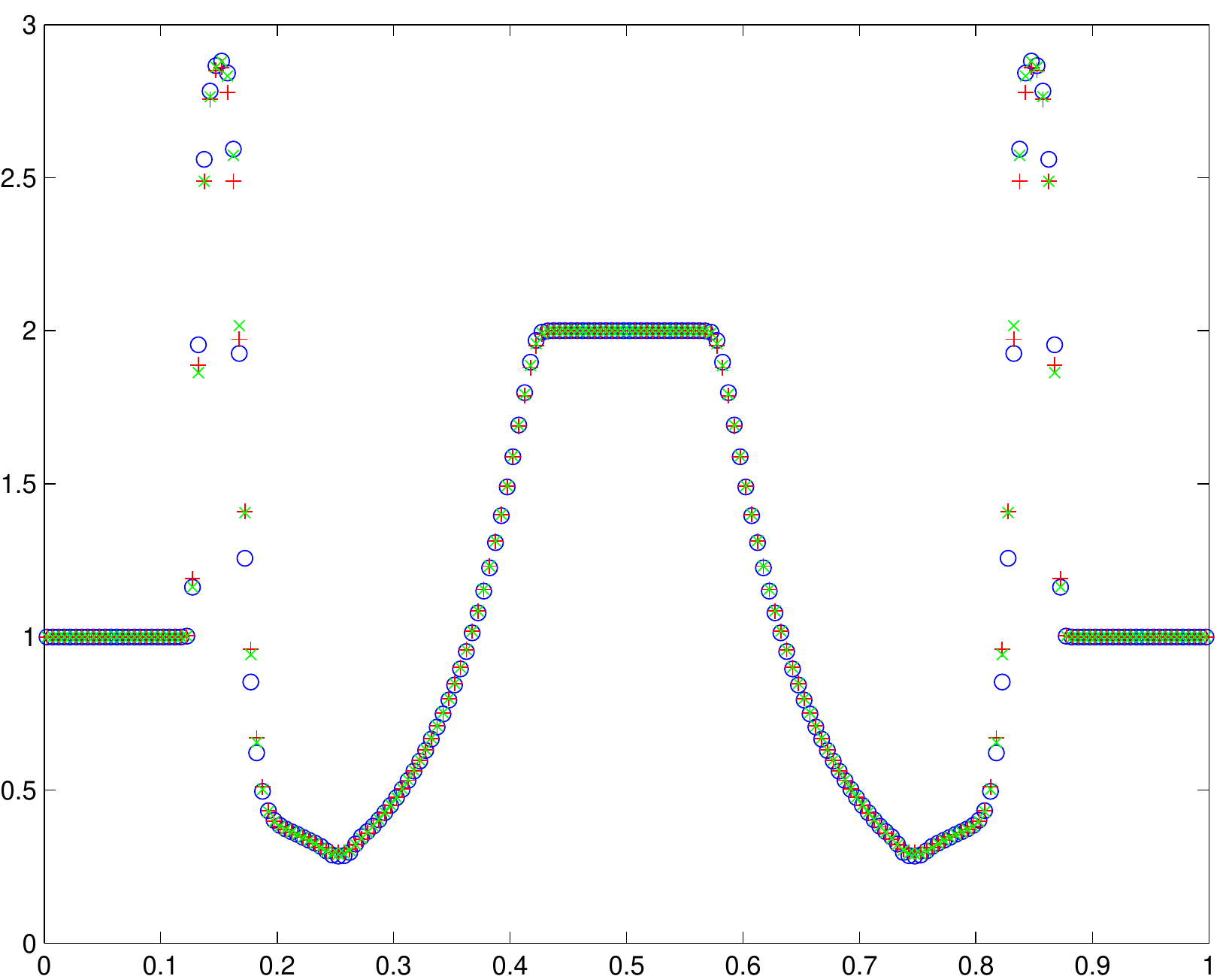}
 }
 \subfigure[$p$]{
   \includegraphics[width=0.4\textwidth]{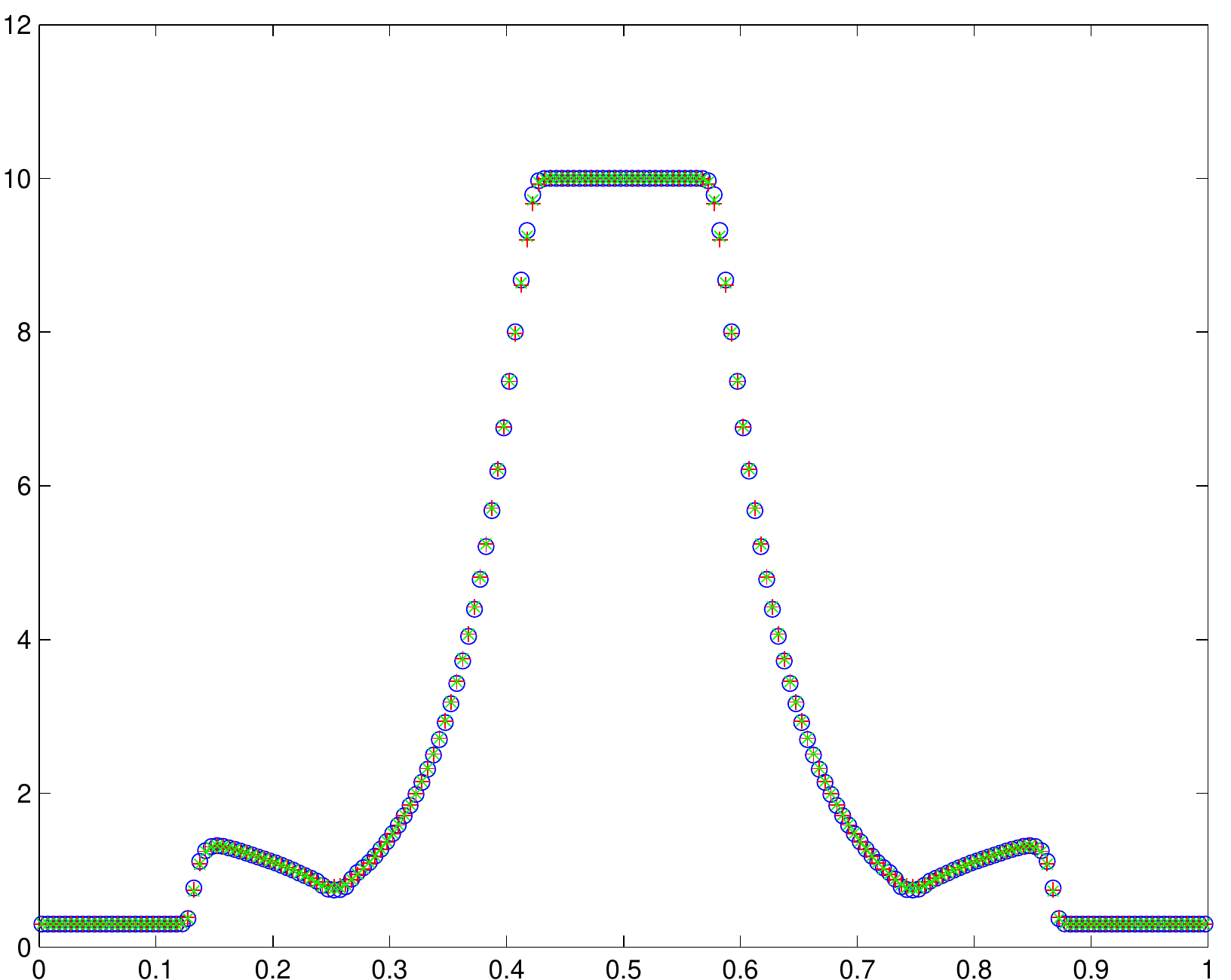}
 }
 \caption{Example \ref{ex:cylindrical}: Comparison of the number density $ n $ and pressure $p$ along the line $y = 0.5$. The symbols ``{$\circ$}", ``$\times$", and ``+" denote the solutions obtained by the BGK, BGK-type, and KFVS schemes
 on the mesh of $200\times200$ uniform cells, respectively.}
 \label{fig:ccompare}
\end{figure}

\begin{example}[Ultra-relativistic jet]\rm\label{ex:Jet}
 The dynamics of relativistic jet relevant in astrophysics has been  widely studied by numerical methods in the literature \cite{3DJET1999,RAM2006,Marti1997}.
 This test  simulates a relativistic jet
 with  the computational region  $[0,12]\times[-3.5,3.5]$  and $\alpha = C_1=C_2=1$.
 The initial states for the relativistic jet beam are
 \begin{align*}
   ( n _b,u_{1,b},u_{2,b},p_b)=(0.01,0.99,0.0,10.0),\ \
   ( n _m,u_{1,m},u_{2,m},p_m)=(1.0,0.0,0.0,10.0),
 \end{align*}
 where the subscripts $b$ and $m$  correspond to the beam and  medium, {respectively}.
\end{example}
The initial relativistic jet is injected through a unit wide nozzle located at the middle of left boundary while a reflecting boundary is used outside of the nozzle.
Outflow boundary conditions with zero gradients of variables are imposed at the other part of the domain boundary. Fig. \ref{fig:Jet}  shows the numerical results at $t=5,6,7,8$ obtained by our BGK scheme
 on the mesh of $600\times350$ uniform cells.
The average speed of the jet head is 0.91 which matches the theoretical estimate 0.87 in \cite{Marti1997}.
\begin{figure}
 \centering
 \subfigure[$t=5$]{
   \includegraphics[width=0.4\textwidth]{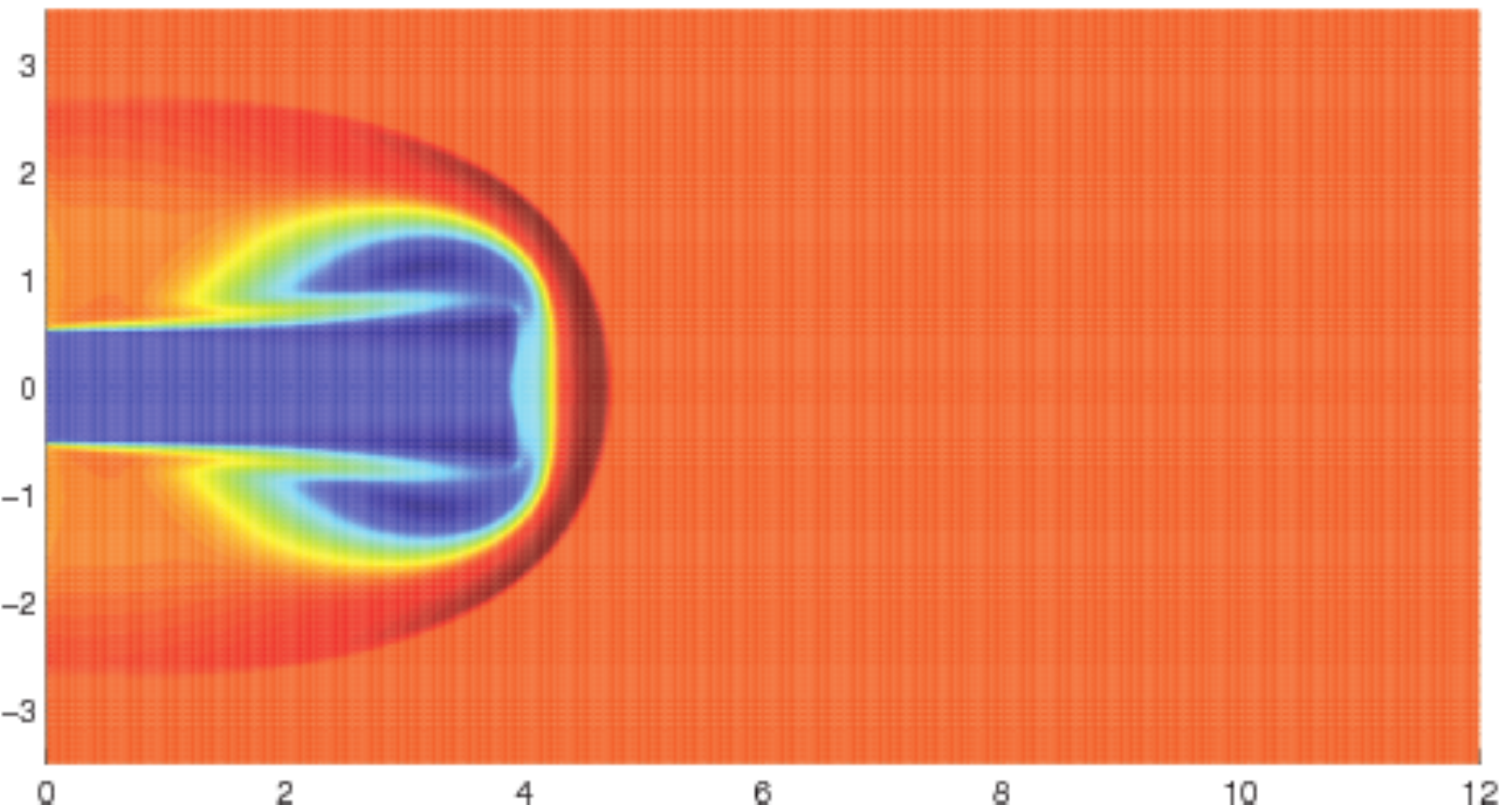}
 }
 \subfigure[$t=6$]{
   \includegraphics[width=0.4\textwidth]{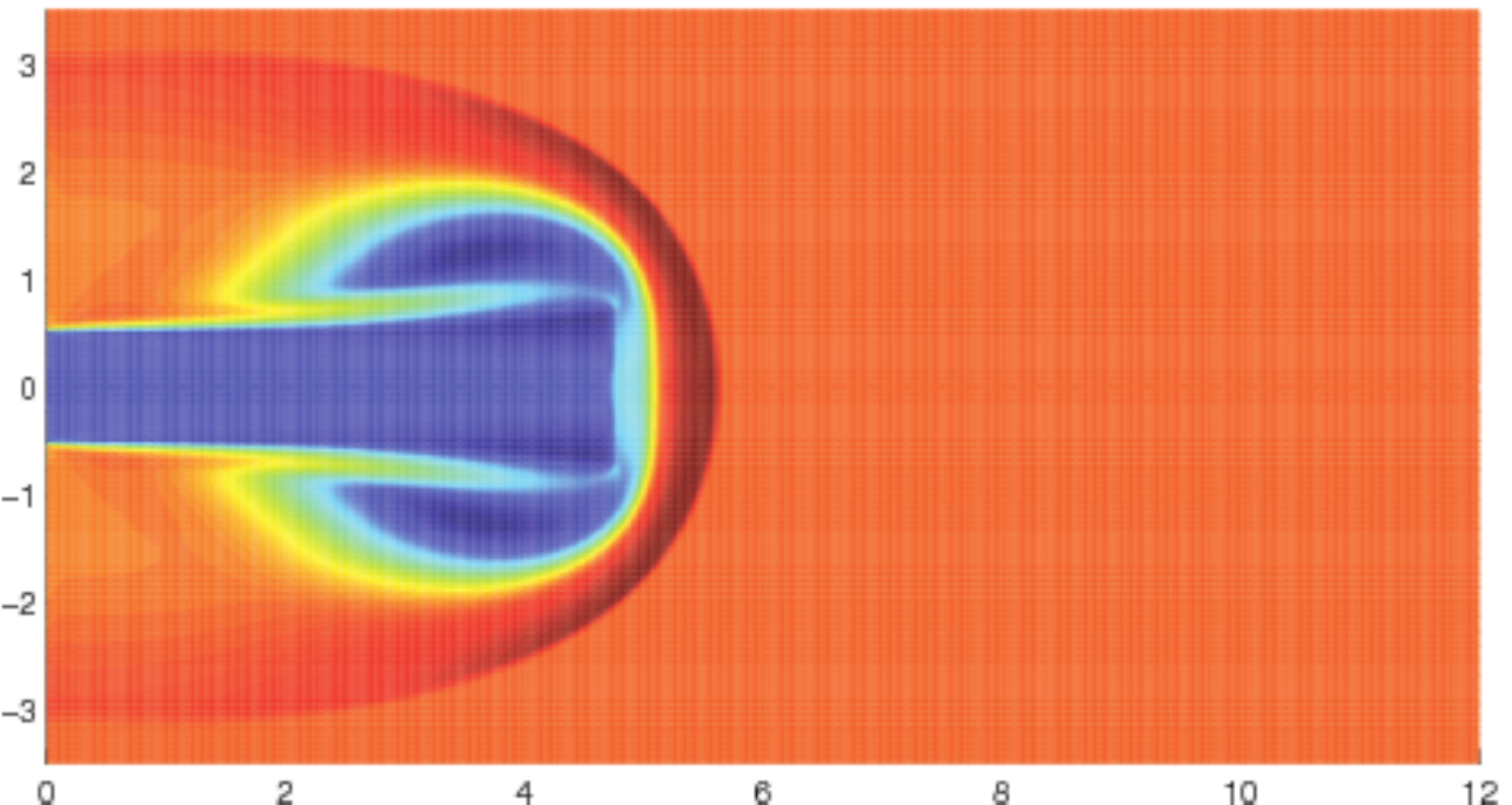}
 }
 \subfigure[$t=7$]{
   \includegraphics[width=0.4\textwidth]{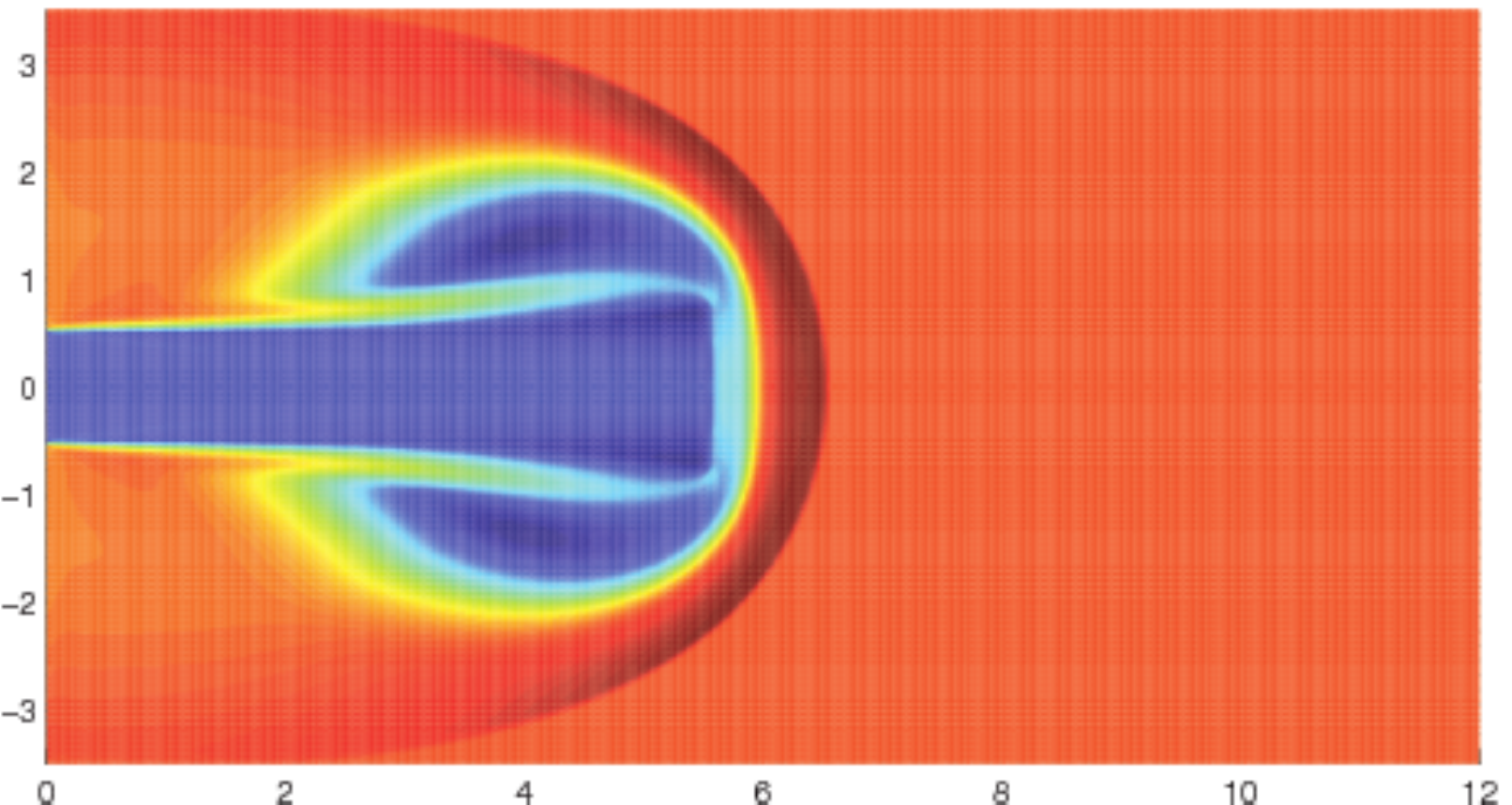}
 }
 \subfigure[$t=8$]{
   \includegraphics[width=0.4\textwidth]{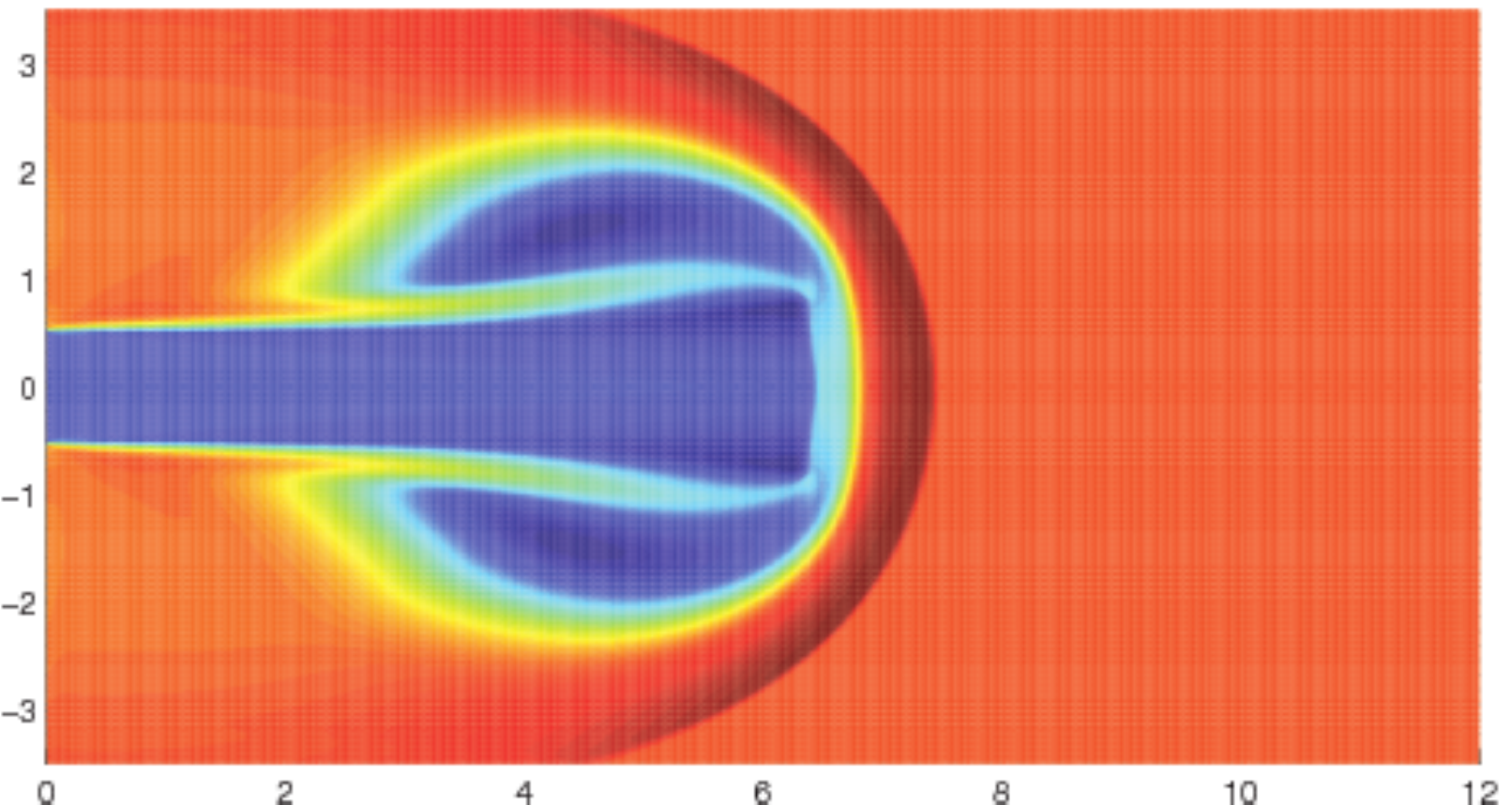}
 }
 \caption{Example \ref{ex:Jet}: Schlieren images of the number density logarithm
$\ln  n $  at several different times obtained by using the BGK
 scheme with $600\times350$ uniform cells in the domain $[0, 12]\times[-3.5, 3.5]$.}
 \label{fig:Jet}
\end{figure}

\subsection{Navier-Stokes case}
 This section designs two examples of  viscous {flow}  to test the genuine BGK
 scheme  \eqref{eq:moment1} for the ultra-relativistic Navier-Stokes equations.
 Because  the extrapolation \eqref{eq:extra} requires the numerical solutions at  $t=t_{n-1}$ and
 $t_{n-2}$, 
 the ``initial'' data at first several time levels have to be specified for
  the BGK scheme  in advance.
In the following examples, the macroscopic variables at $t=t_0+0.5\Delta t_0$ and $t_0+\Delta t_0$ are first obtained by using the initial data, time partial derivatives at $t=t_0$, and   BGK scheme proposed in Section \ref{sec:GKSNS}, where
the first order  partial derivatives in time are derived by using the exact solutions.
Then, the time partial derivatives at $t=t_0+\Delta t_0$ for the macroscopic variables are calculated by using the extrapolation \eqref{eq:extra}, and   the solutions are
further evolved in time by the BGK scheme with  the extrapolation \eqref{eq:extra}.

\begin{example}[longitudinally boost-invariant system]\rm \label{ex:boost}
  For ease of numerical implementation,
this test focuses on the longitudinally boost-invariant systems.
  They are conveniently described in curvilinear coordinates $x_m=(\tilde{t}, y, z, \eta)$, where $\tilde{t} = \sqrt{t^2-x^2}$ is the longitudinal proper time, $\eta=\frac{1}{2}\ln\left(\frac{t+x}{t-x}\right)$ is the space-time rapidity and $(y, z)$ are the usual Cartesian coordinates in the plane transverse to the beam direction $x$.
The systems are realized by assuming a specific ``scaling" velocity profile $u_1 = x/t$ along the beam direction, and the initial conditions are independent on the longitudinal reference frame (boost invariance), that is to say, they do not depend on $\eta$. The readers are referred to \cite{Song2009}  for more details.

Our computations consider the boost-invariant longitudinal expansion without transverse flow, so that the relativistic Navier-Stokes equations read
 \begin{equation*}
   \frac{\partial p}{\partial \tilde{t}} + \frac{4}{3\tilde{t}}\left(p-\frac{\mu}{3\tilde{t}}\right)=0, \ \
   \frac{\partial  n }{\partial \tilde{t}} = - n \partial_{\alpha} U^{\alpha}.
 \end{equation*}
 Since {$u_{1}=\frac{x}{t}$,} $U^0=t/\tilde{t}$ and $U^1=x/\tilde{t}$, it holds that $\partial_{\alpha} U^{\alpha} = 1/\tilde{t}$.
 Thus the equation for $ n $ becomes
 \begin{equation*}
   \frac{\partial  n }{\partial \tilde{t}} = -\frac{ n }{\tilde{t}}.
 \end{equation*}
 The analytical solutions can be given by
 \begin{equation*}
   p = C_1 \tilde{t}^{-\frac{4}{3}} + \frac{4}{3}\mu\tilde{t}^{-1},\ \
    n   = C_2\tilde{t}^{-1},
 \end{equation*}
 where $C_1=p_0(t_0^2-x_0^2)^{{\frac{2}{3}}}-\frac{4\mu}{3}(t_0^2-x_0^2)^{{\frac{1}{6}}}$ and $C_2= n _0\sqrt{t_0^2-x_0^2}$.
 We take $x_0=0, t_0=1, p_0=1,  n _0=1$, $\mu=0.0005$, and  $\Omega=[-\frac{t_0}{2},\frac{t_0}{2}]$.
 Moreover, the time partial derivatives of $ n , {u_{1}}, p$ at $t={t_0}$ are given by the exact solution.

Fig. \ref{fig:boost} shows the number density, velocity and pressure at $t=1.2$ obtained by
our 1D BGK scheme  with 20 cells (``${\triangle}$") and 40 cells
(``{$\circ$}"), respectively. The results show that the numerical results predicted by
our BGK scheme   fit the exact solutions very well.
Table \ref{tab:boost} lists the $l^1$- and $l^2$-errors at $t=1.2$
and corresponding convergence rates for our BGK scheme.
Those data show that a second-order rate of convergence can be obtained by our BGK scheme.
\end{example}

\begin{figure}[htbp]
 \centering
 \subfigure[$ n $]{
 \includegraphics[width=0.3\textwidth]{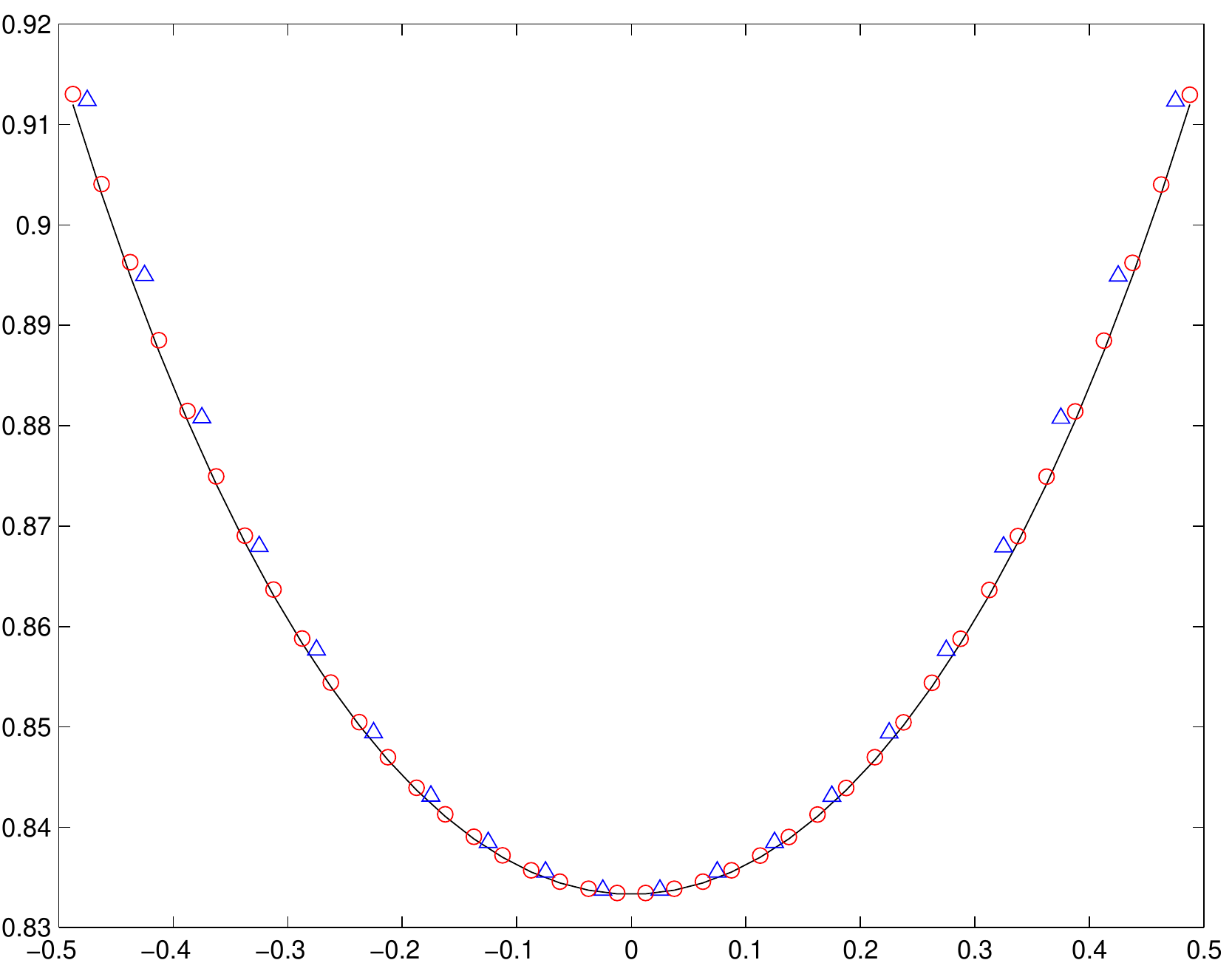}
 }
 \subfigure[$u_1$]{
 \includegraphics[width=0.3\textwidth]{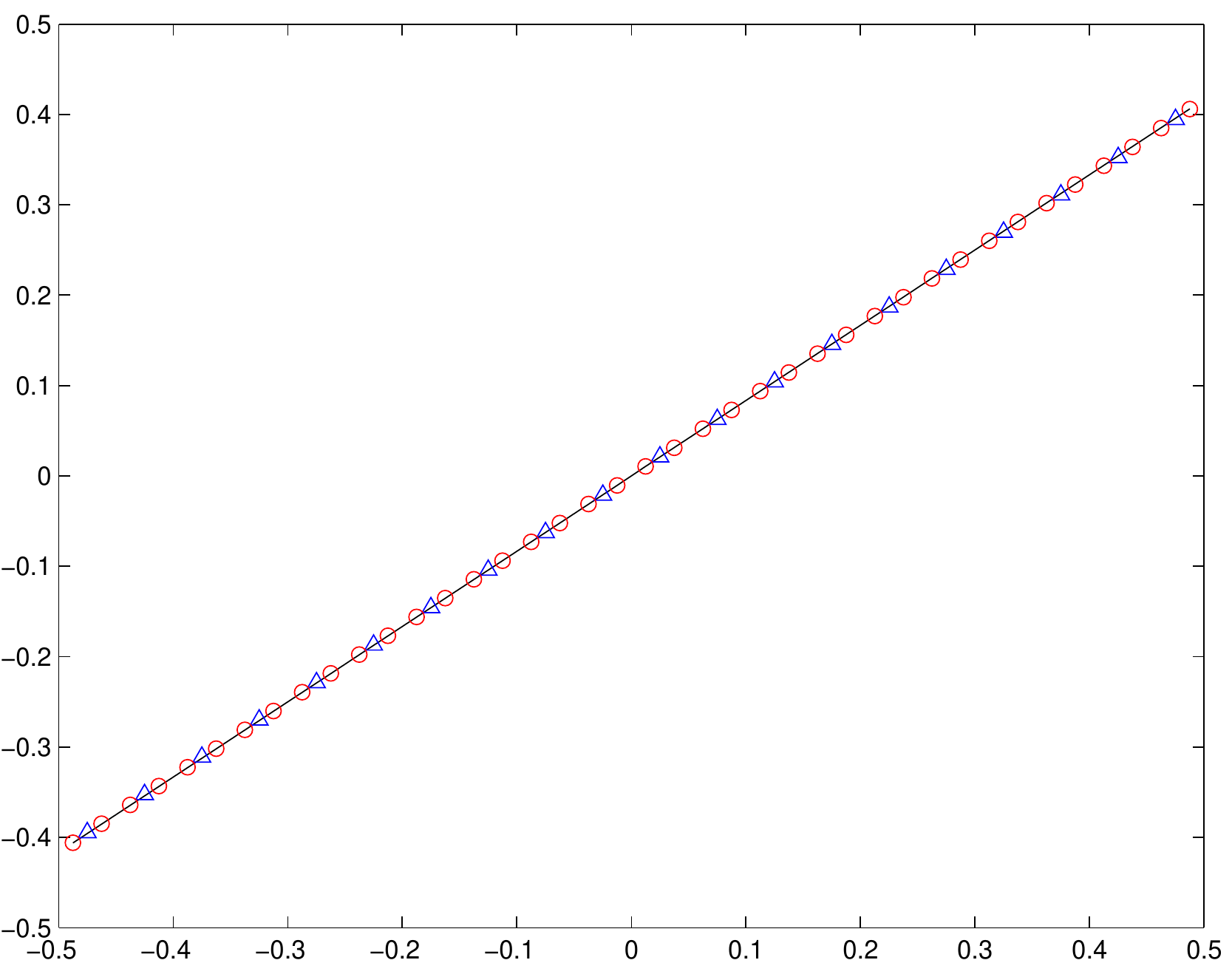}
 }
 \subfigure[$p$]{
 \includegraphics[width=0.3\textwidth]{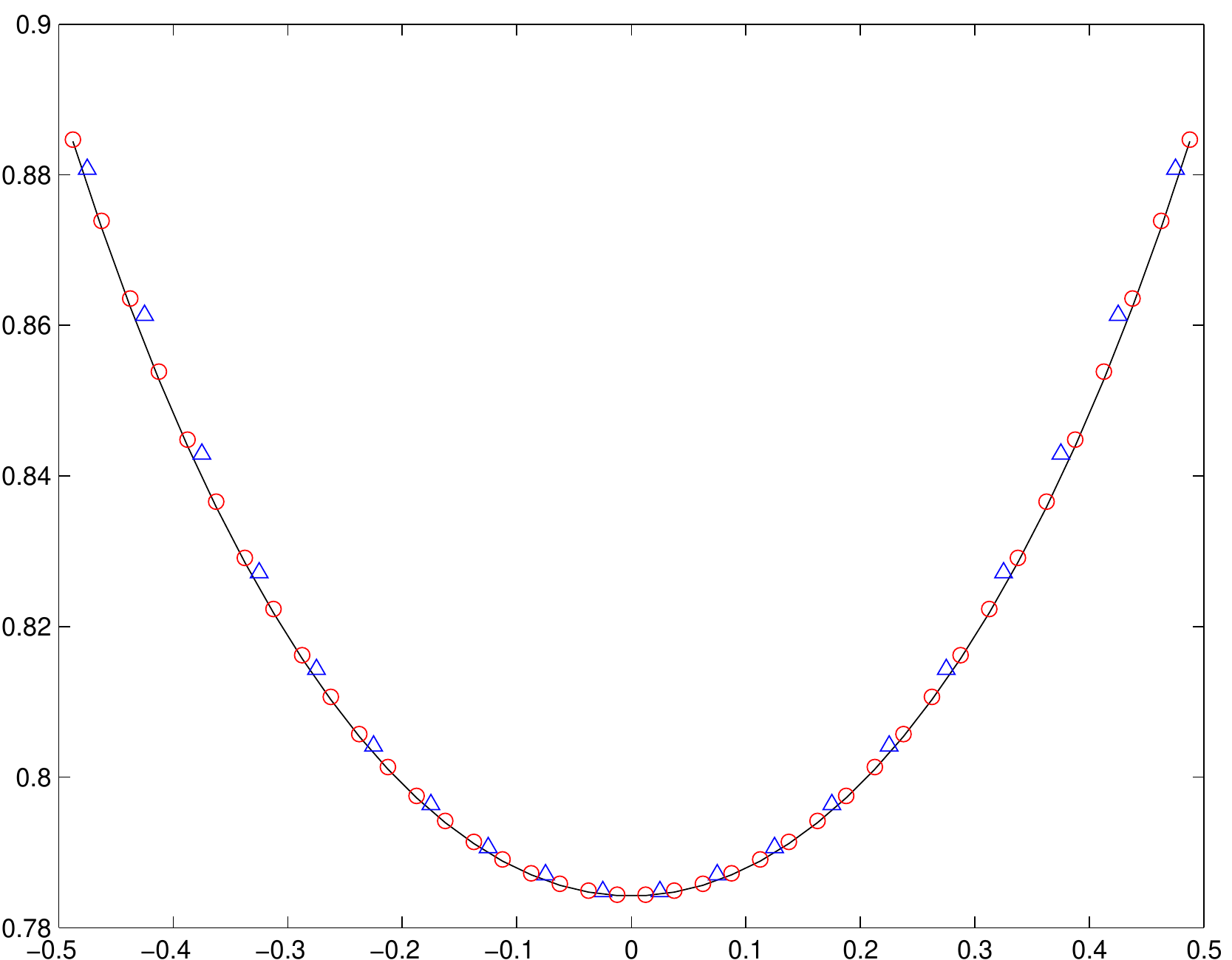}
 }

 \caption{Example \ref{ex:boost}: The number density, velocity and pressure at $t=1.2$
 are obtained by our BGK scheme   with 20 cells (``${\triangle}$") and 40 cells (``{$\circ$}"), respectively. The solid line represents the exact solution.}
 \label{fig:boost}
\end{figure}

\begin{table}[H]
 \setlength{\abovecaptionskip}{0.cm}
 \setlength{\belowcaptionskip}{-0.cm}
 \caption{Example \ref{ex:boost}: Numerical errors of $ n $ in $l^1$ and $l^2$-norm and convergence rates at $t = 1.2$.}\label{tab:boost}
 \begin{center}
   \begin{tabular}{c|cc|cc}
     \hline
     $N$    &  $l^1$ error  & $l^1$ order  &  $l^2$ error  &  $l^2$ order \\
     \hline
     10   & 8.9214e-03    & --           & 1.2028e-02    & --  \\
     20   & 1.9291e-03    & 2.2094       & 2.4837e-03    & 2.2759  \\
     40   & 4.9766e-04    & 1.9546       & 6.2325e-04    & 1.9946  \\
     80   & 1.3682e-04    & 1.8629       & 1.6760e-04    & 1.8948  \\
     \hline
   \end{tabular}
 \end{center}
\end{table}


\begin{example}[Heat conduction]\rm\label{ex:heat}
This test considers the problem of heat conduction
between two parallel plates, which  are assumed to be infinite
 and separated by a distance $H$.
Moreover, both  plates are always stationary.
 The temperatures of the lower and upper plates are given
 by $T_0$ and $T_1$, respectively. The viscosity $\mu$ is a constant.

 Based on the above assumptions,  the Navier-Stokes equations   can be simplified as
 \begin{equation*}
   \frac{\partial }{\partial y}\left(\frac{1}{T^2}\frac{\partial T}{\partial y}\right) = 0,\quad T(0)=T_0,\quad T(H)=T_1,
 \end{equation*}
whose  analytic solution is gotten as follows
 \begin{equation}\label{eq:heatexact}
   T(y)=\frac{HT_0T_1}{HT_1-(T_1-T_0)y}.
 \end{equation}
Our computation  takes $H=1, p=0.8, u_1=0.2, u_2=0, \mu=5\times10^{-3}, T_0=0.1, T_1=1.0002T_0$, and $0.5(T_0+T_1)$ as the initial value  for the temperature $T$ in the entire domain.
Moreover, the initial time partial derivatives  are given by $ n _t(x,0)=0, v_{1t}(x,0)=0, v_{2t}(x,0)=0$ and $p_t(x,0)=0$. Because $u_1\neq 0$, the the 2D BGK scheme should be used for numerical simulation.

The left figure in Fig. \ref{fig:heat} plots the numerical temperature (``{$\circ$}") obtained by the 2D BGK scheme in comparison with  the steady-state  analytic solution (solid line)
given by \eqref{eq:heatexact}. It is seen that the numerical solution is well comparable with the analytic.
The right figure in Fig. \ref{fig:heat} shows
convergence of the temperature  to the steady state
measured in the $l^1$-error between the numerical  and analytic solutions.
\end{example}

\begin{figure}[htbp]
 \centering
 \subfigure[Temperature $T$]{
   \includegraphics[width=0.4\textwidth]{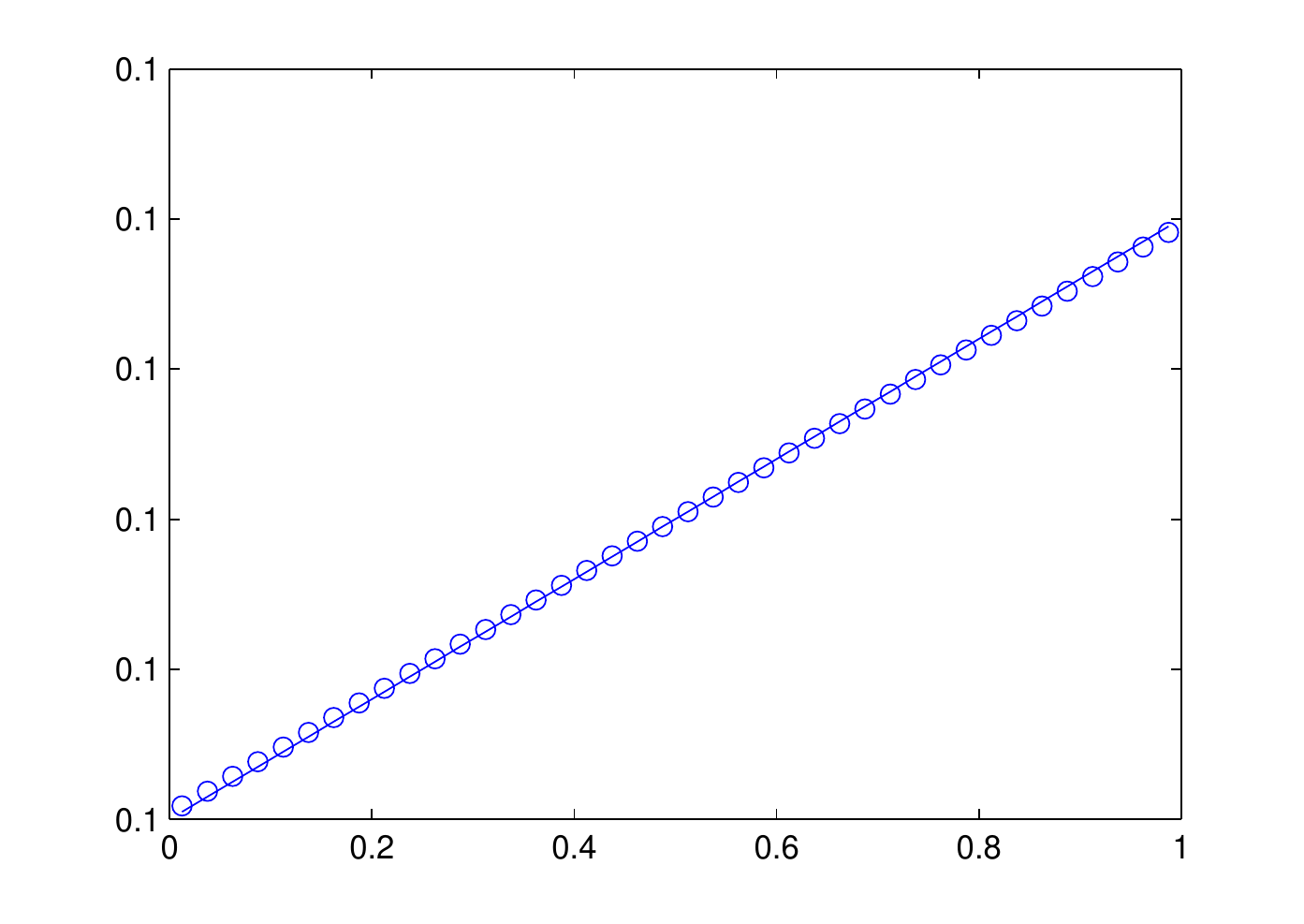}
 }
 \subfigure[$l^1$-error in temperature]{
   \includegraphics[width=0.4\textwidth]{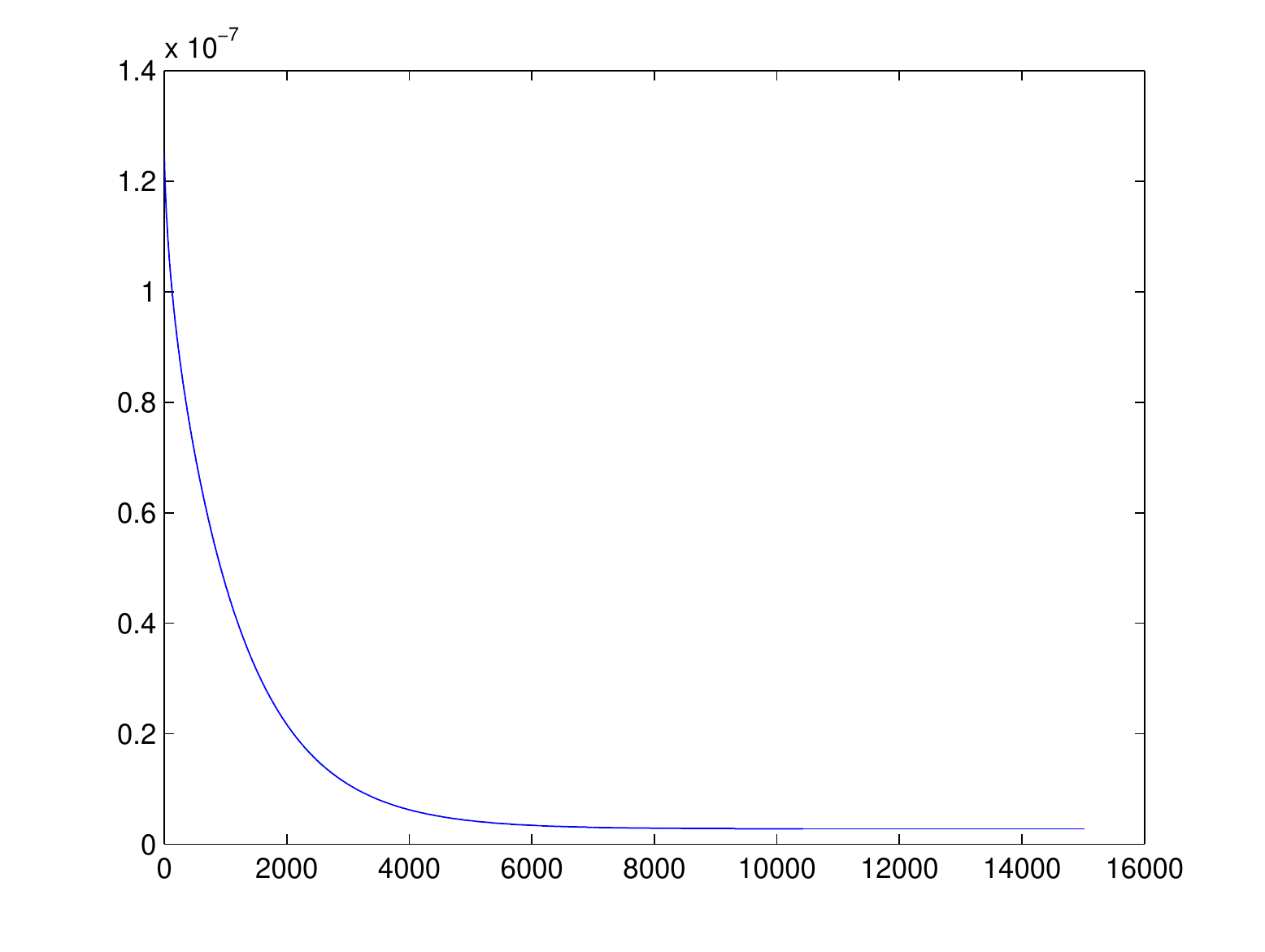}
 }
 \caption{Example \ref{ex:heat}: Left: The numerical temperature (``$\circ$") is obtained by the 2D BGK scheme
 with 40 cells  in comparison with the exact solutions (solid line); right:
 Convergence of the temperature to the steady state measured in the $l^1$-error.}
 \label{fig:heat}
\end{figure}

\section{Conclusions}
\label{sec:conclusion}
The paper developed  second-order accurate genuine  BGK 
schemes  in the framework of finite volume method for
the 1D and 2D  ultra-relativistic flows.
Different from the existing KFVS or BGK-type schemes for
the ultra-relativistic Euler equations
the  present genuine  BGK schemes were  derived from the analytical solution of the
Anderson-Witting model, which was   given for the first time and included
the ``genuine'' particle collisions in the gas transport process.
The genuine BGK schemes  were also developed for the  ultra-relativistic viscous flows
and two ultra-relativistic viscous examples were designed.
Several 1D and 2D numerical experiments were conducted
to demonstrate that the proposed BGK schemes were accurate and stable in simulating
 ultra-relativistic inviscid and viscous  flows, and had  higher resolution at the contact discontinuity  than the KFVS or BGK-type schemes.
The present BGK schemes could be easily  extended to the 3D Cartesian grid for the ultra-relativistic  flows and it was
interesting to develop the genuine BGK schemes for the special and general relativistic flows.


\section*{Acknowledgements}
This work was partially supported by
the Science Challenge Project, No. JCKY2016212A502,
the Special Project on High-performance Computing
       under the National Key R\&D Program (No. 2016YFB0200603),
       and
the National Natural Science
Foundation of China (Nos. 91330205, 91630310,   11421101).

\bibliography{ref/journalname,ref/pkuth}
           \bibliographystyle{plain}
\end{document}